\documentclass[12pt,oneside,english,american]{amsart}
\usepackage[T1]{fontenc}
\usepackage[latin9]{inputenc}
\usepackage{amsthm}
\usepackage{amssymb}
\usepackage{esint}

\makeatletter
\numberwithin{equation}{section}
\numberwithin{figure}{section}
\theoremstyle{plain}
\newtheorem{thm}{\protect\theoremname}[section]
  \theoremstyle{definition}
  \newtheorem{defn}[thm]{\protect\definitionname}
  \theoremstyle{definition}
  \newtheorem{example}[thm]{\protect\examplename}
  \theoremstyle{remark}
  \newtheorem{rem}[thm]{\protect\remarkname}
  \theoremstyle{plain}
  \newtheorem{lem}[thm]{\protect\lemmaname}
  \theoremstyle{plain}
  \newtheorem{prop}[thm]{\protect\propositionname}
  \theoremstyle{plain}
  \newtheorem{cor}[thm]{\protect\corollaryname}

\usepackage{fancyhdr}
\usepackage[us,12hr]{datetime} 

\newcommand{\norm}[1]{\Vert #1 \Vert}

\newcommand{\cF}{\mathcal{F}}

\usepackage{babel}

\makeatother

\usepackage{babel}
  \addto\captionsamerican{\renewcommand{\corollaryname}{Corollary}}
  \addto\captionsamerican{\renewcommand{\definitionname}{Definition}}
  \addto\captionsamerican{\renewcommand{\examplename}{Example}}
  \addto\captionsamerican{\renewcommand{\lemmaname}{Lemma}}
  \addto\captionsamerican{\renewcommand{\propositionname}{Proposition}}
  \addto\captionsamerican{\renewcommand{\remarkname}{Remark}}
  \addto\captionsamerican{\renewcommand{\theoremname}{Theorem}}
  \addto\captionsenglish{\renewcommand{\corollaryname}{Corollary}}
  \addto\captionsenglish{\renewcommand{\definitionname}{Definition}}
  \addto\captionsenglish{\renewcommand{\examplename}{Example}}
  \addto\captionsenglish{\renewcommand{\lemmaname}{Lemma}}
  \addto\captionsenglish{\renewcommand{\propositionname}{Proposition}}
  \addto\captionsenglish{\renewcommand{\remarkname}{Remark}}
  \addto\captionsenglish{\renewcommand{\theoremname}{Theorem}}
  \providecommand{\corollaryname}{Corollary}
  \providecommand{\definitionname}{Definition}
  \providecommand{\examplename}{Example}
  \providecommand{\lemmaname}{Lemma}
  \providecommand{\propositionname}{Proposition}
  \providecommand{\remarkname}{Remark}
\providecommand{\theoremname}{Theorem}

\begin{document}
\title[Tensorial Function Theory]{Tensorial Function Theory:\\{\tiny{From Berezin Transforms to Taylor's Taylor Series and Back}}}

\author{Paul S. Muhly}

\address{Department of Mathematics \\
University of Iowa\\
 Iowa City, IA 52242}

\email{paul-muhly@uiowa.edu }

\author{Baruch Solel}

\address{Department of Mathematics\\
 Technion\\
32000 Haifa, Israel}

\email{mabaruch@techunix.technion.ac.il}

\thanks{The research of both authors was supported in part by a US-Israel Binational Science Foundation grant. The second author was also supported by the Technion V.P.R. Fund.}

\begin{abstract}
Let $H^{\infty}(E)$ be the Hardy algebra of a $W^{*}$-correspond\-ence
$E$ over a $W^{*}$-algebra $M$. Then the ultraweakly continuous
completely contractive representations of $H^{\infty}(E)$ are parametrized
by certain sets $\mathcal{AC}(\sigma)$ indexed by $NRep(M)$ - the
normal $*$-repres\-ent\-ations $\sigma$ of $M$. Each set $\mathcal{AC}(\sigma)$
has analytic structure, and each element $F\in H^{\infty}(E)$ gives
rise to an analytic operator-valued function $\widehat{F}_{\sigma}$
on $\mathcal{AC}(\sigma)$ that we call the $\sigma$-Berezin transform
of $F$. The sets $\{\mathcal{AC}(\sigma)\}_{\sigma\in\Sigma}$ and
the family of functions $\{\widehat{F}_{\sigma}\}_{\sigma\in\Sigma}$
exhibit ``matricial structure'' that was introduced by Joeseph Taylor
in his work on noncommutative spectral theory in the early 1970s.
Such structure has been exploited more recently in other areas of
free analysis and in the theory of linear matrix inequalities. Our
objective here is to determine the extent to which the matricial structure
characterizes the Berezin transforms.
\end{abstract}

\maketitle

\section{Introduction\label{sec:Introduction}}

Our purpose in this paper is to explore relations among three subjects:
(1) the theory of Berezin transforms that arise from the representation
theory of tensor algebras and Hardy algebras of $W^{*}$-correspondences,
(2) infinite dimensional holomorphy, and (3) the theory of  free holomorphic
functions initiated by Joseph Taylor in \cite{Tay72c,Tay73a}. In
this introduction we indicate the connections we have in mind and
provide a bit of context. Details and a fuller account, including
relevant definitions of terms left undefined here, will be given in
subsequent sections. 

Suppose first that $M$ is a $W^{*}$-algebra and that $E$ is a $W^{*}$-correspond\-ence
over $M$.%
\footnote{We shall assume throughout that $M$ has a separable predual and that
$E$ is countably generated.%
} With these ingredients one may build two operator algebras, $\mathcal{T}_{+}(E)$
and $H^{\infty}(E)$, that are generated by a copy of $M$ and the
creation operators $\{T_{\xi}\mid\xi\in E\}$ acting on the full Fock
space $\mathcal{F}(E)=M\oplus E\oplus E^{\otimes2}\oplus E^{\otimes3}\oplus\cdots$:
$\mathcal{T}_{+}(E)$, \emph{the tensor algebra of} $E$, is the norm
closed algebra generated by these objects and $H^{\infty}(E)$, \emph{the
Hardy algebra} of $E$, is its ultraweak closure. One may think of
$\mathcal{T}_{+}(E)$ as a generalization of the disc algebra $A(\mathbb{D})$,
and $H^{\infty}(E)$ may be viewed as a generalization of the classical
Hardy space, $H^{\infty}(\mathbb{D})$, consisting of the bounded
analytic functions on $\mathbb{D}$. Indeed, when $M=\mathbb{C}=E$,
then $\mathcal{T}_{+}(E)$ is naturally completely isometrically isomorphic
to $A(\mathbb{D})$ and $H^{\infty}(E)$ is naturally completely isometrically
isomorphic and weak-$*$ homeomorphic to $H^{\infty}(\mathbb{D})$.
Another important example to keep in mind is that which arises when
$M=\mathbb{C}$ and $E=\mathbb{C}^{d}$, $d\geq2$. In this case,
$\mathcal{T}_{+}(E)$ is naturally completely isometrically isomorphic
to Gelu Popescu's noncommutative disc algebra $\mathcal{A}_{d}$ \cite{P1996b}
and $H^{\infty}(E)$ has come to be called the \emph{algebra of noncommutative
analytic Toeplitz operators. }The terminology is due to Davidson and
Pitts \cite{DPit98a}, and much of the initial theory of these algebras
is due to them. Other interesting finite dimensional settings can
be constructed from graphs or quivers. (See \cite{Muhly1997,Muhly1999,Katsoulis2006}
for examples.)

It is worthwhile to emphasize that even if one were interested only
in these finite dimensional examples, it is useful to work with the
general theory. One reason is that the general theory is invariant
under Morita equivalence \cite{Muhly2000a,Muhly2011b} and, among
many things, Morita theory allows one to study every $\mathcal{T}_{+}(E)$
and $H^{\infty}(E)$ in terms of analytic crossed products - generalizations
of twisted polynomial rings - that have played such a prominent role
in the theory of non-self-adjoint operator algebras. The point is
that although $\mathcal{T}_{+}(E)$ and $H^{\infty}(E)$ look very
much like algebras of functions of several (noncommutative) variables,
as we shall see, they behave much more like algebras of functions
of one variable than one might expect.

In \cite[Theorem 2.9]{Muhly1998a}, we showed that every completely
contractive representation of $\mathcal{T}_{+}(E)$ is given by a
pair, $(\sigma,\mathfrak{z})$, where $\sigma$ is a normal representation
of $M$ on a Hilbert space $H_{\sigma}$ and $\mathfrak{z}:E\otimes_{\sigma}H_{\sigma}\to H_{\sigma}$
is an operator of norm at most $1$ that intertwines $\sigma^{E}\circ\varphi$
and $\sigma$, where $\sigma^{E}$ is the representation of $\mathcal{L}(E)$
that is induced by $\sigma$ in the sense of Rieffel \cite{R1974b}
and where $\varphi$ denotes the left action of $M$ on $E$. We denote
the representation associated to $(\sigma,\mathfrak{z})$ by $\sigma\times\mathfrak{z}$. 

(In general, if $\mathcal{A}$ is a not-necessarily-self-adjoint algebra
and if $\pi$ and $\rho$ are two representations of $\mathcal{A}$
by bounded operators on Hilbert spaces $H_{\pi}$ and $H_{\rho}$,
respectively, then we shall write $\mathcal{I}(\pi,\sigma)$ for the
collection of all operators $C$ from $H_{\pi}$ to $H_{\rho}$ such
that $C\pi(a)=\rho(a)C$ for all $a\in\mathcal{A}$ and we call $\mathcal{I}(\pi,\rho)$
the \emph{intertwiner space} or the \emph{space of intertwiners from
$\pi$ to $\rho$}.) 

For each normal representation $\sigma$ of $M$, we endow $\mathcal{I}(\sigma^{E}\circ\varphi,\sigma)$
with the operator norm and we write $\mathbb{D}(0,1,\sigma)$ for
the open unit ball in $\mathcal{I}(\sigma^{E}\circ\varphi,\sigma)$.
Then, when $\sigma$ is fixed, each $F\in\mathcal{T}_{+}(E)$ gives
rise to a $B(H_{\sigma})$-valued function $\widehat{F}_{\sigma}$
defined on the closed unit ball $\overline{\mathbb{D}(0,1,\sigma)}$
by the formula 
\begin{equation}
\widehat{F}_{\sigma}(\mathfrak{z}):=\sigma\times\mathfrak{z}(F).\label{eq:Berezin transform}
\end{equation}
Because of formulas that we derived in \cite[Theorem 13]{Muhly2009},
we call $\widehat{F}_{\sigma}$ the \emph{$\sigma$-Berezin transform}
of $F$.%
\footnote{We are indebted to Lew Coburn for calling our attention to the connection
between our formulas and the classical Berezin transform associated
with the Hardy space on the open unit disc in the complex plane. We
note, too, that our terminology agrees with that of Gelu Popescu \cite{P2008d}
in those settings where his theory and ours overlap. %
} The collection of all $\sigma$-Berezin transforms, $\{\widehat{F}_{\sigma}\}_{\sigma\in NRep(M)}$,
obtained by letting $\sigma$ range over the collection of all normal
representations of $M$ on separable Hilbert space, $NRep(M)$, is
called \emph{the} \emph{Berezin transform} of $F$ and will be written
simply as $\widehat{F}$. It is easy to see that for $F\in\mathcal{T}_{+}(E)$,
each $\sigma$-Berezin transform $\widehat{F}_{\sigma}$ is a continuous
function on $\overline{\mathbb{D}(0,1,\sigma)}$ with values in $B(H_{\sigma})$,
where both spaces are given the norm operator topology. Further, $\widehat{F}_{\sigma}$
is holomorphic in the sense of Frechet \cite[112 and 778]{HP1974}
when restricted to $\mathbb{D}(0,1,\sigma)$. 

A similar sort of representation exists for elements $F\in H^{\infty}(E)$.
However, for these $F$, $\widehat{F}_{\sigma}$ makes good sense
only on the set of $\mathfrak{z}$s in $\overline{\mathbb{D}(0,1,\sigma)}$
such that $\sigma\times\mathfrak{z}$ extends from $\mathcal{T}_{+}(E)$
to an ultraweakly continuous representation of $H^{\infty}(E)$ in
$B(H_{\sigma})$. We denote the set of such points by $\mathcal{AC}(\sigma)$
and for reasons spelled out in \cite{Muhly2011a} we call them the
\emph{absolutely continuous points} in $\overline{\mathbb{D}(0,1,\sigma)}$.
It turns out that $\mathbb{D}(0,1,\sigma)\subseteq\mathcal{AC}(\sigma)$
\cite[Corollary 2.14]{Muhly2004a}. Thus, for $F\in H^{\infty}(E)$,
$\widehat{F}_{\sigma}$ makes sense as a function on $\mathbb{D}(0,1,\sigma)$.
It is, in fact, bounded and holomorphic with respect to the norm topologies
on $\mathbb{D}(0,1,\sigma)$ and $B(H_{\sigma})$.

The Frechet power series of $\widehat{F}_{\sigma}$ is easy to calculate
and has a remarkably simple expression in terms of the tensorial ``Fourier
series'' in which $F$ may be expressed using the gauge automorphism
group built from the number operator on the full Fock space $\mathcal{F}(E)$.
We call power series with this special form \emph{tensorial power
series} (Definition \ref{def:Tensorial_power_series}). It is natural
to inquire about the structure of such power series, in general, and
one soon sees that much of standard elementary theory of complex analysis
on the open unit disc can be recapitulated in the more general setting
we are describing. When $M=\mathbb{C}$ and $E=\mathbb{C}^{d}$, this
has been done already by Popescu in \cite{P2006a}. 

For a given $\sigma$, we write $B(\sigma)$ for $\{\widehat{F}_{\sigma}\mid F\in H^{\infty}(E)\}$.
Then $B(\sigma)$ is an algebra under pointwise multiplication that
may be identified naturally with a quotient of $H^{\infty}(E)$ by
an ultraweakly closed ideal and the map $F\to\widehat{F}_{\sigma}$
is a complete quotient map. If $\sigma$ is a faithful normal representation
of $M$ of infinite multiplicity, then the map is a completely isometric
isomorphism from $H^{\infty}(E)$ onto $B(\sigma)$ \cite[Lemma 3.8]{Muhly2008b}.
In general, however, $B(\sigma)$ is a proper quotient of $H^{\infty}(E)$.
Indeed, when $M=\mathbb{C}$, and $E=\mathbb{C}^{d}$, with $d\geq2$,
then each normal representation of $M$ is, of course, (unitarily
equivalent to) a multiple of the representation $\sigma_{1}$ of $\mathbb{C}$
on the one-dimensional Hilbert space, $\mathbb{C}$. The $\sigma_{1}$-Berezin
transforms $\{\widehat{F}_{\sigma_{1}}\mid F\in H^{\infty}(E)\}$
are the multipliers of the Drury-Arveson space and so form a commutative
algebra. On the other hand, $B(\infty\sigma_{1})$ is isomorphic to
the algebra of noncommutative analytic Toeplitz operators\emph{ }in
the fashion just described. For $2\leq n<\infty$, $B(n\sigma_{1})$
is a completion of the algebra of $d$ generic $n\times n$ matrices.
So, as soon as $n\geq2$, it is noncommutative. But also, by virtue
of the polynomial identities these algebras satisfy, it is easy to
see that when $n\neq m$, $B(n\sigma_{1})\ncong B(m\sigma_{1})$,
and that no finite dimensional representation of $\mathbb{C}$ yields
a faithful representation of either $H^{\infty}(E)$ or of $\mathcal{T}_{+}(E)$
in terms of Berezin transforms. Thus there arise very natural questions:
How much is lost when forming the $\sigma$-Berezin transform of an
element in $H^{\infty}(E)$? How might one reconstruct an $F$ from
its finite dimensional $\sigma$-Berezin transforms? What extra information
is required? While we are still far from giving definitive answers
to these questions, we believe that what we accomplish here is a helpful
start.

There is an important feature of the discs $\mathbb{D}(0,1,\sigma)$
that plays a central roll in our theory: For any two normal representations
of $M$, $\sigma$ and $\tau$, 
\begin{equation}
\mathbb{D}(0,1,\sigma)\oplus\mathbb{D}(0,1,\tau)\subseteq\mathbb{D}(0,1,\sigma\oplus\tau).\label{eq:Sum}
\end{equation}
The meaning of this inclusion is easy to understand when one realizes
that $\mathcal{I}((\sigma\oplus\tau)^{E}\circ\varphi,\sigma\oplus\tau)$
may be viewed as a set of operator matrices $\begin{bmatrix}\mathfrak{z_{11}} & \mathfrak{z_{12}}\\
\mathfrak{z_{21}} & \mathfrak{z_{22}}
\end{bmatrix}$ acting as operators from $H_{\sigma^{E}\circ\varphi}\oplus H_{\tau^{E}\circ\varphi}$
to $H_{\sigma}\oplus H_{\tau}$, where $\mathfrak{z_{11}}\in\mathcal{I}(\sigma^{E}\circ\varphi,\sigma)$,
$\mathfrak{z_{12}}\in\mathcal{I}(\tau^{E}\circ\varphi,\sigma)$, $\mathfrak{z_{21}}\in\mathcal{I}(\sigma^{E}\circ\varphi,\tau)$,
and $\mathfrak{z_{22}}\in\mathcal{I}(\tau^{E}\circ\varphi,\tau)$.
So $\mathbb{D}(0,1,\sigma)\oplus\mathbb{D}(0,1,\tau)$ is just the
collection of those matrices $\begin{bmatrix}\mathfrak{z_{11}} & \mathfrak{z_{12}}\\
\mathfrak{z_{21}} & \mathfrak{z_{22}}
\end{bmatrix}\in\mathbb{D}(0,1,\sigma\oplus\tau)$ in which the off-diagonal entries vanish. 
\begin{defn}
\label{def:Matricial_family_of_sets} A family of sets $\{\mathcal{U}(\sigma)\}_{\sigma\in NRep(M)}$,
with $\mathcal{U}(\sigma)\subseteq\mathcal{I}(\sigma^{E}\circ\varphi,\sigma)$,
satisfying $\mathcal{U}(\sigma)\oplus\mathcal{U}(\tau)\subseteq\mathcal{U}(\sigma\oplus\tau)$
is called a \emph{matricial family} of sets.
\end{defn}
Matricial families, in particular, families of discs $\{\mathbb{D}(0,1,\sigma)\}_{\sigma\in NRep(M)}$,
enjoy properties that are very similar to the properties of the domains
that Taylor first considered in \cite[Section 6]{Tay72c}, when he
introduced a notion of localization for free algebras. They are very
closely related to the fully matricial sets of Dan Voiculescu \cite{Voi2004,Voi2010},
which are matricial sets in our terminology but also satisfy additional
conditions which we do not use here. In contexts when $E$ is finite
dimensional, matricial sets are essentially the \emph{noncommutative
sets} that Bill Helton, Igor Klep, Scott McCullough and others study
in the setting of linear matrix inequalities \cite{HKM2011a,HKM2011b,HKMS2009}.
They are also closely connected to the \emph{noncommutative sets}
in the work of Dmitry Kaliuzhny\u{i}-Verbovetsky\u{i} and Victor
Vinnikov devoted to noncommutative function theory \cite{K-VV2009,Kaliuzhnyi-Verbovetskyi2010,K-VVPrep}.

There is another property that the discs also enjoy, viz. for any
contraction $\mathfrak{t}\in\mathcal{I}(\tau,\sigma)$, the inclusion,
\begin{equation}
\mathfrak{t}\mathbb{D}(0,1,\tau)(I_{E}\otimes\mathfrak{t}^{*})\subseteq\mathbb{D}(0,1,\sigma),\label{eq:matricially convex}
\end{equation}
holds. This shows that the discs are \emph{matricially convex} in
the sense of operator space theory (see \cite{Effros2000}). While
matricial convexity does not play a role in our immediate considerations,
it already has proved useful elsewhere.

The Berezin transform, $\widehat{F}=\{\widehat{F}_{\sigma}\}_{\sigma\in NRep(M)}$,
of an element $F\in H^{\infty}(E)$ satisfies the equation
\begin{equation}
\widehat{F}_{\sigma\oplus\tau}(\mathfrak{z}\oplus\mathfrak{w})=\widehat{F}_{\sigma}(\mathfrak{z})\oplus\widehat{F}_{\tau}(\mathfrak{w}),\qquad\mathfrak{z}\oplus\mathfrak{w}\in\mathbb{D}(0,1,\sigma)\oplus\mathbb{D}(0,1,\tau).\label{eq:Direct_sum_decomposition}
\end{equation}
This, too, is a critical feature of the functions in Taylor's theory
and in the other places just cited. In fact, the Berezin transforms
have an additional property that we will use repeatedly: 
\begin{defn}
\label{def:Matricial_family_of_functions} Suppose $\{\mathcal{U}(\sigma)\}_{\sigma\in NRep(M)}$
is a matricial family of sets and suppose that for each $\sigma\in NRep(M)$,
$f_{\sigma}:\mathcal{U}(\sigma)\to B(H_{\sigma})$ is a function.
We say that $f:= \{f_{\sigma}\}_{\sigma\in NRep(M)}$ is a \emph{matricial
family of functions }in case 
\begin{equation}
Cf_{\sigma}(\mathfrak{z})=f_{\tau}(\mathfrak{w})C\label{eq:respects_intertwiners}
\end{equation}
for every $\mathfrak{z}\in\mathcal{U}(\sigma)$, every $\mathfrak{w}\in\mathcal{U}(\tau)$
and every $C\in\mathcal{I}(\sigma,\tau)$ such that 
\begin{equation}
C\mathfrak{z}=\mathfrak{w}(I_{E}\otimes C).\label{eq:basic_intertwining}
\end{equation}

When the family $\{\mathcal{U}(\sigma)\}_{\sigma\in NRep(M)}$ is
$\{\mathbb{D}(0,1,\sigma)\}_{\sigma\in NRep(M)},$ and $f = \{f_{\sigma}\}_{\sigma\in NRep(M)}$
is a Berezin transform, then it is easy to see that the assumptions
on an operator $C:H_{\sigma}\to H_{\tau}$ that $C\in\mathcal{I}(\sigma,\tau)$
and satisfies equation \eqref{eq:basic_intertwining} express the
fact that $C$ lies in $\mathcal{I}(\sigma\times\mathfrak{z},\tau\times\mathfrak{w})$.
But then, equation \eqref{eq:respects_intertwiners} is immediate.
It is simply a manifestation of the structure of the commutant of
the representation $(\sigma\oplus\tau)\times(\mathfrak{z}\oplus\mathfrak{w})$.
In this setting also, the defining hypothesis for a matricial family
can be written simply as
\begin{equation}
\mathcal{I}(\sigma\times\mathfrak{z},\tau\times\mathfrak{w})\subseteq\mathcal{I}(f_{\sigma}(\mathfrak{z}),f_{\tau}(\mathfrak{w})),\label{eq:Resp_intertwiners}
\end{equation}
for all $\sigma,\tau\in NRep(M)$, $\mathfrak{z}\in\mathcal{AC}(\sigma),$
and $\mathfrak{w}\in\mathcal{AC}(\tau)$. Consequently, we sometimes
say  that a matricial family \emph{respects intertwiners}. Observe
that if a family respects intertwiners, then it automatically satisfies
equations like \eqref{eq:Direct_sum_decomposition}.

What is surprising is the following converse - a nonlinear double
commutant theorem of sorts - which extends the double commutant theorem
for induced representations of Hardy algebras \cite[Corollary 3.10]{Muhly2004a}.
We shall prove a more refined statement in Theorem \ref{thm:_Relative_double_commutant}. \end{defn}
\begin{thm}
\label{thm:Double_Commutant} If $f=\{f_{\sigma}\}_{\sigma\in NRep(M)}$
is a matricial family of functions, with $f_{\sigma}$ defined on
$\mathcal{AC}(\sigma)$ and mapping to $B(H_{\sigma}),$ then there
is an $F\in H^{\infty}(E)$ such that $f$ is the Berezin transform
of $F$, i.e., $f_{\sigma}=\widehat{F}_{\sigma}$ for every $\sigma$.
\end{thm}
Three features of Theorem \ref{thm:Double_Commutant} deserve comment:
(1) the domain of each $f_{\sigma}$ is $\mathcal{AC}(\sigma)$, (2)
\emph{all} the representations in $NRep(M)$ are used, and (3) no
hypotheses on the \emph{nature} of the functions $f_{\sigma}$ are
made. They are not assumed to be analytic or continuous; they are
not even assumed to be bounded. If we relax (1) and assume only that
the functions are defined on $\mathbb{D}(0,1,\sigma)$, then the theorem
breaks down: In Example \ref{eg:_Failure_of_resolvents} we will exhibit
an \emph{unbounded} family $f=\{f_{\sigma}\}_{\sigma\in NRep(M)}$,
where $f_{\sigma}$ is defined only on $\mathbb{D}(0,1,\sigma)$,
that satisfies \eqref{eq:Resp_intertwiners}. However, we shall also
prove that if the $f_{\sigma}:\mathbb{D}(0,1,\sigma)\to B(H_{\sigma})$
are \emph{bounded uniformly in} $\sigma\in NRep(M)$ and if they respect
intertwiners, then $f$ is a Berezin transform (Theorem \ref{boundedhol}). 

With regard to (2), it is helpful to reflect on the fact that $NRep(M)$
is really a $W^{*}$-category in the sense of \cite{GLR1985}. In
fact, it is essentially Riefel's category $Normod-M$ \cite{R1974a}.
We do not need much of the theory of such categories to achieve our
goals here, but that theory has guided our thinking. The \emph{objects}
of $NRep(M)$ are the normal representations of $M$ on separable
Hilbert space. (We have abused notation a bit and have simply written
$\sigma\in NRep(M)$, when $\sigma$ is a normal representation of
$M$.) The set of morphisms from $\sigma$ to $\tau$ is the intertwiner
space $\mathcal{I}(\sigma,\tau)$. We want to consider subcategories
$\Sigma$ of $NRep(M)$ with the property that if $\sigma$ and $\tau$
are in $\Sigma$, then so is $\sigma\oplus\tau$. Such a category
is an \emph{additive} subcategory of $NRep(M)$. In the literature,
particularly that dealing with the setting where $M=\mathbb{C}$ and
$E=\mathbb{C}^{d}$, it is important to extract as much information
as is possible from the finite dimensional representations of $M$,
and, of course, the finite dimensional representations of $M$ determine
an additive subcategory of $NRep(M)$. Although, we do not know how
to modify the hypotheses in Theorem \ref{thm:Double_Commutant} so
that we can restrict to every additive subcategory, we can prove that
if $\Sigma$ is a full additive subcategory of $NRep(M)$ that consists
of faithful representations of $M$, if $f=\{f_{\sigma}\}_{\sigma\in\Sigma}$
is a uniformly bounded matricial family of functions, with $f_{\sigma}:\mathbb{D}(0,1,\sigma)\to B(H_{\sigma})$,
then each of the $f_{\sigma}$s is Frechet holomorphic in $\mathbb{D}(0,1,\sigma)$
and the Frechet-Taylor expansion of $f_{\sigma}$ can be expressed
explicitly in terms of tensors from the tensor powers of $E$ in the
same way that the Taylor series for Berezin transforms can be expressed,
i.e., in terms of our tensorial power series. But to do this, we must
first use ideas from Kaliuzhny\u{i}-Verbovetsky\u{i} and Vinnikov
\cite{K-VV2009,Kaliuzhnyi-Verbovetskyi2010,K-VVPrep} to show that
the $f_{\sigma}$s have  Taylor series in the sense of Joseph Taylor
\cite{Tay73a} and then use a variant of our duality theorem \cite[Theorem 3.6]{Muhly2004a}
to show that Taylor's Taylor coefficients are tensors of the desired
type. Thus, we will show that there are strong interconnections among
three notions of holomorphy and their associated power series: Frechet
holomorphy for maps between Banach spaces, Taylor's notion of free
holomorphy, and tensorial holomorphy. We were led to these by studying
Berezin transforms of elements of $H^{\infty}(E)$. Although Theorem
\ref{boundedhol} shows that a matricial family of functions $f=\{f_{\sigma}\}_{\sigma\in NRep(M)}$,
with $f_{\sigma}:\mathbb{D}(0,1,\sigma)\to B(H_{\sigma})$, is a Berezin
transform if and only if $f$ is uniformly bounded in $\sigma$, the
proof very much depends upon properties of $NRep(M)$ that are not
shared by all subcategories. The best we can say at this stage is
that if $\Sigma$ is a full additive subcategory of $NRep(M)$ consisting
of faithful representations, if $f=\{f_{\sigma}\}_{\sigma\in\Sigma}$
is a family of functions with $f_{\sigma}:\mathbb{D}(0,1,\sigma)\to B(H_{\sigma})$,
and if $\sup_{\sigma\in\Sigma}\sup_{\mathfrak{z}\in\mathbb{D}(0,1,\sigma)}\Vert f_{\sigma}(\mathfrak{z})\Vert<\infty$,
then $f$ is a family of tensorial power series if and only if $f$
is a matricial family (Theorem \ref{thm:Matricial_implies_Frechet_hol.}).
Additional information appears to be needed to conclude that such
an $f$ is a Berezin transform. 

Section \ref{sec:Preliminaries} is dedicated to recapping a number
of facts we need from the theory of tensor algebras and Hardy algebras.
Tensorial power series are also introduced and some of their properties
developed. Section \ref{sec:Matricial Families and Functions} is
devoted to developing general properties of matricial sets and functions.
Section \ref{sec:A_special_generator} is devoted to proving a refined
version of Theorem \ref{thm:Double_Commutant}. The connection between
Taylor's Taylor series and our Tensorial Taylor series is made in
Section \ref{Function Theory without a generator}. Finally in Section
\ref{sec:Series-for-matricial} we extend our work in Section \ref{Function Theory without a generator}
to cover maps between correspondence duals.

\subsection*{Acknowledgement}

We are pleased to thank Victor Vinnikov for enlightening conversations
we had with him over matricial function theory. His work in this area,
especially that with Dmitry Kaliuzhny\u{i}-Verbovetsky\u{i} and
Mihai Popa, has been a source of inspiration to us.

\section{Preliminaries\label{sec:Preliminaries}}

Throughout, $M$ will be a fixed $W^{*}$-algebra with separable predual.
However, we shall not fix a representation in advance, but instead
we shall want to access the entire category of normal representations
of $M$ on separable Hilbert space, $NRep(M)$. We shall always assume
that such representations are unital. Strictly speaking, we shall
not want to identify unitarily equivalent representations. But when
$M=\mathbb{C}$, there is no harm in doing so. Also, throughout this
paper, $\Sigma$ will denote an additive subcategory of $NRep(M)$.
This means that whenever $\sigma$ and $\tau$ are in $\Sigma$ so
is $\sigma\oplus\tau$, where $H_{\sigma\oplus\tau}:=H_{\sigma}\oplus H_{\tau}$
and $\sigma\oplus\tau(a):=\sigma(a)\oplus\tau(a)$. In particular,
we shall be interested in additive subcategory generated by a single
representation $\sigma$. It consists of all the \emph{finite} multiples
of $\sigma$, which we write $n\sigma$. When we want to consider
an infinite multiple of $\sigma$, we will be explicit about this
and denote it by $\infty\sigma$. Only at certain points will we need
to assume that $\Sigma$ is a \emph{full} subcategory of $NRep(M)$,
meaning that when $\sigma$ and $\tau$ are in $\Sigma$, then \emph{all}
the intertwiners between $\sigma$ and $\tau$ are also in $\Sigma$.
We will be explicit about where this assumption is used. We will,
however, always assume that if $\sigma$ lies in $\Sigma$ and if
$\tau=m\sigma$, then the natural injection $\iota_{k}$ of $H_{\sigma}$
into $H_{\tau}$, that identifies $H_{\sigma}$ with the $k^{th}$
summand of $H_{\tau}$, is a morphism in $\Sigma$, and so is its
adjoint, $\iota_{k}^{*}$, which is the projection of $H_{\tau}$
onto $H_{\sigma}$.

We shall fix a $W^{*}$-correspondence $E$ over $M$. Thus, $E$
is a self-dual right Hilbert module over $M$ that is endowed with
a left action given by a faithful normal representation $\varphi$
of $M$ in the algebra of all continuous linear right module maps
on $E$, $\mathcal{L}(E)$. (Recall that because $E$ is a self-dual
right Hilbert module over $M$, every bounded right module map on
$E$ is automatically adjointable.) As with Hilbert space representations,
we shall assume normal homomorphisms between $W^{*}$-algebras are
unital and, in particular, we shall assume that $\varphi$ is unital.
Like $M$, $E$ is a dual space. We shall refer to the weak-$*$ topology
on a $W^{*}$-algebra or on a $W^{*}$-correspondence as the \emph{ultraweak
topology} \cite[Paragraph 2.2]{Muhly2011a}.

The $k$-fold tensor powers of $E$, balanced over $M$ (via $\varphi$),
will be denoted $E^{\otimes k}$. Recall that $E^{\otimes k}$ is
the self-dual completion of the $C^{*}$-tensor power and so is a
$W^{*}$-correspondence over $M$. The left action of $M$ on $E^{\otimes k}$
will be denoted by $\varphi_{k}$, i.e., $\varphi_{k}(a)(\xi_{1}\otimes\xi_{2}\otimes\cdots\otimes\xi_{k}):=(\varphi(a)\xi_{1})\otimes\xi_{2}\otimes\cdots\otimes\xi_{k}$.
In particular, $\varphi_{0}$ is just left multiplication of $M$
on $M$ and $\varphi_{1}=\varphi$ \cite[Paragraph 2.7]{Muhly2011a}. 

There are a number of spaces and algebras that we will want to build
from the tensor powers $E^{\otimes k}$. First, there is the Fock
space $\mathcal{F}(E):=M\oplus E\oplus E^{\otimes2}\oplus E^{\otimes3}\oplus\cdots$
Here we interpret the sum as the self-dual completion of the Hilbert
$C^{*}$-module direct sum. It is a natural $W^{*}$-correspondence
over $M$, with the left action $\varphi_{\infty}:=\varphi_{0}\oplus\varphi_{1}\oplus\varphi_{2}\oplus\cdots$
\cite[Paragraph 2.7]{Muhly2011a}. For $\xi\in E$, we write $T_{\xi}$
for the creation operator determined by $\xi$. By definition, $T_{\xi}\eta:=\xi\otimes\eta$,
$\eta\in\mathcal{F}(E)$. It is easy to check that $T_{\xi}$ is bounded,
with norm dominated by $\Vert\xi\Vert$. Since we are assuming $\varphi$
is injective and unital, it follows that $\Vert T_{\xi}\Vert=\Vert\xi\Vert$.
The \emph{norm-closed} subalgebra of $\mathcal{L}(\mathcal{F}(E))$
generated by $\varphi_{\infty}(M)$ and $\{T_{\xi}\}_{\xi\in E}$
will be denoted $\mathcal{T}_{+}(E)$ and called the \emph{tensor
algebra} of $E$. The \emph{Hardy algebra} of $E$ is defined to be
the ultraweak closure of $\mathcal{T}_{+}(E)$ in $\mathcal{L}(\mathcal{F}(E))$.
The (nonclosed) subalgebra of $\mathcal{T}_{+}(E)$ generated by $\varphi_{\infty}(M)$
and$\{T_{\xi}\}_{\xi\in E}$ will be denoted by $\mathcal{T}_{0}(E)$
and called the \emph{algebraic tensor algebra }of $E$. In the purely
algebraic setting it would simply be called \emph{the} tensor algebra
of $E$.%
\footnote{Strictly speaking, because our coefficient algebra $M$ may be noncommutative,
in the algebra literature $\mathcal{T}_{0}(E)$ would be called the
\emph{tensor $M$-ring} determined by the bimodule $E$. (See, e.g.,
\cite[P. 134]{Cohn2006}.)%
} 

Before getting too far along with the theory, it will be helpful to
have an example that we can follow as the theory develops.
\begin{example}
\label{Eg:The basic Example}(The Basic Example) In this example we
let $M=\mathbb{C}$ and we let $E=\mathbb{C}^{d}$. The left and right
actions coincide and the inner product is the usual one, except we
choose it to be conjugate linear in the right hand variable. We interpret
$\mathbb{C}^{d}$ as $\ell^{2}(\mathbb{N})$, when $d=\infty$. The
representations of $M$ on Hilbert space, in this case, are all multiples
of the identity representation $\sigma_{1}$ which represents $\mathbb{C}$
on $\mathbb{C}$ via multiplication, i.e., we may view the objects
in $NRep(M)$ as $\{n\sigma_{1}\mid n\in\mathbb{N}\cup\{\infty\}\}$.
For $m,n\in\mathbb{N}\cup\{\infty\}$, the intertwiner space $\mathcal{I}(m\sigma_{1},n\sigma_{1})$
may be identified with the $n\times m$ matrices, where when either
$m$ or $n$ is infinite, we interpret the $n\times m$ matrices as
bounded linear operators. To keep the notation simple, when it is
convenient we shall write $M_{n}(\mathbb{C})$ for $B(\ell^{2}(\mathbb{N}))$
when $n=\infty$. An additive subcategory of $NRep(M)$ is determined
by any additive subsemigroup of $\mathbb{N}\cup\{\infty\}$.

The Fock space $\mathcal{F}(E)$ is usually called the \emph{full
Fock space of the Hilbert space} $\mathbb{C}^{d}$. It is customary
to identify $\mathcal{F}(\mathbb{C}^{d})$ with $\ell^{2}(\mathbb{F}_{d}^{+})$,
where $\mathbb{F}_{d}^{+}$ denotes the free semigroup on $d$ generators.
The identification is carried out explicitly by letting $\{e_{i}\mid1\leq i\leq d\}$
denote the standard basis for $\mathbb{C}^{d}$ and then sending the
decomposable tensor $e_{i_{1}}\otimes e_{i_{2}}\otimes\cdots\otimes e_{i_{k}}\in(\mathbb{C}^{d})^{\otimes k}\subseteq\mathcal{F}(\mathbb{C}^{d})$
to $\delta_{w}$, where $\delta_{w}$ is the point mass at the word
$w=i_{1}i_{2}\cdots i_{k}$. The vacuum vector $\Omega$ in $\mathcal{F}(\mathbb{C}^{d})$
is sent to the empty word $\emptyset$ in $\mathbb{F}_{d}^{+}$. When
this identification is made, then the creation operator $T_{e_{i}}$
is identified with the $S_{i}$, where $S_{i}$ is the isometry defined
by the equation 
\[
S_{i}\delta_{w}=\delta_{iw},\qquad w\in\mathbb{F}_{d}^{+}.
\]

Also under this identification, the algebraic tensor algebra $\mathcal{T}_{0}(\mathbb{C}^{d})$
is identified with the free algebra on $d$ generators, $\mathbb{C}\langle X_{1},X_{2},\cdots,X_{d}\rangle$,
realized as convolution operators on $\ell^{2}(\mathbb{F}_{d}^{+})$
through the correspondence $X_{i}\leftrightarrow S_{i}$. Of course,
$\mathcal{T}_{+}(\mathbb{C}^{d})$ may be viewed as the norm closed,
unital algebra generated by the $S_{i}$, and, as we mentioned in
the introduction, in this guise it is known as Popescu's \emph{noncommutative
disc algebra} \cite{P1996b}. Its ultraweak closure, $H^{\infty}(E)$,
which in this case coincides with its weak operator closure, is the
\emph{noncommutative analytic Toeplitz algebra}.\end{example}
\begin{defn}
The full Cartesian product $\Pi_{k\geq0}E^{\otimes k}$ will be called
the \emph{formal tensor series algebra} determined by $M$ and $E$
and will be denoted by $\mathcal{T}_{+}((E))$. Elements in $\mathcal{T}_{+}((E))$
will be denoted by formal sums, $\theta\sim\sum_{k\geq0}\theta_{k}$,
$\theta_{k}\in E^{\otimes k}$. 
\end{defn}
It is easy to see that $\mathcal{T}_{+}((E))$ is, indeed, an algebra
under coordinatewise addition and scalar multiplication, with the
product given by the formula 
\[
\theta*\eta=\zeta,
\]
 where $\theta\sim\sum_{k\geq0}\theta_{k}$, $\eta\sim\sum_{k\geq0}\eta_{k}$,
and $\zeta\sim\sum_{k}\zeta_{k}$ are related by the equation 
\[
\zeta_{k}:=\sum_{k=l+m}\theta_{l}\otimes\eta_{m},\qquad k,l,m\geq0.
\]
 Further, we may view $\mathcal{T}_{0}(E)$, $\mathcal{T}_{+}(E)$
and $H^{\infty}(E)$ as subsets of $\mathcal{T}_{+}((E))$ using
the ``Fourier coefficient operators'' calculated with respect to
the gauge automorphism group acting on $\mathcal{L}(\mathcal{F}(E))$
\cite[Paragraph 2.9]{Muhly2011a}. That is, if $P_{n}$ is the projection
of $\mathcal{F}(E)$ onto $E^{\otimes n}$ and if 
\[
W_{t}:=\sum_{n=0}^{\infty}e^{int}P_{n},
\]
 then $\{W_{t}\}_{t\in\mathbb{R}}$ is a one-parameter ($2\pi$-periodic)
unitary group in $\mathcal{L}(\mathcal{F}(E))$ such that if $\{\gamma_{t}\}_{t\in\mathbb{R}}$
is defined by the formula $\gamma_{t}=Ad(W_{t})$, then $\{\gamma_{t}\}_{t\in\mathbb{R}}$
is an ultraweakly continuous group of $*$-automorphisms of $\mathcal{L}(\mathcal{F}(E))$,
called the gauge automorphism group of $\mathcal{L}(\mathcal{F}(E))$.
This group leaves each of the algebras $\mathcal{T}_{0}(E)$, $\mathcal{T}_{+}(E)$,
and $H^{\infty}(E)$ invariant and so do each of the \emph{Fourier
coefficient operators} $\{\Phi_{j}\}_{j\in\mathbb{Z}}$ defined on
$\mathcal{L}(\mathcal{F}(E))$ by the formula 
\begin{equation}
\Phi_{j}(F):=\frac{1}{2\pi}\int_{0}^{2\pi}e^{-int}\gamma_{t}(F)\, dt,\qquad F\in\mathcal{L}(\mathcal{F}(E)),\label{FourierOperators}
\end{equation}
where the integral converges in the ultraweak topology. Further, if
$F\in H^{\infty}(E)$, then $\Phi_{j}(F)$ is of the form $T_{\theta_{j}}$,
where $\theta_{j}\in E^{\otimes j}$. This is clear if $F\in\mathcal{T}_{0}(E)$.
For a general element $F\in H^{\infty}(E),$ the result follows from
the facts that $\mathcal{T}_{0}(E)$ is ultraweakly dense in $H^{\infty}(E)$
and $\Phi_{j}$ is ultraweakly continuous. In fact, if $\Sigma_{k}$
is defined by the formula 
\begin{equation}
\Sigma_{k}(F):=\sum_{|j|<k}(1-\frac{|j|}{k})\Phi_{j}(F),\label{eq:Cesaro}
\end{equation}
 $F\in\mathcal{L}(\mathcal{F}(E))$, then for $F\in\mathcal{L}(\mathcal{F}(E))$,
$\lim_{k\to\infty}\Sigma_{k}(F)=F$, in the ultra-weak topology. It
follows that for $F\in H^{\infty}(E)$, $F$ determines and is uniquely
determined by its sequence of Fourier coefficients $\theta=\{\theta_{j}\}_{j\geq0}$,
where $T_{\theta_{j}}:=\Phi_{j}(F)$.

The correspondence $E$ determines an endo-functor $\Phi_{E}$ from
$NRep(M)$ into $NRep(M)$. For an object $\sigma$ in $NRep(M),$
$\Phi_{E}(\sigma)=\sigma^{E}\circ\varphi$, where $\sigma^{E}$ is
the representation of $\mathcal{L}(E)$ that is induced from $\sigma$
in the sense of Rieffel \cite{R1974b}: For $X\in\mathcal{L}(E)$,
$\sigma^{E}(X)=X\otimes I_{H_{\sigma}}$ acting on $E\otimes_{\sigma}H_{\sigma}$.
It is \cite[Theorem 5.3]{R1974b} that ensures that for $T\in\mathcal{I}(\sigma,\tau)$,
$\Phi_{E}(T)$ should be $I_{E}\otimes T\in\mathcal{I}(\Phi_{E}(\sigma),\Phi_{E}(\tau))$.

Following the notation in \cite[Paragraph 2.6]{Muhly2011a} and elsewhere
in our work, we set $E^{\sigma*}=\mathcal{I}(\sigma^{E}\circ\varphi,\sigma)$
for each $\sigma\in NRep(M)$. Also, for $\mathfrak{z}_{0}\in E^{\sigma*}$
and $0\leq R\leq\infty$, we write $\mathbb{D}(\mathfrak{z}_{0},R,\sigma,E)$
for the open disc or ball in $E^{\sigma*}$ of radius $R$ centered
at $\mathfrak{z}_{0}$. Usually, $E$ will be dropped from the notation,
since it will be understood from context. 

The space $E^{\sigma*}$ is a bimodule over $\sigma(M)'$, where for
$a,b\in\sigma(M)'$ and $\xi\in E^{\sigma*}$, $a\cdot\xi\cdot b:=a\xi(I_{E}\otimes b)$.
In fact, it is a left Hilbert module over $\sigma(M)'$, where the
inner product is given by the formula $\langle\xi,\eta\rangle=\xi\eta^{*}$
and as such, $E^{\sigma*}$ becomes a correspondence over $\sigma(M)'$.
Although the correspondence properties of $E^{\sigma*}$ have figured
heavily elsewhere in our work, they do not play much of a role here.
Only the bimodule properties are of consequence in this section. (In
Section \ref{sec:Series-for-matricial} we will use properties of
the correspondence $E^{\sigma}$, which is a right Hilbert $W^{*}$-module
over $\sigma(M)'.$\label{E-sigma_correspondence}) 

Once $\sigma$ is fixed, each $\mathfrak{z}\in E^{\sigma*}$ determines
a homomorphism, denoted $\sigma\times\mathfrak{z}$, from $\mathcal{T}_{0}(E)$
to $B(H_{\sigma})$, which is defined on the generators of $\mathcal{T}_{0}(E)$
via the formulas 
\[
\sigma\times\mathfrak{z}(\varphi_{\infty}(a))=\sigma(a),\qquad a\in M,
\]

\[
\sigma\times\mathfrak{z}(T_{\xi})(h)=\mathfrak{z}(\xi\otimes h),\qquad\xi\in E,\, h\in H_{\sigma}.
\]
The fact that these equations define a representation of $\mathcal{T}_{0}(E)$
on $B(H_{\sigma})$ follows from the universal properties of tensor
algebras once it is checked that the map $\xi\to\sigma\times\mathfrak{z}(T_{\xi})$
is a bimodule map in the sense that 
\[
\sigma\times\mathfrak{z}(T_{a\cdot\xi\cdot b})=\sigma(a)(\sigma\times\mathfrak{z}(T_{\xi}))\sigma(b),\qquad a,b\in M,\,\xi\in E.
\]
And conversely, every representation $\rho$ of $\mathcal{T}_{0}(E)$
on Hilbert space such that the restriction of $\rho$ to $\varphi_{\infty}(M)$
in $\mathcal{T}_{0}(E)$ is a normal representation of $M$, say $\sigma$,
and the restriction of $\rho$ to the subspace $\{T_{\xi}\}_{\xi\in E}$
yields an ultraweakly continuous bimodule map of $E$, must be of
the form $\sigma\times\mathfrak{z}$, where $\mathfrak{z}$ is given
by the formula $\mathfrak{z}(\xi\otimes h)=\rho(T_{\xi})h$.

As we indicated in the introduction, it is of fundamental importance
for our theory, that $\sigma\times\mathfrak{z}$ extends from $\mathcal{T}_{0}(E)$
to a \emph{completely contractive representation} of $\mathcal{T}_{+}(E)$
on $B(H_{\sigma})$ if and only if $\mathfrak{z}$ lies in the norm
closed disc $\overline{\mathbb{D}(0,1,\sigma)}$ \cite[Theorem 2.9]{Muhly2011a}.
Moreover, for $\Vert\mathfrak{z}\Vert<1$, $\sigma\times\mathfrak{z}$
extends further to be an \emph{ultraweakly continuous}, completely
contractive representation of $H^{\infty}(E)$ on $H_{\sigma}$ \cite[Corollary 2.14]{Muhly2004a}.
We write $\mathcal{AC}(\sigma,E)$, or simply $\mathcal{AC}(\sigma)$,
when $E$ is understood, for the collection of all $\mathfrak{z}\in\overline{\mathbb{D}(0,1,\sigma)}$
such that $\sigma\times\mathfrak{z}$ extends to an ultraweakly continuous,
completely contractive representation of $H^{\infty}(E)$ on $H_{\sigma}$.
We then call $\mathfrak{z}$ an \emph{absolutely continuous} point
in $\overline{\mathbb{D}(0,1,\sigma)}$ and we say that  $\sigma\times\mathfrak{z}$
is an \emph{absolutely continuous representation}. The terminology
derives from the fact that one can define an $H^{\infty}$-functional
calculus in the sense of Sz.-Nagy and Foia\c{s} for a contraction
operator if and only if its minimal unitary dilation is absolutely
continuous. (See \cite{Muhly2011a}, where this perspective is developed
at length.)
\begin{rem}
\label{Xeta} There is now a bit of ambiguity that needs to be clarified.
When we write $\mathcal{I}(\sigma\times\mathfrak{z},\tau\times\mathfrak{w})$
we shall \emph{always} mean the set of all operators $C:H_{\sigma}\to H_{\tau}$
such that $C\in\mathcal{I}(\sigma,\tau)$ and $C\mathfrak{z}=\mathfrak{w}(I_{E}\otimes C)$.
It is straightforward to see that this happens if and only if $C\in\mathcal{I}(\sigma,\tau)$
and $C\sigma\times\mathfrak{z}(F)=\tau\times\mathfrak{w}(F)C$ for
all $F\in\mathcal{T}_{0}(E)$. Since $\mathcal{T}_{0}(E)$ is norm
dense in $\mathcal{T}_{+}(E)$ and since $\sigma\times\mathfrak{z}$
and $\tau\times\mathfrak{w}$ extend to continuous representations
of $\mathcal{T}_{+}(E)$ when $\Vert\mathfrak{z}\Vert,\Vert\mathfrak{w}\Vert\leq1$,
it follows that $\mathcal{I}(\sigma\times\mathfrak{z},\tau\times\mathfrak{w})=\{C:H_{\sigma}\to H_{\tau}\mid C\sigma\times\mathfrak{z}(F)=\tau\times\mathfrak{w}(F)C\,\mbox{for all}\, F\in\mathcal{T}_{+}(E)\}$,
when $\Vert\mathfrak{z}\Vert,\Vert\mathfrak{w}\Vert\leq1$. Likewise,
since $\mathcal{AC}(\sigma)$ is the set of all $\mathfrak{z}\in E^{\sigma*}$
such that $\sigma\times\mathfrak{z}$ extends from $\mathcal{T}_{+}(E)$
to an ultraweakly continuous completely contractive representation
of $H^{\infty}(E)$, we find that $\mathcal{I}(\sigma\times\mathfrak{z},\tau\times\mathfrak{w})=\{C:H_{\sigma}\to H_{\tau}\mid C\sigma\times\mathfrak{z}(F)=\tau\times\mathfrak{w}(F)C\,\mbox{for all}\, F\in H^{\infty}(E)\}$
when $\mathfrak{z}\in\mathcal{AC}(\sigma)$ and $\mathfrak{w}\in\mathcal{AC}(\tau)$.
\end{rem}
With the notation we have established, it follows that $\mathcal{T}_{0}(E)$
may be represented as a space polynomial-like, $B(H_{\sigma})$-valued
functions on $E^{\sigma*}$ via the formula: 
\begin{equation}
\widehat{F}_{\sigma}(\mathfrak{z}):=\sigma\times\mathfrak{z}(F),\qquad F\in\mathcal{T}_{0}(E),\,\mathfrak{z}\in E^{\sigma*}.\label{eq:F_sigma_1}
\end{equation}

Since our primary objective is to understand the nature of the functions
$\widehat{F}_{\sigma}$, it will be helpful to work carefully through
some examples that illuminate their definition. Consider first the
case when $F=T_{\xi_{1}}T_{\xi_{2}}=T_{\xi_{1}\otimes\xi_{2}}$. Then
for $h\in H_{\sigma}$, we have 
\begin{equation}
\begin{gathered}(\sigma\times\mathfrak{z})(T_{\xi_{1}}T_{\xi_{2}})(h)=(\sigma\times\mathfrak{z})(T_{\xi_{1}})(\sigma\times\mathfrak{z})(T_{\xi_{2}})(h)\\
=\mathfrak{z}(\xi_{1}\otimes\mathfrak{z}(\xi_{2}\otimes h))\\
=\mathfrak{z}(I_{E}\otimes\mathfrak{z})(\xi_{1}\otimes\xi_{2}\otimes h).
\end{gathered}
\label{eq:Powers_1}
\end{equation}
 More generally, we find that 
\[
(\sigma\times\mathfrak{z})(T_{\xi_{1}\otimes\xi_{2}\otimes\cdots\otimes\xi_{k}})(h)=\mathfrak{z}(I_{E}\otimes\mathfrak{z})\cdots(I_{E^{\otimes k-1}}\otimes\mathfrak{z})(\xi_{1}\otimes\xi_{2}\otimes\cdots\otimes\xi_{k}\otimes h).
\]
 Observe that $\mathfrak{z}(I_{E}\otimes\mathfrak{z})\cdots(I_{E^{\otimes k-1}}\otimes\mathfrak{z})$
is a bounded linear operator from $E^{\otimes k}\otimes H_{\sigma}$
to $H_{\sigma}$, with norm dominated by $\Vert\mathfrak{z}\Vert^{k}$.
Further, $\mathfrak{z}(I_{E}\otimes\mathfrak{z})\cdots(I_{E^{\otimes k-1}}\otimes\mathfrak{z})$
intertwines $\sigma^{E^{\otimes k}}\circ\varphi_{k}$ and $\sigma$.
Thus $\mathfrak{z}(I_{E}\otimes\mathfrak{z})\cdots(I_{E^{\otimes k-1}}\otimes\mathfrak{z})$
is an element of $\mathcal{I}(\sigma^{E^{\otimes k}}\circ\varphi_{k},\sigma)=(E^{\otimes k})^{\sigma*}$.
In \cite[Page 363]{Muhly2004a} $\mathfrak{z}(I_{E}\otimes\mathfrak{z})\cdots(I_{E^{\otimes k-1}}\otimes\mathfrak{z})$
is called \emph{the $k^{th}$ generalized power} of $\mathfrak{z}$.
We shall denote it here by $\mathfrak{z}^{(k)}$ and we shall denote
the function $\mathfrak{z}\to\mathfrak{z}^{(k)}$, from $E^{\sigma*}$
to $(E^{\otimes k})^{\sigma*}$, by $\mathcal{Z}_{k}$, i.e., $\mathcal{Z}_{k}(\mathfrak{z}):=\mathfrak{z}^{(k)}=\mathfrak{z}(I_{E}\otimes\mathfrak{z})\cdots(I_{E^{\otimes k-1}}\otimes\mathfrak{z})$.
By convention, we set $\mathcal{Z}_{0}(\mathfrak{z})\equiv I_{H_{\sigma}}$. 
\begin{example}
\label{eg: Basic_Example_update_1} In the context of our basic example,
Example \ref{Eg:The basic Example}, $E\otimes_{\sigma}H_{\sigma}$
is $\mathbb{C}^{d}\otimes H_{\sigma}$, which we shall view simply
as columns of length $d$ of vectors from $H_{\sigma}$. We may do
this, since there is no consequential balancing going on. After all,
$\sigma$ is a representation of $\mathbb{C}$. It follows that $E^{\sigma*}:=\mathcal{I}(\sigma^{E}\circ\varphi,\sigma)$
may be identified with all $d$-tuples of operators in $B(H_{\sigma})$
arranged as a row. That is, we may and shall write $E^{\sigma*}=\{\mathfrak{z}=(Z_{1},Z_{2},\cdots,Z_{d})\mid Z_{i}\in B(H_{\sigma})\}$.
So, if $F\in\mathcal{T}_{0}(E)$ and if we identify $\mathcal{T}_{0}(E)$
with the free algebra on $d$ generators, as we did in our basic example,
we may write $F=\sum_{w\in\mathbb{F}_{d}^{+}}a_{w}X_{w}$, and then
equation \eqref{eq:F_sigma_1} becomes
\[
\widehat{F}_{\sigma}(\mathfrak{z})=\sum_{w\in\mathbb{F}_{d}^{+}}a_{w}Z_{w},
\]
where $Z_{w}=Z_{i_{1}}Z_{i_{2}}\cdots Z_{i_{k}}$, $w=i_{1}i_{2}\cdots i_{k}$.
In particular, if $\sigma$ happens to be an $n$-dimensional representation
with $n<\infty$, then $E^{\sigma*}$ is the collection of all $d$-tuples
of $n\times n$ matrices and equation \eqref{eq:F_sigma_1} yields
the representation of the free algebra on $d$ generators on the algebra
of $d$ \emph{generic $n\times n$ matrices}. The function $\mathcal{Z}_{k}$
in this setting assigns to $\mathfrak{z}=(Z_{1},Z_{2},\cdots,Z_{d})$
the tuple $(Z_{w})_{\vert w|=k}$ of length $d^{k}$ that is ordered
lexicographically by the words of length $k$. 
\end{example}
In the following lemma and, indeed, throughout this paper, we identify
$E^{\otimes(k+l)}$ with $E^{\otimes k}\otimes E^{\otimes l}$; $E^{\otimes0}:=M$. 
\begin{lem}
\label{lem:Derivative_of_Z_k} The functions $\mathcal{Z}_{k}$, from
$E^{\sigma*}$ to $(E^{\otimes k})^{\sigma*}$, satisfy the equation
\[
\mathcal{Z}_{k+l}(\mathfrak{z})=\mathcal{Z}_{k}(\mathfrak{z})(I_{E^{\otimes k}}\otimes\mathcal{Z}_{l}(\mathfrak{z})),\qquad\mathfrak{z}\in E^{\sigma*}.
\]
Further, each $\mathcal{Z}_{k}$ is Frechet differentiable and the
Frechet derivative of $\mathcal{Z}_{k}$, denoted $D\mathcal{Z}_{k}$,
is given by the formula 
\begin{equation}
(D\mathcal{Z}_{k})(\mathfrak{z})[\zeta]=\sum_{l=0}^{k-1}\mathcal{Z}_{l}(\mathfrak{z})(I_{E^{\otimes l}}\otimes\zeta)(I_{E^{\otimes(l+1)}}\otimes\mathcal{Z}_{k-l-1}(\mathfrak{z}))\label{derivative-1}
\end{equation}
 for all $\mathfrak{z},\zeta\in E^{\sigma*}$. 
\end{lem}
The proof is a straightforward calculation based on the definition
of $\mathcal{Z}_{k}$ and will be omitted. Observe that for each $\mathfrak{z}\in E^{\sigma*}$,
$D\mathcal{Z}_{k}(\mathfrak{z})$ is a bounded linear operator from
$E^{\sigma*}$ to $(E^{\otimes k})^{\sigma*}$ whose norm is dominated
by $k\Vert\mathfrak{z}\Vert^{k-1}$.

If $\theta\in E^{\otimes k}$, we shall write $L_{\theta}$ for the
linear operator from $H_{\sigma}$ to $E^{\otimes k}\otimes_{\sigma}H_{\sigma}$
that maps $h$ to $\theta\otimes h$. Evidently, $\Vert L_{\theta}\Vert\leq\Vert\theta\Vert$,
with equality holding if $\sigma$ is faithful. The calculations that
led from equation \eqref{eq:Powers_1} through Lemma \ref{lem:Derivative_of_Z_k}
immediately yield 
\begin{thm}
\label{thm: Basic_Calculus} Suppose $F=\sum_{k=0}^{n}T_{\theta_{k}}\in\mathcal{T}_{0}(E)$,
with $\theta_{k}\in E^{\otimes k}$. Then for $\sigma\in NRep(M)$
and $\mathfrak{z}\in E^{\sigma*}$, the operator $\widehat{F}_{\sigma}(\mathfrak{z})$
in $B(H_{\sigma})$ is given by the formula 
\begin{equation}
\widehat{F}_{\sigma}(\mathfrak{z})=\sum_{k=0}^{n}\mathcal{Z}_{k}(\mathfrak{z})L_{\theta_{k}}.\label{eq:F_sigma_2}
\end{equation}
 Further, as a function of $\mathfrak{z}\in E^{\sigma*}$, $\widehat{F}_{\sigma}$
is Frechet differentiable and its Frechet derivative, mapping $E^{\sigma*}$
to $B(H_{\sigma})$, is given by the formula 
\begin{equation}
D\widehat{F}_{\sigma}(\mathfrak{z})[\zeta]=\sum_{k=0}^{n}D\mathcal{Z}_{k}(\mathfrak{z})[\zeta]L_{\theta_{k}}.\label{eq:Deriv_F_sigma_2}
\end{equation}

\end{thm}
Of course, we would like to extend our function theory beyond the
realm of the algebraic tensor algebra. For this purpose, we follow
Popescu \cite{P2006a} (who followed Hadamard, who followed Cauchy)
and introduce the following definition. 
\begin{defn}
\label{def:Radius_of_convergence} For $\theta\sim\sum_{k\geq0}\theta_{k}$
in $\mathcal{T}_{+}((E))$, we define $R(\theta)$ to be 
\[
R(\theta)=(\limsup_{k\to\infty}\Vert\theta_{k}\Vert^{\frac{1}{k}})^{-1},
\]
 and we call $R(\theta)$ the \emph{radius of convergence} of the
formal series $\theta$. 
\end{defn}
Of course, $R(\theta)$ is a non-negative real number or $+\infty$. 
\begin{prop}
\label{prop:Power_series_give_holomorphic_fcns} Suppose $\theta\sim\sum_{k\geq0}\theta_{k}$
in $\mathcal{T}_{+}((E))$ and $R=R(\theta)$ is not zero. Then for
each $\sigma\in NRep(M)$, for each $\mathfrak{z}_{0}\in E^{\sigma*}$,
and for each $\mathfrak{z}\in\mathbb{D}(\mathfrak{z}_{0},R,\sigma)$,
the series 
\begin{equation}
\sum_{k=0}^{\infty}\mathcal{Z}_{k}(\mathfrak{z}-\mathfrak{z}_{0})L_{\theta_{k}}\label{eq:Power_series}
\end{equation}
converges in the norm of $B(H_{\sigma})$. The convergence is uniform
on any subdisk $\mathbb{D}(\mathfrak{z}_{0},\rho,\sigma)$, $\rho<R$;
\emph{and it is uniform in} $\sigma$, as well. The resulting function,
$f_{\sigma}$, is analytic as a map from $\mathbb{D}(\mathfrak{z}_{0},R,\sigma)$
to $B(H_{\sigma})$ and the Frechet derivative of $f_{\sigma}$ is
given by the formula 
\[
Df_{\sigma}(\mathfrak{z})[\zeta]=\sum_{k=1}^{\infty}D\mathcal{Z}_{k}(\mathfrak{z}-\mathfrak{z}_{0})[\zeta]L_{\theta_{k}},\qquad\mathfrak{z}\in\mathbb{D}(\mathfrak{z}_{0},R,\sigma),\,\zeta\in E^{\sigma*}.
\]
\end{prop}
\begin{proof}
The uniform convergence of \eqref{eq:Power_series} follows a standard
argument: Fix $\rho,\rho'$ such that $0<\rho<\rho'<R$. Since $\frac{1}{R(\theta)}=\limsup_{k}\norm{\theta_{k}}^{1/k}$,
there is an $m$ such that, for all $k\geq m$, $\norm{\theta_{k}}^{1/k}<1/\rho'$.
For such $k$, for all $\sigma$ and for all $\mathfrak{z}\in\mathbb{D}(\mathfrak{z}_{0},\rho,\sigma)$,
$\norm{\mathcal{Z}_{k}(\mathfrak{z}-\mathfrak{z}_{0})L_{\theta_{k}}}\leq\norm{\mathfrak{z}-\mathfrak{z}_{0}}^{k}\norm{\theta_{k}}\leq(\frac{\rho}{\rho'})^{k}$.
Since $\rho<\rho'$, we see that \eqref{eq:Power_series} converges
uniformly on $\mathbb{D}(\mathfrak{z}_{0},\rho,\sigma)$. The remaining
assertions can be proved by similar elementary estimates. Alternatively,
one can appeal to \cite[Theorems 3.17.1 and 3.18.1]{HP1974}.\end{proof}
\begin{defn}
\label{def:Tensorial_power_series}The series \eqref{eq:Power_series}
is called the \emph{tensorial power series }determined by the formal
tensor series $\theta$, the point $\mathfrak{z}_{0}$, and the representation
$\sigma$.\end{defn}
\begin{rem}
Following the line of thought developed in elementary texts on complex
function theory, one might expect that if $\mathfrak{z}\in E^{\sigma*}$
has norm larger than $R(\theta)$, then the series \eqref{eq:Power_series}
diverges. However, the situation is quite a bit more complicated than
in one complex variable. For example, if $M=E=\mathbb{C}$ and if
$\dim H_{\sigma}\geq2$, then we may identify $E^{\sigma*}$ with
$B(H_{\sigma})$, and if $\mathfrak{z}$ is any \emph{nilpotent} element
in $E^{\sigma*}$, the series \eqref{eq:Power_series} converges no
matter what $R(\theta)$ is. The situation is even more complicated:
There are series $\theta$ with finite $R(\theta)$ such that for
``most'' $\sigma$ the series \eqref{eq:Power_series} converges
on all of $E^{\sigma*}$, i.e., almost all $f_{\sigma}$ are entire.
A particularly instructive example, due to Luminet \cite[Example 2.9]{Lum86},
is constructed as follows. Let $S_{k}(X_{1},X_{2},\cdots,X_{k})$
denote the standard identity in $k$ noncommuting variables: 
\[
S_{k}(X_{1},X_{2},\cdots,X_{k})=\sum_{s\in\mathfrak{S}_{k}}{\rm sign}(s)X_{s(1)}X_{s(2)}\cdots X_{s(k)},
\]
 where $\mathfrak{S}_{k}$ denotes the permutation group on $k$ letters
and where ${\rm sign}(s)$ is $1$ if $s$ is even and $-1$ if $s$
is odd. Then the set of these identities determines an element $\theta$
in $\mathcal{T}_{+}((\mathbb{C}^{2}))$ by identifying $\mathcal{T}_{+}((\mathbb{C}^{2}))$
with the free formal series in two noncommuting variables $X_{1}$
and $X_{2}$ and setting 
\[
\theta(X_{1},X_{2})=\sum_{k\geq2}S_{k}(X_{1},X_{1}X_{2},\cdots,X_{1}X_{2}^{k-1}).
\]
 When this series is written as a series of words in $X_{1}$ and
$X_{2}$, 
\[
\theta(X_{1},X_{2})=\sum_{w\in\mathbb{F}_{2}^{+}}\lambda_{w}X_{w},
\]
one sees that for each positive integer $d$ there is a word $w$
of length at least $d$ so that $\lambda_{w}=1$. It follows that
$R(\theta)<\infty$. However, if $\sigma$ is any finite dimensional
representation of $\mathbb{C}=M$, the series \eqref{eq:Power_series}
converges throughout $E^{\sigma*}$. Indeed, for any given $\mathfrak{z}$
there are only finitely many nonzero terms in the series. On the other
hand, we will be able to show later that given $R'>R(\theta)$, there
is a $\sigma\in NRep(M)$ and an element $\mathfrak{z}\in E^{\sigma*}$
with $\Vert\mathfrak{z}\Vert=R'$ such that the series \eqref{eq:Power_series}
diverges. 
\end{rem}
The next proposition establishes a bridge between the theory developed
in \cite{Muhly2004a,Muhly2005b,Muhly2008b,Muhly2009} and the focus
of this paper. First, recall that if $F\in H^{\infty}(E)$, then the
Fourier coefficients $\Phi_{j}(F)$ are of the form $T_{\theta_{j}}$
for $\theta_{j}\in E^{\otimes j}$. Further, the series $\sum_{j\geq0}T_{\theta_{j}}$
is Cesaro summable to $F$ \eqref{eq:Cesaro}. On the other hand,
each Fourier coefficient operator $\Phi_{j}$ is contractive and so
the norm of each $\theta_{j}$ is dominated by $\Vert F\Vert$. Thus
the radius of convergence of the series $\theta\sim\sum_{j\geq0}\theta_{j}$
is at least one. Since $\sigma\times\mathfrak{z}(T_{\theta})=\mathcal{Z}_{j}(\mathfrak{z})L_{\theta}$
for $\theta\in E^{\otimes j}$ and $\mathfrak{z}\in\mathbb{D}(0,1,\sigma)$,
the validity of the following proposition is immediate. 
\begin{prop}
\label{prop:H_infty_values}For $F\in H^{\infty}(E)$ and $\sigma\in\Sigma$,
the $\sigma$-Berezin transform $\widehat{F}_{\sigma}$, defined on
$\mathcal{AC}(\sigma)$ by the formula, $\widehat{F}_{\sigma}(\mathfrak{z}):=\sigma\times\mathfrak{z}(F)$,
admits the tensorial power series expansion, 
\begin{equation}
\widehat{F}_{\sigma}(\mathfrak{z})=\sum_{j=0}^{\infty}\mathcal{Z}_{j}(\mathfrak{z})L_{\theta_{j}},\label{expression}
\end{equation}
where $T_{\theta_{j}}=\Phi_{j}(F)$. The series converges in norm
for $\mathfrak{z}\in\mathbb{D}(0,1,\sigma)$, uniformly in $\sigma$
and on any sub-disk, $\mathbb{D}(0,\rho,\sigma)$, $\rho<1$. Moreover,
the function $\widehat{F}_{\sigma}(\mathfrak{z})$ is bounded by $\Vert F\Vert$
throughout $\mathcal{AC}(\sigma)$. \end{prop}
\begin{rem}
\label{Rem:Arveson's_discovery} In view of Proposition \ref{prop:H_infty_values},
one might wonder if every bounded analytic function on $\mathbb{D}(0,1,\sigma)$,
with tensorial power series $\sum_{j=0}^{\infty}\mathcal{Z}_{j}(\mathfrak{z})L_{\theta_{j}}$,
comes from a function in $H^{\infty}(E)$. Thanks to a very detailed
study by Arveson, such is not the case. He shows that in one of the
simplest settings, when $M=\mathbb{C}$, $E=\mathbb{C}^{d}$ ($d\geq2$),
and $\sigma$ is one-dimensional, there are bounded analytic functions
on $\mathbb{D}(0,1,\sigma)$ (which is the unit ball in $\mathbb{C}^{d}$)
which are not of the form $\widehat{F}_{\sigma}(\mathfrak{z})$ for
any element $F$ in $H^{\infty}(\mathbb{C}^{d})$ (see \cite[Theorem 3.3]{Arv1998}). 
\end{rem}

\section{Matricial Families and Functions\label{sec:Matricial Families and Functions}}

Suppose $\sigma$ and $\tau$ are two normal representations of our
$W^{*}$-algebra $M$ on Hilbert spaces $H_{\sigma}$ and $H_{\tau}$,
respectively. We have noted that if one writes $\sigma\oplus\tau$
matricially as $(\sigma\oplus\tau)(\cdot)=\begin{bmatrix}\sigma(\cdot) & 0\\
0 & \tau(\cdot)
\end{bmatrix}$, then $E^{(\sigma\oplus\tau)*}=\mathcal{I}((\sigma\oplus\tau)^{E}\circ\varphi,(\sigma\oplus\tau))$
may be written as matrices of operators $\begin{bmatrix}\mathfrak{z_{11}} & \mathfrak{z_{12}}\\
\mathfrak{z_{21}} & \mathfrak{z_{22}}
\end{bmatrix}$, viewed as operators from $H_{\sigma^{E}\circ\varphi}\oplus H_{\tau^{E}\circ\varphi}$
to $H_{\sigma}\oplus H_{\tau}$, where $\mathfrak{z_{11}}\in\mathcal{I}(\sigma^{E}\circ\varphi,\sigma)$,
$\mathfrak{z_{12}}\in\mathcal{I}(\tau^{E}\circ\varphi,\sigma)$, $\mathfrak{z_{21}}\in\mathcal{I}(\sigma^{E}\circ\varphi,\tau)$,
and where $\mathfrak{z_{22}}\in\mathcal{I}(\tau^{E}\circ\varphi,\tau)$.
In particular, note that all the matrices of the form $\begin{bmatrix}\mathfrak{z_{11}} & 0\\
0 & \mathfrak{z_{22}}
\end{bmatrix}$, with $\mathfrak{z_{11}}\in\mathcal{I}(\sigma^{E}\circ\varphi,\sigma)$
$ $and $\mathfrak{z_{22}}\in\mathcal{I}(\tau^{E}\circ\varphi,\tau)$
are contained $E^{(\sigma\oplus\tau)*}$. We abbreviate this fact
by writing $E^{\sigma*}\oplus E^{\tau*}\subseteq E^{(\sigma\oplus\tau)*}$.
We will have occasion later to think about higher order matrices.
That is, if $\sigma_{1},\sigma_{2},\cdots,\sigma_{n}$ are $n$ (not-necessarily
distinct) representations in $NRep(M)$, then $E^{(\sigma_{1}\oplus\sigma_{2}\oplus\cdots\oplus\sigma_{n})*}$
can be viewed as $n\times n$ matrices whose $i,j$-entries lie in
$\mathcal{I}(\sigma_{j}^{E}\circ\varphi,\sigma_{i})$. When this is
done, we will be able to write $\sum_{1\leq i\leq n}^{\oplus}E^{\sigma_{i}*}\subseteq E^{(\sigma_{1}\oplus\sigma_{2}\oplus\cdots\oplus\sigma_{n})*}$.
\begin{rem}
\label{def:matricial_family_of_sets} In Definition \ref{def:Matricial_family_of_sets},
we defined the notion of a matricial family of sets indexed by all
of $NRep(M)$. Evidently, that notion can be ``relativized'' to additive
subcategories. So if $\Sigma$ is an additive subcategory of $NRep(M)$,
then a family of sets $\mathcal{U}=\{\mathcal{U}(\sigma)\}_{\sigma\in\Sigma}$
that satisfies the following two properties will be called a \emph{matricial
$E,\Sigma$-family} \emph{of sets}. If $\Sigma$ or $E$ are clear
from context, we shall simply call $\mathcal{U}$ a matricial family
of sets.
\begin{enumerate}
\item Each $\mathcal{U}(\sigma)$ is contained in $E^{\sigma*}$. 
\item The family $\{\mathcal{U}(\sigma)\}_{\sigma\in\Sigma}$ is \emph{closed
with respect to taking direct} sums, i.e., 
\[
\mathcal{U}(\sigma)\oplus\mathcal{U}(\tau):=\begin{bmatrix}\mathcal{U}(\sigma) & 0\\
0 & \mathcal{U}(\tau)
\end{bmatrix}\subseteq\mathcal{U}(\sigma\oplus\tau).
\]

\end{enumerate}
\end{rem}
If each of the members of $\mathcal{U}$ is described by a common
property, such as being open or a domain, then we shall adjust the
terminology appropriately, e.g., by saying that the family is a matricial
$E,\Sigma$-family of open sets or domains. We shall say that $\mathcal{U}$
\emph{unitarily invariant,} if each $\mathcal{U}(\sigma)$ is \emph{unitarily
invariant }in the sense that for each unitary operator $u\in\sigma(M)'$
and each $\mathfrak{z}\in\mathcal{U}(\sigma)$, $u^{-1}\cdot\mathfrak{z}\cdot u:=u^{-1}\mathfrak{z}(I_{E}\otimes u)$
lies in $\mathcal{U}(\sigma)$. Relatedly, $\mathcal{U}$ is called
\emph{matricially convex}, if for any $\sigma$ and $\tau$ in $\Sigma$
and for all $v\in\mathcal{I}(\sigma,\tau)$ such that $vv^{*}=I_{H_{\tau}}$,
$v\cdot\mathcal{U}(\sigma)\cdot v^{*}=v\mathcal{U}(\sigma)(I_{E}\otimes v^{*})\subseteq\mathcal{U}(\tau)$. 
\begin{example}
\label{eg:Examples_matricial_E,sigma_sets} We already have noted
that $\{\mathbb{D}(0,1,\sigma)\}_{\sigma\in NRep(M)}$ is a matricially
convex, matricial family of sets. However, the best we can say, in
general, is that $\{\mathcal{AC}(\sigma)\}_{\sigma\in NRep(M)}$ is
unitarily invariant. For a more general class of discs, suppose that
$\{\zeta_{\sigma}\}_{\sigma\in\Sigma}$ is a family of vectors with
$\zeta_{\sigma}\in E^{\sigma*}$ such that $\zeta_{\sigma\oplus\tau}=\zeta_{\sigma}\oplus\zeta_{\tau}$
for all $\sigma,\tau\in\Sigma$. We shall call $\zeta:=\{\zeta_{\sigma}\}_{\sigma\in\Sigma}$
an \emph{additive field of vectors} over $\Sigma$. Given such a field
$\zeta$ and an $R$, $0\leq R\leq\infty$, we shall call $\mathbb{D}(\zeta,R):=\{\mathbb{D}(\zeta_{\sigma},R,\sigma)\}_{\sigma\in\Sigma}$
the \emph{matricial disc} determined by the field $\zeta$. If $\zeta_{\sigma}=0$
for all $\sigma\in\Sigma$, we shall simply write $\mathbb{D}(0,R)$,
calling it the matricial disc $\{\mathbb{D}(0,R,\sigma)\}_{\sigma\in\Sigma}$
we already have defined. It is clear that $\mathbb{D}(\zeta,R)$ is
a matricial set when $\zeta$ is an additive field of vectors over
$\Sigma$, but in general $\mathbb{D}(\zeta,R)$ won't be matricially
convex. However, it will be matricially convex when $\zeta=\{\zeta_{\sigma}\}_{\sigma\in\Sigma}$
is a \emph{central }additive field in the sense that $\zeta_{\tau}(I_{E}\otimes C)=C\zeta_{\sigma}$
for all $C\in\mathcal{I}(\sigma,\tau)$, $\sigma,\tau\in\Sigma$,
as is evident from the definition. Note that for $\zeta$ to be a
central additive field requires more than each $\zeta_{\sigma}$ being
central in the sense of \cite[Definition 4.11]{Muhly2008b}, which
simply means that $\zeta_{\sigma}(I_{E}\otimes C)=C\zeta_{\sigma}$
for all $C\in\mathcal{I}(\sigma,\sigma)=\sigma(M)'$. However, if
each $\zeta_{\sigma}$ is central in that sense, then $\mathbb{D}(\zeta,R)$
is a unitarily invariant matricial set. Note, too, that if $\Sigma$
is the subcategory consisting of all multiples of a single representation
$\sigma$, together with all the spaces $\mathcal{I}(n\sigma,m\sigma)$,
$n,m\in\mathbb{N}$, and if $\zeta_{\sigma}$ is a central vector
in $E^{\sigma*}$, then $\{\zeta_{m\sigma}\}_{m\in\mathbb{N}}$ is
a central additive field on $\Sigma$, where $\zeta_{m\sigma}$ is
the $m$-fold direct sum of copies of $\zeta_{\sigma}$.
\end{example}
To see these examples in the simplest, most concrete setting, consider
the following addition to our basic example, Example \ref{Eg:The basic Example}. 
\begin{example}
\label{eg:Concrete example} Let $M=\mathbb{C}$ and let $E=\mathbb{C}^{d}$.
Then since every representation $\sigma$ of $\mathbb{C}$ is simply
a multiple of the basic, $1$-dimensional representation $\sigma_{1}$,
i.e., $\sigma=n\sigma_{1}$ for a suitable positive integer $n$ or
$\infty$, it follows that $E^{\sigma*}=\{\mathfrak{z}=(Z_{1},Z_{2},\cdots Z_{d})\mid Z_{i}\in M_{n}(\mathbb{C})\}$,
where we interpret $M_{\infty}(\mathbb{C})$ as $B(\ell^{2}(\mathbb{N}))$.
The disc \foreignlanguage{english}{$\mathbb{D}(0,R,\sigma)$ is} $\{\mathfrak{z}=(Z_{1},Z_{2},\cdots Z_{d})\mid\Vert\sum Z_{i}Z_{i}^{*}\Vert^{\frac{1}{2}}<R\}$,
and a $\mathfrak{z}=(Z_{1},Z_{2},\cdots,Z_{d})\in\overline{\mathbb{D}(0,1,\sigma)}$
is in $\mathcal{AC}(\sigma),$ when $\sigma$ is a \emph{finite} multiple
of $\sigma_{1}$, if and only if $\mathfrak{z}$ is completely noncoisometric
\cite[Corollary 5.7]{Muhly2011a} in the sense that $\Vert\mathfrak{z}^{(k)*}h\Vert\to0$
for every $h\in\mathbb{C}^{n}$ \cite[Remark 7.2]{Muhly2004a}. (No
such simple characterization of $\mathcal{AC}(\infty\sigma_{1})$
is known.) For a $\mathfrak{z}\in E^{\sigma*}$ and unitary $u\in\sigma(M)'=M_{n}(\mathbb{C})$,
$u^{-1}\cdot\mathfrak{z}\cdot u=u^{-1}\mathfrak{z}(I_{E}\otimes u)=(u^{-1}Z_{1}u,u^{-1}Z_{2}u,\cdots,u^{-1}Z_{d}u)$.
Thus $\mathfrak{z}$ is central if and only if each $Z_{i}$ is a
scalar multiple of the identity. Further, as we have seen, an additive
subcategory $\Sigma$ of $NRep(M)$ is simply determined by an additive
subsemigroup of $\mathbb{N}\cup\{\infty\}$ and a central additive
field $\{\zeta_{\sigma}\}_{\sigma\in\Sigma}$ is a family of central
elements such that for each $i$, $1\leq i\leq d$, the scalar that
forms the scalar multiple of the identity in the $i^{th}$ element
of $\zeta_{\sigma}$ is independent of $\sigma$. \end{example}
\begin{rem}
\label{rem:Rel_matricial_function} The notion of a matricial family
of functions indexed by $NRep(M)$, defined in Definition \ref{def:Matricial_family_of_functions},
can be ``relativized'' to additive subcategories, too. So, if $\{\mathcal{U}(\sigma)\}_{\sigma\in\Sigma}$
is a matricial $E,\Sigma$-family, where $\Sigma$ is an additive
subcategory of $NRep(M)$, then a family of functions $f=\{f_{\sigma}\}_{\sigma\in\Sigma}$,
with 
\[
f_{\sigma}:\mathcal{U}(\sigma)\to B(H_{\sigma}),\qquad\sigma\in\Sigma,
\]
is an $E,\Sigma$-\emph{matricial family of functions,} or simply\emph{
a matricial family} \emph{of} \emph{functions, }in case $f$ \emph{respects
intertwiners} in the sense that for every $\mathfrak{z}\in\mathcal{U}(\sigma)$,
every $\mathfrak{w}\in\mathcal{U}(\tau)$, $\mathcal{I}(\sigma\times\mathfrak{z},\tau\times\mathfrak{w})\subseteq\mathcal{I}(f_{\sigma}(\mathfrak{z}),f_{\tau}(\mathfrak{w}))$.
(Thanks to our convention established in Remark \ref{Xeta}, $\mathcal{I}(\sigma\times\mathfrak{z},\tau\times\mathfrak{w})\subseteq\mathcal{I}(f_{\sigma}(\mathfrak{z}),f_{\tau}(\mathfrak{w}))$
if and only if every $C\in\mathcal{I}(\sigma,\tau)$ that satisfies
the equation $C\mathfrak{z}=\mathfrak{w}(I_{E}\otimes C)$ also satisfies
the equation $Cf_{\sigma}(\mathfrak{z})=f_{\tau}(\mathfrak{w})C$.)
We shall say that $f$ is continuous, or holomorphic, etc., in case
each $f_{\sigma}$ has the indicated property. \end{rem}
\begin{prop}
\label{prop:Matricial_holo_fcn} Let $\Sigma$ be an additive subcategory
of $NRep(M)$ and let $\mathbb{D}(\zeta,R)$ be the matricial disc
determined by a central additive field $\zeta$ on $\Sigma$, where
$0<R\leq\infty$. If $\theta\in\mathcal{T}_{+}((E))$ has radius of
convergence at least $R$, then the collection $\{f_{\sigma}\}_{\sigma\in\Sigma}$
of tensorial power series determined by $\theta$, $f_{\sigma}(\mathfrak{z}):=\sum_{k\geq0}\mathcal{Z}_{k}(\mathfrak{z}-\zeta_{\sigma})L_{\theta_{k}}$,
forms a matricial family of holomorphic functions on $\mathbb{D}(\zeta,R)$.\end{prop}
\begin{proof}
We saw in Proposition \ref{prop:Power_series_give_holomorphic_fcns}
that each $f_{\sigma}$ is a holomorphic $B(H_{\sigma})$-valued function
defined throughout $\mathbb{D}(\zeta_{\sigma},R,\sigma)$. We therefore
need only check that $\{f_{\sigma}\}_{\sigma\in\Sigma}$ preserves
intertwiners. For this, it suffices to check that for each $k$, the
$k^{th}$ terms of $\{f_{\sigma}\}_{\sigma\in\Sigma}$ preserve intertwiners.
If $C\in\mathcal{I}(\sigma\times\mathfrak{z},\tau\times\mathfrak{w})$,
then by definition $C\in\mathcal{I}(\sigma,\tau)$, so $C\in\mathcal{I}(\mathcal{Z}_{0}(\mathfrak{z}-\zeta_{\sigma})L_{\theta_{0}},\mathcal{Z}_{0}(\mathfrak{w}-\zeta_{\tau})L_{\theta_{0}})$
because $\mathcal{Z}_{0}(\mathfrak{z}-\zeta_{\sigma})$ is identically
$I_{H_{\sigma}}$ in $\mathfrak{z}$, similarly for $\mathcal{Z}_{0}(\mathfrak{w}-\zeta_{\tau}),$
and because when we view $L_{\theta_{0}}$ as a map from $H_{\sigma}$
to $M\otimes_{\sigma(M)}H_{\sigma}$ we identify $\theta_{0}$ with
$\sigma(\theta_{0})$, when we identify $H_{\sigma}$ with $M\otimes_{\sigma(M)}H_{\sigma}$,
as is customary; similarly for $L_{\theta_{0}}$ and $\tau(\theta_{0})$.
To handle $\mathcal{I}(\mathcal{Z}_{1}(\mathfrak{z}-\zeta_{\sigma})L_{\theta_{1}},\mathcal{Z}_{1}(\mathfrak{w}-\zeta_{\tau})L_{\theta_{1}})$,
observe first that $L_{\theta_{1}}C=(I_{E}\otimes C)L_{\theta_{1}}$
whether $L_{\theta_{1}}$ is viewed as a map from $H_{\sigma}$ to
$E\otimes_{\sigma}H_{\sigma}$ or from $H_{\tau}$ to $E\otimes_{\tau}H_{\tau}$.
Further, $C\zeta_{\sigma}=\zeta_{\tau}(I_{E}\otimes C)$ by definition
of a central family and $C\mathfrak{z}=\mathfrak{w}(I_{E}\otimes C)$
by the hypothesis that $C\in\mathcal{I}(\sigma\times\mathfrak{z},\tau\times\mathfrak{w})$.
Thus $C\in\mathcal{I}(\mathcal{Z}_{1}(\mathfrak{z}-\zeta_{\sigma})L_{\theta_{1}},\mathcal{Z}_{1}(\mathfrak{w}-\zeta_{\tau})L_{\theta_{1}})$.
The general case $\mathcal{I}(\mathcal{Z}_{k}(\mathfrak{z}-\zeta_{\sigma})L_{\theta_{k}},\mathcal{Z}_{k}(\mathfrak{w}-\zeta_{\tau})L_{\theta_{k}})$
is handled by noting that it suffices to check the intertwining condition
when $\theta_{k}$ is a decomposable tensor, say $\theta_{k}=\xi_{1}\otimes\xi_{2}\otimes\cdots\otimes\xi_{k}$
and noting that in this case, $\mathcal{Z}_{k}(\mathfrak{z}-\zeta_{\sigma})L_{\theta_{k}}=\mathcal{Z}_{1}(\mathfrak{z}-\zeta_{\sigma})L_{\xi_{1}}\mathcal{Z}_{1}(\mathfrak{z}-\zeta_{\sigma})L_{\xi_{2}}\cdots\mathcal{Z}_{1}(\mathfrak{z}-\zeta_{\sigma})L_{\xi_{k}}$
and similarly $\mathcal{Z}_{k}(\mathfrak{w}-\zeta_{\tau})L_{\theta_{k}}=\mathcal{Z}_{1}(\mathfrak{w}-\zeta_{\tau})L_{\xi_{1}}\mathcal{Z}_{1}(\mathfrak{w}-\zeta_{\tau})L_{\xi_{2}}\cdots\mathcal{Z}_{1}(\mathfrak{w}-\zeta_{\tau})L_{\xi_{k}}$
. For these expressions, it is obvious that $C$ intertwines, by virtue
of the fact that $C\in\mathcal{I}(\mathcal{Z}_{1}(\mathfrak{z}-\zeta_{\sigma})L_{\theta_{1}},\mathcal{Z}_{1}(\mathfrak{w}-\zeta_{\tau})L_{\theta_{1}})$. 
\end{proof}
A concept that is closely related to the notion of a matricial family
of functions is given in 
\begin{defn}
\label{def:_matricial_maps} Suppose $E$ and $F$ are two $W^{*}$-correspondences
over the same $W^{*}$-algebra, $M$ and suppose $\Sigma$ is an additive
subcategory of $NRep(M)$. If $\{\mathcal{U}(\sigma)\}_{\sigma\in\Sigma}$
is an $E,\Sigma$-matricial family of sets and if $\{\mathcal{V}(\sigma)\}_{\sigma\in\Sigma}$
is an $F,\Sigma$-matricial family of sets, then we call a family
of maps $\{f_{\sigma}\}_{\sigma\in\Sigma}$, where $f_{\sigma}$ maps
$\mathcal{U}(\sigma)$ to $\mathcal{V}(\sigma)$, an $E,F,\Sigma$-\emph{matricial
family of maps, }or for short, a\emph{ matricial family of maps, }in
case
\[
Cf_{\sigma}(\mathfrak{z})=f_{\tau}(\mathfrak{w})(I_{F}\otimes C)
\]
for every $C:H_{\sigma}\to H_{\tau}$ in $\mathcal{I}(\sigma,\tau)$
such that $C\mathfrak{z}=\mathfrak{w}(I_{E}\otimes C)$, for all $\mathfrak{z}\in\mathcal{U}(\sigma)$
and all $\mathfrak{w}\in\mathcal{V}(\tau)$. To say the same thing
more succinctly, $\{f_{\sigma}\}_{\sigma\in\Sigma}$ is a matricial
family of maps in case 
\begin{equation}
\mathcal{I}(\sigma\times\mathfrak{z},\tau\times\mathfrak{w})\subseteq\mathcal{I}(\sigma\times f_{\sigma}(\mathfrak{z}),\tau\times f_{\tau}(\mathfrak{w})),\qquad\mathfrak{z}\in\mathcal{U}(\sigma),\,\mathfrak{w}\in\mathcal{U}(\tau).\label{eq:preserv_intertwiners}
\end{equation}

\end{defn}
To help forestall confusion that may develop, we emphasize that we
will be consistent in our distinction between functions and maps:
matricial families of functions map matricial families of sets to
algebras of operators; matricial families of maps are families of
maps between matricial families of sets. As we shall see later, matricial
families of maps are connected to homomorphisms between Hardy algebras.

\section{A special Generator\label{sec:A_special_generator}}

In ring theory, a module $G$ is a \emph{generator} for the category
of left modules over a ring $R$, $\,_{R}\mathfrak{M}$, in case every
$M\in\,_{R}\mathfrak{M}$ is the image of a homomorphism from the
algebraic direct sum of a suitable number of copies of $G$ \cite[Page 193]{AF1992}.
In \cite{R1974a}, Rieffel defines a generator for $NRep(M)$ in a
similar fashion, but allows infinite Hilbert space direct sums. In
\cite[Proposition 1.3]{R1974a}, he proves that a representation $\sigma$
is a generator for $NRep(M)$ in this extended sense if and only if
$\sigma$ is faithful. Here we shall develop a useful analogue of
the notion of a generator for the category of ultraweakly continuous,
completely contractive representations of $H^{\infty}(E)$.
\begin{defn}
\label{def:_Infinite_generator} We shall say that a generator $\pi$
for $NRep(M)$ is an \emph{infinite generator} in case it is an infinite
multiple of a generator for $NRep(M)$, i.e., an infinite multiple
of a faithful normal representation of $M$. We shall say that $\sigma_{0}$
is a \emph{special generator} for $NRep(M)$ if $\sigma_{0}=\pi^{\mathcal{F}(E)}\circ\varphi_{\infty}$
for an infinite generator for $NRep(M)$.\end{defn}
\begin{rem}
\label{rm:_explain_generator} Of course $\sigma_{0}$ and $\pi$
are equivalent in $NRep(M)$ if $\pi$ is an infinite generator. However,
we want to consider additive subcategories of $NRep(M)$ that are
not necessarily closed under forming infinite direct sums. Consequently,
it is important for our considerations to make a distinction between
$\sigma_{0}$ and $\pi$.
\end{rem}
If $\sigma_{0}=\pi^{\mathcal{F}(E)}\circ\varphi_{\infty}$, acting
on $H_{\sigma_{0}}=\mathcal{F}(E)\otimes_{\pi}H_{\pi}$, is a special
generator for $NRep(M)$, and if $\mathfrak{s}_{0}$ is defined by
the formula 
\[
\mathfrak{s}_{0}(\xi\otimes h)=T_{\xi}h,\qquad\xi\in E,\, h\in\mathcal{F}(E)\otimes_{\pi}H_{\pi},
\]

then $\sigma_{0}\times\mathfrak{s}_{0}$ is an induced representation
of $H^{\infty}(E)$ in the sense of \cite{Muhly1999}. In \cite[Proposition 2.3]{Muhly2011a}
we show that $\sigma_{0}\times\mathfrak{s}_{0}$ is unique up to unitary
equivalence in the sense that if $\pi'$ has the same properties as
$\pi$ and if $\sigma_{0}'\otimes\mathfrak{s}_{0}'$ is constructed
from $\pi'$ in a similar fashion to $\sigma_{0}\times\mathfrak{s}_{0}$,
then $\sigma_{0}'\otimes\mathfrak{s}_{0}'$ is unitarily equivalent
to $\sigma_{0}\times\mathfrak{s}_{0}$. Further, if $\sigma\times\mathfrak{z}$
is any induced representation of $H^{\infty}(E)$, then there is a
subspace $\mathcal{K}$ of $H_{\pi}$ that reduces $\pi$ such that
$\sigma\times\mathfrak{z}$ is unitarily equivalent to $\sigma_{0}\times\mathfrak{s}_{0}\vert\mathcal{F}(E)\otimes_{\pi}\mathcal{K}$.
Observe that by construction $\sigma_{0}\times\mathfrak{s}_{0}$ is
absolutely continuous, so $\mathfrak{s}_{0}\in\mathcal{AC}(\sigma_{0})$.
In fact, $\sigma_{0}\times\mathfrak{s}_{0}$ is a generator for the
category of all ultraweakly continuous completely contractive representations
of $H^{\infty}(E)$. This assertion is the content of the following
theorem, which is a summary of Theorems 4.7 and 4.11 of \cite{Muhly2011a}. 
\begin{thm}
\label{abscont} Suppose $\sigma\in NRep(M)$ and $\mathfrak{z}\in\overline{\mathbb{D}(0,1,\sigma)}$.
Then $\mathfrak{z}$ lies in $\mathcal{AC}(\sigma)$ if and only if
$H_{\sigma}$ is the closed linear span of the ranges of the operators
in $\mathcal{I}(\sigma_{0}\times\mathfrak{s}_{0},\sigma\times\mathfrak{z})$.
In this event, the stronger equation holds: 
\[
H_{\sigma}=\bigcup\{Ran(C)\mid C\in\mathcal{I}(\sigma_{0}\times\mathfrak{s}_{0},\sigma\times\mathfrak{z})\}.
\]

\end{thm}
With this theorem at our disposal, we are able to prove the following
theorem that has Theorem \ref{thm:Double_Commutant} as an immediate
corollary.
\begin{thm}
\label{thm:_Relative_double_commutant} Suppose $\Sigma$ is a full
additive subcategory of $NRep(M)$ that contains a special generator
$\sigma_{0}$ for $NRep(M)$. If $f=\{f_{\sigma}\}_{\sigma\in\Sigma}$
is a family of functions such that each $f_{\sigma}$ maps $\mathcal{AC}(\sigma)$
to $B(H_{\sigma})$, then $f$ is a $Berezin$ transform (restricted
to $\{\mathcal{AC}(\sigma)\}_{\sigma\in\Sigma}$) if and only if $f$
is an $E,\Sigma$-matricial family.\end{thm}
\begin{proof}
We already have noted Berezin transforms are matricial families of
functions on $\{\mathcal{AC}(\sigma)\}_{\sigma\in NRep(M)}$. So certainly,
their restrictions to $\{\mathcal{AC}(\sigma)\}_{\sigma\in\Sigma}$
are matricial $E,\Sigma$-families. For the converse, suppose that
$\{f_{\sigma}\}_{\sigma\in\Sigma}$ is a matricial family of functions
defined on $\{\mathcal{AC}(\sigma)\}_{\sigma\in\Sigma}$. Since $\Sigma$
is assumed to be full, for every $\sigma$ and $\tau$ in $\Sigma$,
$\mathcal{I}_{\Sigma}(\sigma,\tau)=\mathcal{I}_{NRep(M)}(\sigma,\tau$)
where the subscripts indicate the category under consideration. It
follows that for every $\mathfrak{z}\in\mathcal{AC}(\sigma)$ and
for every $\mathfrak{w}\in\mathcal{AC}(\tau)$, $\mathcal{I}_{\Sigma}(\sigma\times\mathfrak{z},\tau\times\mathfrak{w})=\mathcal{I}_{NRep(M)}(\sigma\times\mathfrak{z},\tau\times\mathfrak{w})$.
So our hypotheses guarantee that for every $C\in\mathcal{I}_{NRep(M)}(\sigma_{0}\times\mathfrak{s}_{0},\sigma_{0}\times\mathfrak{s}_{0})$,
$Cf_{\sigma_{0}}(\mathfrak{s}_{0})=f_{\sigma_{0}}(\mathfrak{s}_{0})C$.
That is, $f_{\sigma_{0}}(\mathfrak{s}_{0})$ lies in the double commutant
of $\sigma_{0}\times\mathfrak{s}_{0}(H^{\infty}(E))$. However, $\sigma_{0}\times\mathfrak{s}_{0}$
is the restriction of $\pi^{\mathcal{F}(E)}$ to $H^{\infty}(E)$,
where $\sigma_{0}=\pi^{\mathcal{F}(E)}\circ\varphi_{\infty}$, and
$\pi^{\mathcal{F}(E)}(H^{\infty}(E))$ is its own double commutant
by \cite[Corollary 3.10]{Muhly2004a}. Thus there is an $F\in H^{\infty}(E)$
so that 
\begin{equation}
f_{\sigma_{0}}(\mathfrak{s}_{0})=\widehat{F}_{\sigma_{0}}(\mathfrak{s}_{0}).\label{eq:f-F_equality}
\end{equation}
If $\sigma$ is an arbitrary representation in $\Sigma$ and if $\mathfrak{z}\in\mathcal{AC}(\sigma)$,
then for every $C\in\mathcal{I}_{NRep(M)}(\sigma_{0}\times\mathfrak{s}_{0},\sigma\times\mathfrak{z})$,
$f_{\sigma}(\mathfrak{z})C=Cf_{\sigma_{0}}(\mathfrak{s}_{0})$ because
$\Sigma$ is full and $\{f_{\sigma}\}_{\sigma\in\Sigma}$ preserves
intertwiners by hypothesis. However, by \eqref{eq:f-F_equality},
$Cf_{\sigma_{0}}(\mathfrak{s}_{0})=C\widehat{F}_{\sigma_{0}}(\mathfrak{s}_{0})$.
Hence we have 
\[
f_{\sigma}(\mathfrak{z})C=C\widehat{F}_{\sigma_{0}}(\mathfrak{s}_{0})=\widehat{F}_{\sigma}(\mathfrak{z})C,
\]
where the second equality is justified by Remark~\ref{Xeta}. Since
the ranges of the $C$ in $\mathcal{I}_{NRep(M)}(\sigma_{0}\times\mathfrak{s}_{0},\sigma\times\mathfrak{z})$
cover all of $H_{\sigma}$, by Theorem \ref{abscont}, we conclude
that $f_{\sigma}(\mathfrak{z})=\widehat{F}_{\sigma}(\mathfrak{z})$. 
\end{proof}
We digress momentarily to provide an example promised in the introduction
that shows that Theorems \ref{thm:_Relative_double_commutant} and
\ref{thm:Double_Commutant} can fail if the hypothesis that the matricial
function $f=\{f_{\sigma}\}_{\sigma\in\Sigma}$ in question is defined
only on $\{\mathbb{D}(0,1,\sigma)\}_{\sigma\in\Sigma}$ and not on
the collection of larger sets, $\{\mathcal{AC}(\sigma)\}_{\sigma\in\Sigma}$. 
\begin{example}
\label{eg:_Failure_of_resolvents} We let $M=\mathbb{C}=E$. Then
$NRep(M)$ may be identified with $\{n\sigma_{1}\}_{0<n\leq\infty}$,
where, recall, $\sigma_{1}$ is the one-dimensional representation
of $\mathbb{C}$ on $\mathbb{C}$. The disc $\mathbb{D}(0,1,\sigma)$
is just the collection of all operators of norm less than $1$ in
$B(H_{\sigma})$. We set $f_{\sigma}(\mathfrak{z})=(I_{H_{\sigma}}-\mathfrak{z})^{-1}$,
i.e., $f_{\sigma}$ is just the resolvent operator restricted to $\mathbb{D}(0,1,\sigma)$.
Then it is immediate that $f=\{f_{\sigma}\}_{\sigma\in NRep(M)}$
is a matricial family of functions defined on $\{\mathbb{D}(0,1,\sigma)\}_{\sigma\in NRep(M)}$.
However, since none of the $f_{\sigma}$s is bounded, $f$ is not
a Berezin transform of an element in $H^{\infty}(E)\simeq H^{\infty}(\mathbb{T})$.
\end{example}
This example should be compared with Theorem \ref{boundedhol}.

We would like to use Theorem \ref{thm:_Relative_double_commutant}
to obtain information about which matricial families of functions
come from tensorial power series that have a given radius of convergence.
First, however, we take up an issue that was left hanging after Proposition
\ref{prop:Power_series_give_holomorphic_fcns}. 
\begin{prop}
\label{RadiusConv} Suppose $\theta\sim\sum\theta_{k}\in\mathcal{T}_{+}((E))$
has a finite radius of convergence $R=R(\theta)$. If $R'>R$, then
there a $\sigma\in NRep(M)$ and $\mathfrak{z}\in E^{\sigma*}$, with
$\norm{\mathfrak{z}}=R'$, such that the tensorial power series $\sum_{k}\mathcal{Z}_{k}(\cdot)L_{\theta_{k}}$
diverges at $\mathfrak{z}$; indeed, $\sum_{k}\norm{\mathcal{Z}_{k}(\mathfrak{z})L_{\theta_{k}}}=\infty$. \end{prop}
\begin{proof}
Choose $\rho$, with $R<\rho<R'$, let $\sigma$ be $\sigma_{0}$
and set $\mathfrak{z}=R'\mathfrak{s}_{0}$. Since $\frac{1}{\rho}<\frac{1}{R}$,
there are infinitely many $k$s for which $\norm{\theta_{k}}^{1/k}>\frac{1}{\rho}$.
On the other hand, $\mathcal{Z}_{k}(\mathfrak{z})L_{\theta_{k}}=\mathcal{Z}_{k}(R'\mathfrak{s}_{0})L_{\theta_{k}}=R'^{k}(T_{\theta_{k}}\otimes_{\pi}I_{K_{0}})$.
Consequently, $\norm{\mathcal{Z}_{k}(\mathfrak{z})L_{\theta_{k}}}=R'^{k}\norm{\theta_{k}}$.
So for each $k$ satisfying $\norm{\theta_{k}}^{1/k}>\frac{1}{\rho}$,
we have $\norm{\mathcal{Z}_{k}(\mathfrak{z})L_{\theta_{k}}}>(\frac{R'}{\rho})^{k}>1$.
Since there are infinitely many such $k$s, the series $\sum_{k}\norm{\mathcal{Z}_{k}(\mathfrak{z})L_{\theta_{k}}}$
diverges to $\infty$. \end{proof}
\begin{thm}
\label{RD} Suppose $\Sigma$ is a full additive subcategory of $NRep(M)$
containing a special generator for $NRep(M)$. If $f=\{f_{\sigma}\}_{\sigma\in\Sigma}$
is a family of functions, with $f_{\sigma}$ mapping $\mathbb{D}(0,R,\sigma)$
to $B(H_{\sigma})$, then there is a formal tensor series $\theta$
with $R(\theta)\geq R$ such that $f$ is the family of tensorial
power series determined by $\theta$, $\{\sum_{k\geq0}\mathcal{Z}_{k}(\mathfrak{z})L_{\theta_{k}}\mid\mathfrak{z}\in\mathbb{D}(0,R,\sigma)\}_{\sigma\in\Sigma},$
if and only if $f$ is an $E,\Sigma$-matricial family of functions.\end{thm}
\begin{proof}
Arguments we have used before show that if the $f_{\sigma}$s admit
tensorial power series expansions of the indicated kind, then the
family preserves intertwiners. Consequently, we shall concentrate
on the converse assertion. So assume that $\{f_{\sigma}\}_{\sigma\in\Sigma}$
preserves intertwiners. Since $\Sigma$ is assumed to be full, we
may and will drop the subscripts $\Sigma$ and $NRep(M)$ on all intertwining
spaces. Also, we will let $\sigma_{0}=\pi^{\mathcal{F}(E)}\circ\varphi_{\infty}$,
for a suitable infinite generator $\pi$. $ $ The key to our analysis
is to focus on $f_{\sigma_{0}}$ in order to bring properties of $H^{\infty}(E)$
into play and then to use the intertwining property of the family
$\{f_{\sigma}\}_{\sigma\in\Sigma}$ to propagate to the other discs
$\mathbb{D}(0,R,\sigma)$. Fix $0<r<R$ and consider $f_{\sigma_{0}}(r\mathfrak{s}_{0})$.
Recall that $\sigma_{0}\times\mathfrak{s}_{0}(H^{\infty}(E))=\pi^{\mathcal{F}(E)}(H^{\infty}(E))$.
Consequently, every $C$ in the commutant of $\pi^{\cF(E)}(H^{\infty}(E))$
lies in $\mathcal{I}(\sigma_{0}\times\mathfrak{s}_{0},\sigma_{0}\times\mathfrak{s}_{0})=\mathcal{I}(\sigma_{0}\times r\mathfrak{s}_{0},\sigma_{0}\times r\mathfrak{s}_{0})$
and, thus commutes with $f_{\sigma_{0}}(r\mathfrak{s}_{0})$. Since
$\pi^{\cF(E)}(H^{\infty}(E))$ equals its own double commutant, there
is an element $F_{r}\in H^{\infty}(E)$ such that $f_{\sigma_{0}}(r\mathfrak{s}_{0})=\pi^{\cF(E)}(F_{r})=\widehat{F_{r}}_{\sigma_{0}}(\mathfrak{s}_{0})$.
Now take a $\mathfrak{z}\in\mathbb{D}(0,r,\sigma)$. Then $\norm{\frac{1}{r}\mathfrak{z}}<1$
and so $\sigma\times\frac{1}{r}\mathfrak{z}$ is absolutely continuous.
We conclude, by Theorem \ref{abscont}, that 
\begin{equation}
\bigvee\{Ran(C)\mid C\in\mathcal{I}(\sigma_{0}\times\mathfrak{s}_{0},\sigma\times\frac{1}{r}\mathfrak{z})\}=H_{\sigma}.\label{eq:span_Ran(C)}
\end{equation}
Also, for every $C\in\mathcal{I}(\sigma_{0}\times\mathfrak{s}_{0},\sigma\times\frac{1}{r}\mathfrak{z})=\mathcal{I}(\sigma_{0}\times r\mathfrak{s}_{0},\sigma\times\mathfrak{z})$,
$Cf_{\sigma_{0}}(r\mathfrak{s}_{0})=f_{\sigma}(\mathfrak{z})C$ by
hypothesis. Since $f_{\sigma_{0}}(r\mathfrak{s}_{0})=\widehat{F_{r}}_{\sigma_{0}}(\mathfrak{s}_{0})$,
we see that $f_{\sigma}(\mathfrak{z})C=C\widehat{F_{r}}_{\sigma_{0}}(\mathfrak{s}_{0})$.
But by Remark~\ref{Xeta}, $C\widehat{F_{r}}_{\sigma_{0}}(\mathfrak{s}_{0})=\widehat{F_{r}}_{\sigma}(\frac{1}{r}\mathfrak{z})C$,
so we conclude from \eqref{eq:span_Ran(C)} that ,
\begin{equation}
f_{\sigma}(\mathfrak{z})=\widehat{F_{r}}_{\sigma}(\frac{1}{r}\mathfrak{z}),\qquad\norm{\mathfrak{z}}<r<R.\label{fsigma}
\end{equation}
We need to remove the dependence of $F_{r}$ on $r$. So if $0<r<r_{1}<r_{2}<R$
and if $\norm{\mathfrak{z}}\leq r$, we obtain the equation $\widehat{F_{r_{1}}}_{\sigma}(\frac{1}{r_{1}}\mathfrak{z})=\widehat{F_{r_{2}}}_{\sigma}(\frac{1}{r_{2}}\mathfrak{z})$.
In particular, 
\begin{equation}
\widehat{F_{r_{1}}}_{\sigma_{0}}(\frac{r}{r_{1}}\mathfrak{s}_{0})=\widehat{F_{r_{2}}}_{\sigma_{0}}(\frac{r}{r_{2}}\mathfrak{s}_{0}).\label{eq1}
\end{equation}
We now would like to apply the ``Fourier coefficient maps'' $\Phi_{k}$
to equation (\ref{eq1}). To give this its proper meaning, note that,
whenever $X\in H^{\infty}(E)$ and $0<t<1$, 
\begin{equation}
\widehat{\Phi_{k}(X)}_{\sigma_{0}}(t\mathfrak{s}_{0})=t^{k}\Phi_{k}(X)\otimes I_{H_{\pi}}.\label{eq:Phi_k_ts_0}
\end{equation}
This is easy to verify by first taking $X=T_{\xi}$ for $\xi\in E^{\otimes m}$,
and then using linearity and ultra-weak continuity. Thus, applying
$\Phi_{k}$ to equation (\ref{eq1}), we obtain 
\begin{equation}
\frac{r^{k}}{r_{1}^{k}}\Phi_{k}(F_{r_{1}})\otimes I_{H_{\pi}}=\frac{r^{k}}{r_{2}^{k}}\Phi_{k}(F_{r_{2}})\otimes I_{H_{\pi}}.\label{eq2}
\end{equation}
 for all $k\geq0$, which implies that $r\mapsto\frac{1}{r^{k}}\Phi_{k}(F_{r})$
is constant in $r,$ $0<r<R$, for every $k$. Consequently, since
the image of $\Phi_{k}$ is $\{T_{\theta_{k}}\mid\theta_{k}\in E^{\otimes k}\}$,
there is a $\theta_{k}\in E^{\otimes k}$, independent of $r$, so
that 
\[
T_{\theta_{k}}\otimes I_{H_{\pi}}=\frac{1}{r^{k}}\Phi_{k}(F_{r})\otimes I_{H_{\pi}},\qquad0<r<R.
\]
Canceling ``$\otimes I_{H_{\pi}}$'', as we may, we conclude that
\begin{equation}
\Phi_{k}(F_{r})=r^{k}T_{\theta_{k}}=T_{r^{k}\theta_{k}},\qquad0<r<R.\label{eq:thetak}
\end{equation}
Now fix $0<r<R$ and $\mathfrak{z}\in E^{\sigma*}$ with $\norm{\mathfrak{z}}<r$.
For $0\leq k$, let $\xi_{k}=r^{k}\theta_{k}\in E^{\otimes k}$. Then
with $F_{r}$ in place of $F$ and $\frac{1}{r}\mathfrak{z}$ in place
of $\mathfrak{z}$, we find that 
\[
\widehat{F_{r}}_{\sigma}(\frac{1}{r}\mathfrak{z})=\sum_{k\geq0}\frac{1}{r^{k}}\mathcal{Z}_{k}(\mathfrak{z})L_{\xi_{k}}
\]
and $\sum\frac{1}{r^{k}}\norm{\mathcal{Z}_{k}(\mathfrak{z})L_{\xi_{k}}}<\infty$.
By (\ref{fsigma}), we conclude that $f_{\sigma}(\mathfrak{z})=\widehat{F_{r}}_{\sigma}(\frac{1}{r}\mathfrak{z})=\sum_{k\geq0}\frac{1}{r^{k}}\mathcal{Z}_{k}(\mathfrak{z})L_{\xi_{k}}=\sum_{k\geq0}\mathcal{Z}_{k}(\mathfrak{z})L_{\theta_{k}}$.
Thus $\sum\norm{\mathcal{Z}_{k}(\mathfrak{z})L_{\theta_{k}}}<\infty$
and $f_{\sigma}(\mathfrak{z})=\sum_{k\geq0}\mathcal{Z}_{k}(\mathfrak{z})L_{\theta_{k}}$.
By Theorem~\ref{RadiusConv}, $R(\theta)\geq R$.

It remains to show that the series $\theta\sim\sum\theta_{k}$ is
uniquely determined by $\{f_{\sigma}\}_{\sigma\in\Sigma}$. In fact,
it is uniquely determined by $f_{\sigma_{0}}$. Suppose $\theta'\sim\sum\theta_{k}'$
is another series with $R(\theta')\geq R$ and suppose 
\[
\sum_{k\geq0}\mathcal{Z}_{k}(\mathfrak{z})L_{\theta_{k}}=f_{\sigma_{0}}(\mathfrak{z})=\sum_{k\geq0}\mathcal{Z}_{k}(\mathfrak{z})L_{\theta'_{k}}
\]
for all $\mathfrak{z}\in\mathbb{D}(0,R,\sigma_{0})$. However, as
we have seen from our analysis that yielded \eqref{eq:Phi_k_ts_0},
$\mathcal{Z}_{k}(t\mathfrak{s}_{0})L_{\theta_{k}}=t^{k}T_{\theta_{k}}\otimes I_{H_{\pi}}$.
Thus we conclude that for $0\leq t<\min\{1,R(\theta)\}$, 
\begin{gather*}
\sum_{k\geq0}t^{k}T_{\theta_{k}}\otimes I_{H_{\pi}}=\sum_{k\geq0}\mathcal{Z}_{k}(t\mathfrak{s}_{0})L_{\theta_{k}}=\sum_{k\geq0}\mathcal{Z}_{k}(t\mathfrak{s}_{0})L_{\theta'_{k}}=\sum_{k\geq0}t^{k}T_{\theta'_{k}}\otimes I_{H_{\pi}},
\end{gather*}
where all the series converge in the operator norm on $B(\mathcal{F}(E)\otimes_{\pi}H_{\pi})$.
Since the map $X\mapsto\pi^{\mathcal{F}(E)}(X)$, $X\in\mathcal{L}(\mathcal{F}(E))$,
is a faithful normal representation of $\mathcal{L}(\mathcal{F}(E))$,
\[
\sum_{k\geq0}t^{k}T_{\theta_{k}}=\sum_{k\geq0}t^{k}T_{\theta'_{k}}
\]
as norm-convergent series in $H^{\infty}(E)$. So, if we apply $\Phi_{k}$
to both sides, we conclude that $t^{k}T_{\theta_{k}}=t^{k}T_{\theta'_{k}}$
for every $k$. Hence $\theta_{k}=\theta'_{k}$ for every $k$ and
$\theta=\theta'$. 
\end{proof}
As we noted in Remark \ref{Rem:Arveson's_discovery}, it can happen
that a series $\sum_{k\geq0}\mathcal{Z}_{k}(\mathfrak{z})L_{\theta_{k}}$
represents a bounded holomorphic function $f_{\sigma}$ on $\mathbb{D}(0,1,\sigma)$
for a particular $\sigma$ and $R(\theta)\geq1$, $\theta\sim\sum\theta_{k}$,
but $f_{\sigma}$ is not a $\sigma$-Berezin transform. However, the
following proposition shows that if $f_{\sigma}$ is a member of a
matricial family of functions $f=\{f_{\sigma}\}_{\sigma\in\Sigma}$
that are uniformly bounded in $\sigma$, then $f$ is a Berezin transform. 
\begin{thm}
\label{boundedhol} Suppose $\Sigma$ is an additive full subcategory
of $NRep(M)$ that contains a special generator for $NRep(M)$ and
suppose $f=\{f_{\sigma}\}_{\sigma\in\Sigma}$ is an $E,\Sigma$-matricial
family of functions, with $f_{\sigma}$ defined on $\mathbb{D}(0,1,\sigma)$
and mapping to $B(H_{\sigma})$. Then $f$ is a Berezin transform
restricted to $\{\mathbb{D}(0,1,\sigma)\}_{\sigma\in\Sigma}$ if and
only if 
\begin{equation}
\sup\{\norm{f_{\sigma}(\mathfrak{z})}\mid\sigma\in\Sigma,\;\mathfrak{z}\in\mathbb{D}(0,1,\sigma)\}<\infty.\label{eqboundedhol}
\end{equation}
 \end{thm}
\begin{proof}
If there is an $F\in H^{\infty}(E)$ such that $f_{\sigma}=\widehat{F}_{\sigma}$
for all $\sigma$, then certainly $\norm{f_{\sigma}(\mathfrak{z})}\leq\norm{F}$
for all $\sigma$ and $\mathfrak{z}$. So we shall attend to the converse
and suppose $\sup\{\norm{f_{\sigma}(\mathfrak{z})}\mid\sigma\in NRep_{O}(M),\;\mathfrak{z}\in\mathbb{D}(0,1,\sigma)\}=A<\infty$.
If $0<r<1$, then as we saw in the proof of Theorem~\ref{RD}, there
is an $F_{r}\in H^{\infty}(E)$ such that $f_{\sigma_{0}}(r\mathfrak{s}_{0})=\pi^{\cF(E)}(F_{r})$
and $f_{\sigma}(\mathfrak{z})=\widehat{F_{r}}_{\sigma}(\frac{1}{r}\mathfrak{z})$
for all $\norm{\mathfrak{z}}<r$ and all $\sigma$ (see (\ref{fsigma})).

Also, it follows from Theorem~\ref{RD} that $f$ is a tensorial
power series $\{\sum_{k\geq0}\mathcal{Z}_{k}(\cdot)L_{\theta_{k}}\mid\mathfrak{z}\in\mathbb{D}(0,1,\sigma)\}$
where the series $\theta\sim\sum_{k\geq0}\theta_{k}$ in $\mathcal{T}_{+}((E))$
has $R(\theta)\geq1$. For $0<t<r<1$, we thus have 
\[
\widehat{F_{r}}_{\sigma_{0}}(\frac{t}{r}\mathfrak{s}_{0})=f_{\sigma_{0}}(t\mathfrak{s}_{0})=\sum_{k\geq0}\mathcal{Z}_{k}(t\mathfrak{s}_{0})L_{\theta_{k}}=\sum_{k\geq0}t^{k}T_{\theta_{k}}\otimes I_{H_{\pi}}
\]
 and it follows that for every integer $m\geq0$, 
\[
\Phi_{m}(\widehat{F_{r}}_{\sigma_{0}}(\frac{t}{r}\mathfrak{s}_{0}))=t^{m}T_{\theta_{k}}\otimes I_{H_{\pi}}.
\]
 Therefore 
\[
\Phi_{m}(\widehat{F_{r}}_{\sigma_{0}}(\mathfrak{s}_{0}))=r^{m}T_{\theta_{k}}\otimes I_{H_{\pi}}
\]
 for every $0<r<1$.

Note that, for every $0<r<1$, $||F_{r}||=||\pi^{\cF(E)}(F_{r})||=||f_{\sigma_{0}}(r\mathfrak{s}_{0})||\leq A$.
Thus $\{F_{r}\}$ is a bounded set.

If $r_{n}\nearrow1$ and if $F$ is an ultraweak limit point of $\{F_{r_{n}}\}$,
say $F=\lim_{\alpha}F_{r_{n_{\alpha}}}$ for an appropriate subnet
of $\{r_{n}\}$, then for every $m\geq0$ we have $\Phi_{m}(\pi^{\cF(E)}(F))=\lim\Phi_{m}(\pi^{\cF(E)}(F_{r_{n_{\alpha}}}))=\lim\Phi_{m}(\widehat{F_{r_{n_{\alpha}}}}_{\sigma_{0}}(\mathfrak{s}_{0}))=\lim r_{n_{\alpha}}^{m}T_{\xi_{m}}\otimes_{\pi}I_{H_{\pi}}=T_{\xi_{m}}\otimes_{\pi}I_{H_{\pi}}$.
It follows that $\Phi_{m}(F)=T_{\theta_{m}}$ and, using Proposition~\ref{prop:H_infty_values},
we have, for every $\sigma\in\Sigma$ and every $\mathfrak{z}\in\mathbb{D}(0,1,\sigma)$,
\[
f_{\sigma}(\mathfrak{z})=\sum_{k\geq0}\mathcal{Z}_{k}(\mathfrak{z})L_{\theta_{k}}=\widehat{F}_{\sigma}(\mathfrak{z}).
\]
 Thus $f_{\sigma}=\widehat{F}_{\sigma}$.
\end{proof}
In the following theorem we want to consider two $W^{*}$-correspondences
over $M$, $E$ and $F$, and relate \emph{maps} between their families
of absolutely continuous representations to homomorphisms between
their Hardy algebras. 
\begin{thm}
\label{UEtoUF} Suppose $\Sigma$ is a full additive subcategory of
$NRep(M)$ that contains a special generator for $NRep(M)$. Suppose
also that $E$ and $F$ are two $W^{*}$-correspondences over $M$
and that $f=\{f_{\sigma}\}_{\sigma\in\Sigma}$ is a family of maps,
with $f_{\sigma}:\mathcal{AC}(\sigma,E)\to\mathcal{AC}(\sigma,F)$.
Then $f$ is an $E,F,\Sigma$-matricial family of maps if and only
if there is an ultraweakly continuous homomorphism $\alpha:H^{\infty}(F)\to H^{\infty}(E)$
such that for every $\mathfrak{z}\in\mathcal{AC}(\sigma,E)$ and every
$Y\in H^{\infty}(F)$, 
\begin{equation}
\widehat{\alpha(Y)}(\mathfrak{z})=\widehat{Y}(f_{\sigma}(\mathfrak{z})).\label{eq:thmUEtoUF}
\end{equation}
 \end{thm}
\begin{proof}
Suppose that $f=\{f_{\sigma}\}_{\sigma\in\Sigma}$ is a matricial
family of maps and let $Y$ be an element of $H^{\infty}(F)$. For
each $\sigma\in\Sigma$, define $g_{\sigma}:\mathcal{AC}(\sigma,E)\to B(H_{\sigma})$
by $g_{\sigma}(\mathfrak{z})=\widehat{Y}(f_{\sigma}(\mathfrak{z}))$.
If $\mathfrak{z}\in\mathcal{AC}(\sigma,E$), if $\mathfrak{w}\in\mathcal{AC}(\tau,E)$,
and if $C\in\mathcal{I}(\sigma\times\mathfrak{z},\tau\times\mathfrak{w})$,
then since $f$ is assumed to be a matricial family of maps, $C\in\mathcal{I}(\sigma\times f_{\sigma}(\mathfrak{z}),\tau\times f_{\tau}(\mathfrak{w}))$.
By Theorem~\ref{thm:_Relative_double_commutant}, $Cg_{\sigma}(\mathfrak{z})=g_{\tau}(\mathfrak{w})C$.
Thus, by Theorem~\ref{thm:_Relative_double_commutant} again, there
is an operator $\alpha(Y)$ in $H^{\infty}(E)$ such that $\widehat{\alpha(Y)}(\mathfrak{z})=g_{\sigma}(\mathfrak{z})$
for all $\mathfrak{z}\in\mathcal{AC}(\sigma,E)$. Thus equation \ref{eq:thmUEtoUF}
is satisfied. Note that $\alpha(Y)$ is uniquely determined by virtue
of the uniqueness assertion in Theorem \ref{RD}. Since 
\begin{equation}
\widehat{\alpha(Y)}(\mathfrak{z})=g_{\sigma}(\mathfrak{z})=\widehat{Y}(f_{\sigma}(\mathfrak{z})),\label{eq:def_alpha}
\end{equation}
it is clear that $\alpha$ is a homomorphism. It remains to prove
that $\alpha$ is ultraweakly continuous. For that purpose, let $\sigma$
be the special generator $\sigma_{0}$ and let $\mathfrak{z}=\mathfrak{s}_{0}$
in \eqref{eq:def_alpha}, to conclude that 
\[
\alpha(Y)\otimes I_{H_{\pi}}=\widehat{Y}(f_{\sigma_{0}}(\mathfrak{s}_{0})),
\]
 from which it follows immediately that $\alpha$ is ultraweakly continuous.

For the converse, suppose the family $f$ implements a homomorphism
$\alpha$ via equation \eqref{eq:def_alpha}. To show that $f$ is
an $E,F,\Sigma$-matricial family of maps, we must show that the family
preserves intertwiners in the sense of equation \eqref{eq:preserv_intertwiners}.
So let $C\in\mathcal{I}(\sigma\times\mathfrak{z},\tau\times\mathfrak{w})$
and apply Theorem~\ref{thm:_Relative_double_commutant} to conclude
that $C\widehat{Y}(f_{\sigma}(\mathfrak{z}))=C\widehat{\alpha(Y)}(\mathfrak{z})=\widehat{\alpha(Y)}(\mathfrak{w})C=\widehat{Y}(f_{\tau}(\mathfrak{w}))$
for all $Y\in H^{\infty}(F)$. In particular, this holds for $Y=T_{\xi}$
for every $\xi\in F$. So, if we write $L_{\xi}^{\sigma}$ for the
map from $H_{\sigma}\to F\otimes_{\sigma}H_{\sigma}$ defined by the
equation $L_{\xi}^{\sigma}h:=\xi\otimes h$ and if we define $L_{\xi}^{\tau}$
similarly, then for $Y=T_{\xi}$ we conclude that $\widehat{Y}(f_{\sigma}(\mathfrak{z}))=f_{\sigma}(\mathfrak{z})L_{\xi}^{\sigma}$
and $\widehat{Y}(f_{\tau}(\mathfrak{w}))=f_{\tau}(\mathfrak{w})L_{\xi}^{\tau}$.
Thus $Cf_{\sigma}(\mathfrak{z})L_{\xi}^{\sigma}=f_{\tau}(\mathfrak{w})L_{\xi}^{\tau}C=f_{\tau}(\mathfrak{w})(I_{E}\otimes C)L_{\xi}^{\sigma}$.
Since this equation holds for all $\xi\in F$, we conclude that $Cf_{\sigma}(\mathfrak{z})=f_{\tau}(\mathfrak{w})(I_{E}\otimes C)$.
Thus $C\in\mathcal{I}(f_{\sigma}(\mathfrak{z}),f_{\tau}(\mathfrak{w}))$. \end{proof}
\begin{rem}
\label{Rm:Center-preservation} Recall that an element $\mathfrak{z}\in E^{\sigma*}$
lies in the \emph{center of} $E^{\sigma*}$ in case $b\cdot\mathfrak{z}=\mathfrak{z}\cdot b$
for all $b\in\sigma(M)'$ \cite[Definition 4.11]{Muhly2008b}, in
which case we write $\mathfrak{z}\in\mathfrak{Z}(E^{\sigma*})$. Thus
$\mathfrak{z}\in\mathfrak{Z}(E^{\sigma*})$ if and only if $\sigma(M)'\subseteq\mathcal{I}(\sigma\times\mathfrak{z},\sigma\times\mathfrak{z})$.
It follows that if $\{f_{\sigma}\}_{\sigma\in NRep(M)}$ is a matricial
family of maps, with $f_{\sigma}:\mathcal{AC}(\sigma,E)\to\mathcal{AC}(\sigma,F)$,
then $f_{\sigma}(\mathcal{AC}(\sigma,E)\cap\mathfrak{Z}(E^{\sigma*}))\subseteq\mathcal{AC}(\sigma,F)\cap\mathfrak{Z}(F^{\sigma*})$. 
\end{rem}

\section{Function Theory without a generator\label{Function Theory without a generator}}

In this section we shall focus on additive subcategories $\Sigma$
of $NRep(M)$ that do not necessarily contain a special generator
and address the problem of deciding which matricial families of functions
$\{f_{\sigma}\}_{\sigma\in\Sigma}$ have tensorial power series representations
as in Proposition \ref{prop:Power_series_give_holomorphic_fcns}.
In particular, we shall prove the following theorem which complements
and is something of a converse to Proposition \ref{prop:Matricial_holo_fcn}.
\begin{thm}
\label{thm:Matricial_implies_Frechet_hol.} Suppose $\Sigma$ is an
additive subcategory of $NRep(M)$ and that $\mathbb{D}(\zeta,R)$
is a matricial disc determined by an additive field $\zeta=\{\zeta_{\sigma}\}_{\sigma\in\Sigma}$
on $\Sigma$ and $R$, $0<R\leq\infty$. Suppose, also, that $f=\{f_{\sigma}\}_{\sigma\in\Sigma}$
is a matricial family of functions defined on $\mathbb{D}(\zeta,R)$
that is locally uniformly bounded on $\mathbb{D}(\zeta,R)$ in the
sense that for each $r<R$, 
\[
\sup_{\sigma\in\Sigma}\sup_{\mathfrak{z}\in\mathbb{D}(\zeta_{\sigma},r)}\Vert f_{\sigma}(\mathfrak{z})\Vert<\infty.
\]
Then: 
\begin{enumerate}
\item Each $f_{\sigma}$ is Frechet analytic on $\mathbb{D}(\zeta_{\sigma},R)$. 
\item If the subcategory is full, if the additive field $\zeta$ is also
central, and if each $\sigma\in\Sigma$ is faithful, then $f$ is
a family of tensorial power series $\{\sum_{k\geq0}\mbox{\ensuremath{\mathcal{Z}}}_{k}(\mathfrak{z}-\zeta_{\sigma})L_{\theta_{k}}\mid\mathfrak{z}\in\mathbb{D}(\zeta_{k},R,\sigma)\}_{\sigma\in\Sigma}$,
where $\theta\sim\sum_{k\geq0}\theta_{k}$ has $R(\theta)\geq R$.
\end{enumerate}
\end{thm}
To achieve this goal, we use the matrix analysis initiated by Taylor
in \cite{Tay72c}, and developed further in the work of Voiculescu
in \cite{Voi2004,Voi2010}, in the work of Helton and his collaborators
\cite{HKMS2009,HKM2011b,HKM2011a} and especially in the investigations
of Kaliuzhny\u{i}-Verbovetsky\u{i} and Vinnikov \cite{K-VV2009,Kaliuzhnyi-Verbovetskyi2010,K-VV2006}.
Indeed, the first half of Theorem \ref{thm:Matricial_implies_Frechet_hol.}
is proved by showing that Taylor's matrix analysis leads to one of
Taylor's Taylor series for each $f_{\sigma}$. The existence of such
a series implies that $f_{\sigma}$ is Frechet analytic. The second
half of the theorem involves showing that each of Taylor's Taylor
coefficients comes from an element of a suitable tensor power of $E^{\otimes k}$. 

Throughout this section, $\Sigma$ will denote a fixed additive subcategory
of $NRep(M)$. No other assumptions will be placed on $\Sigma$, except
as explicitly stated in added hypotheses in the statements of results.
In particular, we do not assume that $\Sigma$ is full, nor do we
assume that $\Sigma$ contains a special generator. The matricial
sets we will consider will primarily be matricial discs $\mathbb{D}(\zeta,R)=\{\mathbb{D}(\zeta_{\sigma},R,\sigma)\}_{\sigma\in\Sigma}$,
where $\zeta$ is an additive field of vectors on $\Sigma$ (see Example
\ref{eg:Examples_matricial_E,sigma_sets}). Note, in particular, that
the assumption that $\zeta$ is an additive field on $\Sigma$ guarantees
that $\zeta_{m\sigma}=\zeta_{\sigma}\oplus\zeta_{\sigma}\oplus\cdots\oplus\zeta_{\sigma}$
($m$-summands). 

The following result can be found in \cite{K-V2012} for the case
when $M=\mathbb{C}$.
\begin{lem}
\label{Definition_Delta} Let $f=\{f_{\sigma}\}_{\sigma\in\Sigma}$
be a matricial family of functions defined on a matricial $E,\Sigma$-family
$\{\mathcal{U}(\sigma)\}_{\sigma\in\Sigma}$ where $\Sigma$ is an
additive subcategory of $NRep(M)$. Suppose $\sigma,\tau\in\Sigma$,
$\mathfrak{z}\in\mathcal{U}(\sigma)$, $\mathfrak{w}\in\mathcal{U}(\sigma)$
and $\mathfrak{u}\in\mathcal{I}(\tau^{E}\circ\varphi,\sigma)$ are
such that $\left(\begin{array}{cc}
\mathfrak{z} & \mathfrak{u}\\
0 & \mathfrak{w}
\end{array}\right)\in\mathcal{U}(\sigma\oplus\tau)$. Then there is an operator $\Delta f_{\sigma,\tau}(\mathfrak{z},\mathfrak{w})(\mathfrak{u})\in B(H_{\tau},H_{\sigma})$
such that 
\begin{enumerate}
\item {}
\[
f_{\sigma\oplus\tau}(\left(\begin{array}{cc}
\mathfrak{z} & \mathfrak{u}\\
0 & \mathfrak{w}
\end{array}\right))=\left(\begin{array}{cc}
f_{\sigma}(\mathfrak{z}) & \Delta f_{\sigma,\tau}(\mathfrak{z},\mathfrak{w})(\mathfrak{u})\\
0 & f_{\tau}(\mathfrak{w})
\end{array}\right).
\]
 
\item If $b\in\sigma(M)'$ is such that $\mathfrak{z}(I_{E}\otimes b)=b\mathfrak{z}$
and such that $\left(\begin{array}{cc}
\mathfrak{z} & b\mathfrak{u}\\
0 & \mathfrak{w}
\end{array}\right)\in\mathcal{U}(\sigma\oplus\tau)$, then $f_{\sigma}(\mathfrak{z})b=bf_{\sigma}(\mathfrak{z})$ and
$\Delta f_{\sigma,\tau}(\mathfrak{z},\mathfrak{w})(b\mathfrak{u})=b\Delta f_{\sigma,\tau}(\mathfrak{z},\mathfrak{w})(\mathfrak{u})$.
In particular $\Delta f_{\sigma,\tau}(\mathfrak{z},\mathfrak{w})(t\mathfrak{u})=t\Delta f_{\sigma,\tau}(\mathfrak{z},\mathfrak{w})(\mathfrak{u})$
for $t\in\mathbb{C}$ such that 
\[
\left(\begin{array}{cc}
\mathfrak{z} & t\mathfrak{u}\\
0 & \mathfrak{w}
\end{array}\right)\in\mathcal{U}(\sigma\oplus\tau).
\]
 
\item If $b\in\tau(M)'$ is such that $\mathfrak{w}(I_{E}\otimes b)=b\mathfrak{w}$
and such that 
\[
\left(\begin{array}{cc}
\mathfrak{z} & \mathfrak{u}(I_{E}\otimes b)\\
0 & \mathfrak{w}
\end{array}\right)\in\mathcal{U}(\sigma\oplus\tau),
\]
 then $f_{\tau}(\mathfrak{w})b=bf_{\tau}(\mathfrak{w})$ and $\Delta f_{\sigma,\tau}(\mathfrak{z},\mathfrak{w})(\mathfrak{u}(I_{E}\otimes b))=\Delta f_{\sigma,\tau}(\mathfrak{z},\mathfrak{w})(\mathfrak{u})b$.
\item If $\sigma=\tau$, $\mathfrak{z}\in\mathcal{U}(\sigma)$ and $b\in\sigma(M)'$
such that $\left(\begin{array}{cc}
\mathfrak{z} & b\mathfrak{u}\\
0 & \mathfrak{z}+b\mathfrak{u}
\end{array}\right)\in\mathcal{U}(\sigma\oplus\sigma)$, then 
\[
f_{\sigma}(\mathfrak{z}+b\mathfrak{u})-f_{\sigma}(\mathfrak{z})=\Delta f_{\sigma,\sigma}(\mathfrak{z},\mathfrak{z}+b\mathfrak{u})(b\mathfrak{u})
\]
 and, if $t\in\mathbb{C}$, 
\[
f_{\sigma}(\mathfrak{z}+t\mathfrak{u})-f_{\sigma}(\mathfrak{z})=t\Delta f_{\sigma,\sigma}(\mathfrak{z},\mathfrak{z}+t\mathfrak{u})(\mathfrak{u}).
\]
 
\end{enumerate}
\end{lem}
\begin{rem}
\label{Rem:_Key1} It should be emphasized that in $(1)$, $\Delta f_{\sigma,\tau}(\mathfrak{z},\mathfrak{w})(\mathfrak{u})$
represents an operator in $B(H_{\tau},H_{\sigma})$, whose full dependence
on $\mathfrak{z}$, $\mathfrak{w}$ and $\mathfrak{u}$ has still
to be determined. Among other things, part (2) of the lemma proves
that $\Delta f_{\sigma,\tau}(\mathfrak{z},\mathfrak{w})(\mathfrak{u})$
extends to be homogeneous in $\mathfrak{u}$ of degree one, but the
\emph{additivity} in $\mathfrak{u}$ will be proved later in Lemma~\ref{additivity}.
It should also be emphasized that to prove that $\Delta f_{\sigma,\tau}(\mathfrak{z},\mathfrak{w})(\mathfrak{u})$
extends to be homogeneous in $\mathfrak{u}$ the full force of the
assumption that $f$ is matricial is not used: We may assume that
$\sigma=\tau$ and that the intertwiners involved are operator matrices
whose entries are scalar multiples of the identity (on $H_{\sigma}$).
These are present in any of the categories we use, as was stated at
the outset of Section \ref{sec:Preliminaries}.\end{rem}
\begin{proof}
Write $I_{\sigma}$ for the identity operator on $H_{\sigma}$. Then
\begin{equation}
\left(\begin{array}{cc}
\mathfrak{z} & \mathfrak{u}\\
0 & \mathfrak{w}
\end{array}\right)\left(\begin{array}{c}
I_{E}\otimes I_{\sigma}\\
0
\end{array}\right)=\left(\begin{array}{c}
I_{\sigma}\\
0
\end{array}\right)\mathfrak{z}\label{eq:Intw_1}
\end{equation}
 and 
\begin{equation}
\left(\begin{array}{cc}
0 & I_{\tau}\end{array}\right)\left(\begin{array}{cc}
\mathfrak{z} & \mathfrak{u}\\
0 & \mathfrak{w}
\end{array}\right)=\mathfrak{w}\left(\begin{array}{cc}
0 & I_{E}\otimes I_{\tau}\end{array}\right).\label{eq:Intw_2}
\end{equation}
 Since $f$ preserves intertwiners, we may write 
\begin{equation}
f_{\sigma\oplus\tau}(\left(\begin{array}{cc}
\mathfrak{z} & \mathfrak{u}\\
0 & \mathfrak{w}
\end{array}\right))\left(\begin{array}{c}
I_{\sigma}\\
0
\end{array}\right)=\left(\begin{array}{c}
I_{\sigma}\\
0
\end{array}\right)f_{\sigma}(\mathfrak{z})\label{eq:Intw_1_bis}
\end{equation}
 and 
\begin{equation}
\left(\begin{array}{cc}
0 & I_{\tau}\end{array}\right)f_{\sigma\oplus\tau}(\left(\begin{array}{cc}
\mathfrak{z} & \mathfrak{u}\\
0 & \mathfrak{w}
\end{array}\right))=f_{\tau}(\mathfrak{w})\left(\begin{array}{cc}
0 & I_{\tau}\end{array}\right).\label{eq:Intw_2_bis}
\end{equation}
 So, if we write $f_{\sigma\oplus\tau}(\begin{bmatrix}\mathfrak{z} & \mathfrak{u}\\
0 & \mathfrak{w}
\end{bmatrix})$ as $\begin{bmatrix}a_{11} & a_{12}\\
a_{21} & a_{22}
\end{bmatrix},$ then \eqref{eq:Intw_1_bis} shows that $a_{11}=f_{\sigma}(\mathfrak{z})$
and $a_{21}=0$, while \eqref{eq:Intw_2_bis} also shows that $a_{21}=0$
as well as $a_{22}=f_{\tau}(\mathfrak{w})$. The remaining entry,
$a_{12}$, is taken as the definition of $\Delta f_{\sigma,\tau}(\mathfrak{z},\mathfrak{w})(\mathfrak{u})$.
This proves (1).

For (2) note that the equality 
\[
\left(\begin{array}{cc}
b & 0\\
0 & I_{\tau}
\end{array}\right)\left(\begin{array}{cc}
\mathfrak{z} & \mathfrak{u}\\
0 & \mathfrak{w}
\end{array}\right)=\left(\begin{array}{cc}
\mathfrak{z} & b\mathfrak{u}\\
0 & \mathfrak{w}
\end{array}\right)\left(\begin{array}{cc}
I_{E}\otimes b & 0\\
0 & I_{E}\otimes I_{\tau}
\end{array}\right)
\]
 implies 
\begin{multline*}
\left(\begin{array}{cc}
b & 0\\
0 & I_{\tau}
\end{array}\right)\left(\begin{array}{cc}
f_{\sigma}(\mathfrak{z}) & \Delta f_{\sigma,\tau}(\mathfrak{z},\mathfrak{w})(\mathfrak{u})\\
0 & f_{\tau}(\mathfrak{w})
\end{array}\right)\\
=\left(\begin{array}{cc}
f_{\sigma}(\mathfrak{z}) & \Delta f_{\sigma,\tau}(\mathfrak{z},\mathfrak{w})(b\mathfrak{u})\\
0 & f_{\tau}(\mathfrak{w})
\end{array}\right)\left(\begin{array}{cc}
b & 0\\
0 & I_{\tau}
\end{array}\right),
\end{multline*}
 proving (2).

The proof of (3) is similar using the equality 
\[
\left(\begin{array}{cc}
I_{\sigma} & 0\\
0 & b
\end{array}\right)\left(\begin{array}{cc}
\mathfrak{z} & \mathfrak{u}(I_{E}\otimes b)\\
0 & \mathfrak{w}
\end{array}\right)=\left(\begin{array}{cc}
\mathfrak{z} & \mathfrak{u}\\
0 & \mathfrak{w}
\end{array}\right)\left(\begin{array}{cc}
I_{E}\otimes I_{\sigma} & 0\\
0 & I_{E}\otimes b
\end{array}\right).
\]

For (4), simply note that 
\[
\left(\begin{array}{cc}
-I_{\sigma} & I_{\tau}\end{array}\right)\left(\begin{array}{cc}
\mathfrak{z} & b\mathfrak{u}\\
0 & \mathfrak{z}+b\mathfrak{u}
\end{array}\right)=\mathfrak{z}\left(\begin{array}{cc}
-I_{E}\otimes I_{\sigma} & I_{E}\otimes I_{\tau}\end{array}\right)
\]
 and use the properties of $f$.
\end{proof}
Even though the linearity in $\mathfrak{u}$ of the operator $\Delta f_{\sigma,\tau}(\mathfrak{z},\mathfrak{w})(\mathfrak{u})$
has still to be shown, $\Delta f_{\sigma,\tau}(\mathfrak{z},\mathfrak{w})$
can be viewed as a noncommutative difference operator with values
in $B(H_{\tau},H_{\sigma})$.
\begin{defn}
\label{def:Taylor_derivative}With the hypotheses and notation as
in Lemma \ref{Definition_Delta}, we call $\Delta f_{\sigma,\tau}(\mathfrak{z},\mathfrak{w})$
the \emph{Taylor difference operator} determined by $f$ and the points
$\mathfrak{z}$ and $\mathfrak{w}$. If $\sigma=\tau$ and $\mathfrak{z}=\mathfrak{w}$,
we call $\Delta f_{\sigma,\sigma}(\mathfrak{z},\mathfrak{z})$ the
\emph{Taylor derivative} of $f$ at $\mathfrak{z}$ and denote it
$\Delta f_{\sigma}(\mathfrak{z},\mathfrak{z})$ or $\Delta f_{\sigma}(\mathfrak{z})$.
\end{defn}
Even though its linearity in $\mathfrak{u}$ has still to be shown,
we can also define non commutative difference operators of higher
order. We will need these first. So to this end, note that by applying
part (1) of Lemma~\ref{Definition_Delta} repeatedly one finds that
for every $\sigma_{0},\ldots,\sigma_{n}$ in $\Sigma$, for every
$\mathfrak{z}_{i}\in\mathcal{U}(\sigma_{i})$ ($0\leq i\leq n$) and
for every $\mathfrak{u}_{j}\in\mathcal{I}(\sigma_{j}^{E}\circ\varphi,\sigma_{j-1})$
($1\leq j\leq n$), the matrix 
\begin{equation}
f_{\sigma_{0}\oplus\sigma_{1}\oplus\cdots\oplus\sigma_{n}}(\left(\begin{array}{ccccc}
\mathfrak{z}_{0} & \mathfrak{u}_{1} & 0 & \cdots & 0\\
0 & \mathfrak{z}_{1} & \ddots & \ddots & \vdots\\
\vdots & \ddots & \ddots & \ddots & 0\\
\vdots &  & \ddots & \mathfrak{z}_{n-1} & \mathfrak{u}_{n}\\
0 & \cdots & \cdots & 0 & \mathfrak{z}_{n}
\end{array}\right))\label{eq:f_multiple_sigmas}
\end{equation}

\noindent has a block upper triangular form with $f_{\sigma_{0}}(\mathfrak{z}_{0}),\ldots,f_{\sigma_{n}}(\mathfrak{z}_{n})$
on the main diagonal, assuming of course that the argument is in $\mathcal{U}(\sigma_{0}\oplus\sigma_{1}\oplus\cdots\oplus\sigma_{n})$,
i.e., 
\begin{multline}
f_{\sigma_{0}\oplus\sigma_{1}\oplus\cdots\oplus\sigma_{n}}(\left(\begin{array}{ccccc}
\mathfrak{z}_{0} & \mathfrak{u}_{1} & 0 & \cdots & 0\\
0 & \mathfrak{z}_{1} & \ddots & \ddots & \vdots\\
\vdots & \ddots & \ddots & \ddots & 0\\
\vdots &  & \ddots & \mathfrak{z}_{n-1} & \mathfrak{u}_{n}\\
0 & \cdots & \cdots & 0 & \mathfrak{z}_{n}
\end{array}\right))\\
=\begin{bmatrix}f_{\sigma_{0}}(\mathfrak{z}_{0}) & a_{01} & a_{02} & \cdots & a_{0n}\\
0 & f_{\sigma_{1}}(\mathfrak{z}_{1}) & a_{12} & \ddots & a_{1n}\\
\vdots &  & \ddots & \ddots & \vdots\\
\vdots &  &  & f_{\sigma_{n-1}}(\mathfrak{z}_{n-1}) & a_{n-1n}\\
0 & \cdots & \cdots &  & f_{\sigma_{n}}(\mathfrak{z}_{n})
\end{bmatrix}\label{eq:Def_Delta_n}
\end{multline}

\begin{defn}
\noindent \label{def:Delta_n} The function of $\mathfrak{u}_{1}$,
$\mathfrak{u}_{2}$, $\cdots$, $\mathfrak{u}_{n}$ defined by the
$0,n$ entry of the right hand side of Equation \ref{eq:Def_Delta_n}
will be called the \emph{$n^{th}$-order Taylor difference operator}
determined by $\mathfrak{z}_{0}$, $\mathfrak{z}_{1}$,$\cdots$,$\mathfrak{z}_{n}$,
and will be denoted $\Delta^{n}f_{\sigma_{0},\sigma_{1},\cdots,\sigma_{n}}(\mathfrak{z}_{0},\ldots,\mathfrak{z}_{n})$.
If $\mathfrak{z}_{0}=\mathfrak{z}_{1}=\cdots=\mathfrak{z}_{n}=\mathfrak{z}$,
we call $\Delta^{n}f_{\sigma,\sigma,\cdots,\sigma}(\mathfrak{z},\mathfrak{z},\cdots,\mathfrak{z}):=\Delta^{n}f_{\sigma}(\mathfrak{z})$
the $n^{th}$\emph{-order Taylor derivative} of $f_{\sigma}$ at $\mathfrak{z}$.
\end{defn}
The arguments for the proofs of parts (2) and (3) of Lemma~\ref{Definition_Delta}
can be adapted easily to prove the following lemma. We omit the details.
\begin{lem}
\label{Delta_n} Suppose $\{\mathfrak{z}_{i}\}$ and $\{\mathfrak{u}_{j}\}$
are as above. Fix $1\leq k\leq n+1$ and choose $b\in\sigma_{k-1}(M)'$
so that $b\mathfrak{z}_{k-1}=\mathfrak{z}_{k-1}(I_{E}\otimes b)$
(in particular, $b$ can be in $\mathbb{C}$), then: 
\begin{enumerate}
\item If $1<k\leq n$ and if both 
\begin{eqnarray*}
\Delta^{n}f_{\sigma_{0},\sigma_{1},\cdots,\sigma_{n}}(\mathfrak{z}_{0},\ldots,\mathfrak{z}_{n})(\mathfrak{u}_{1},\ldots,\mathfrak{u}_{k-1},b\mathfrak{u}_{k},\ldots,\mathfrak{u}_{n})\\
\end{eqnarray*}
 and $\Delta^{n}f_{\sigma_{0},\sigma_{1},\cdots,\sigma_{n}}(\mathfrak{z}_{0},\ldots,\mathfrak{z}_{n})(\mathfrak{u}_{1},\ldots,\mathfrak{u}_{k-1}(I_{E}\otimes b),\mathfrak{u}_{k},\ldots,\mathfrak{u}_{n})$
are well defined, in the sense that the argument matrices in the expression
\ref{eq:f_multiple_sigmas} lie in $\mathcal{U}(\sigma_{0}\oplus\sigma_{1}\oplus\cdots\oplus\sigma_{n}),$
then 
\begin{multline*}
\Delta^{n}f_{\sigma_{0},\sigma_{1},\cdots,\sigma_{n}}(\mathfrak{z}_{0},\ldots,\mathfrak{z}_{n})(\mathfrak{u}_{1},\ldots,\mathfrak{u}_{k-1},b\mathfrak{u}_{k},\ldots,\mathfrak{u}_{n})\\
=\Delta^{n}f_{\sigma_{0},\sigma_{1},\cdots,\sigma_{n}}(\mathfrak{z}_{0},\ldots,\mathfrak{z}_{n})(\mathfrak{u}_{1},\ldots,\mathfrak{u}_{k-1}(I_{E}\otimes b),\mathfrak{u}_{k},\ldots,\mathfrak{u}_{n}).
\end{multline*}

\item If $k=1$ and if $\Delta^{n}f_{\sigma_{0},\sigma_{1},\cdots,\sigma_{n}}f(\mathfrak{z}_{0},\ldots,\mathfrak{z}_{n})(b\mathfrak{u}_{1},\ldots,\mathfrak{u}_{n})$
is well defined, then 
\begin{multline*}
\Delta^{n}f_{\sigma_{0},\sigma_{1},\cdots,\sigma_{n}}(\mathfrak{z}_{0},\ldots,\mathfrak{z}_{n})(b\mathfrak{u}_{1},\ldots,\mathfrak{u}_{n})\\
=b\Delta^{n}f_{\sigma_{0},\sigma_{1},\cdots,\sigma_{n}}(\mathfrak{z}_{0},\ldots,\mathfrak{z}_{n})(\mathfrak{u}_{1},\ldots,\mathfrak{u}_{n})
\end{multline*}
.
\item If $k=n+1$ and if $\Delta^{n}f_{\sigma_{0},\sigma_{1},\cdots,\sigma_{n}}(\mathfrak{z}_{0},\ldots,\mathfrak{z}_{n})(\mathfrak{u}_{1},\ldots,\mathfrak{u}_{n}(I_{E}\otimes b))$
is well defined, then 
\begin{multline*}
\Delta^{n}f_{\sigma_{0},\sigma_{1},\cdots,\sigma_{n}}(\mathfrak{z}_{0},\ldots,\mathfrak{z}_{n})(\mathfrak{u}_{1},\ldots,\mathfrak{u}_{n}(I_{E}\otimes b))\\
=\Delta^{n}f_{\sigma_{0},\sigma_{1},\cdots,\sigma_{n}}(\mathfrak{z}_{0},\ldots,\mathfrak{z}_{n})(\mathfrak{u}_{1},\ldots,\mathfrak{u}_{n})b.
\end{multline*}

\end{enumerate}
\end{lem}
{}
\begin{rem}
\label{Rem:_Key2}Again we want to note that the argument for Lemma
\ref{Delta_n} shows that if we take the $\sigma_{i}$s to be one
and the same $\sigma$, so that the matrix 
\[
\left(\begin{array}{ccccc}
\mathfrak{z}_{0} & \mathfrak{u}_{1} & 0 & \cdots & 0\\
0 & \mathfrak{z}_{1} & \ddots & \ddots & \vdots\\
\vdots & \ddots & \ddots & \ddots & 0\\
\vdots &  & \ddots & \mathfrak{z}_{n-1} & \mathfrak{u}_{n}\\
0 & \cdots & \cdots & 0 & \mathfrak{z}_{n}
\end{array}\right)
\]
acts from $E\otimes_{\sigma}H_{m\sigma}$ to $H_{m\sigma}$, then
the only matrices necessary to show that $\Delta^{n}f_{\sigma_{0},\sigma_{1},\cdots,\sigma_{n}}(\mathfrak{z}_{0},\ldots,\mathfrak{z}_{n})(\mathfrak{u}_{1},\ldots,\mathfrak{u}_{n})$
extends to be homogeneous in each of the $\mathfrak{u}_{i}$s are
block matrices whose entries are scalar multiples of the identity
on $H_{\sigma}$.
\end{rem}
{}
\begin{lem}
\label{fmatrix} Given $\sigma_{0},\ldots,\sigma_{n}$ in $\Sigma$,
$\mathfrak{z}_{i}\in\mathcal{U}(\sigma_{i})$ ($0\leq i\leq n$) and
$\mathfrak{u}_{j}\in\mathcal{I}(\sigma_{j}^{E}\circ\varphi,\sigma_{j-1})$
($1\leq j\leq n$), then for $0\leq i<j\leq n$, the $i,j$ entry,
$a_{ij}$ of the right hand side of Equation \eqref{eq:Def_Delta_n}
is 
\[
a_{i,j}=\Delta^{j}f_{\sigma_{0},\sigma_{1},\cdots,\sigma_{j}}(\mathfrak{z}_{0},\dots,\mathfrak{z}_{j})(\mathfrak{u}_{1},\ldots,\mathfrak{u}_{j}).
\]
\end{lem}
\begin{proof}
The proof proceeds by induction. For $n=1$, the assertion is Lemma~\ref{Definition_Delta}(1).
So assume it holds for $n$ and write $\sigma=\oplus_{i=0}^{n+1}\sigma_{i}$.
We apply Lemma~\ref{Definition_Delta}(1) repeatedly. If we partition
$\sigma$ as $\sigma=(\oplus_{i=0}^{n}\sigma_{i})\oplus\sigma_{n+1}$,
Lemma~\ref{Definition_Delta} and the induction hypothesis prove
the lemma for all $j\leq n$ and for $i=j=n+1$. To obtain the formula
for $a_{i,n+1}$, simply write $\sigma$ as $\sigma=(\oplus_{k=0}^{i-2}\sigma_{k})\oplus(\oplus_{k=i-1}^{n+1}\sigma_{k})$,
then apply Lemma~\ref{Definition_Delta} and the induction assumption.
\end{proof}
The argument used in the proof of the following lemma was shown to
us by Victor Vinnikov. It will appear in his joint work with D. S.
Kaliuzhny\u{i}-Verbovetsky\u{i} \cite{K-VVPrep}.
\begin{lem}
\label{additivity} Let $f=\{f_{\sigma}\}_{\sigma\in\Sigma}$ be a
matricial family of functions defined on a matricial $E,\Sigma$-family
$\{\mathcal{U}(\sigma)\}_{\sigma\in\Sigma}$ where $\Sigma$ is an
additive subcategory of $NRep(M)$. Suppose $\sigma,\tau\in\Sigma$,
$\mathfrak{z}\in\mathcal{U}(\sigma)$, $\mathfrak{w}\in\mathcal{U}(\sigma)$
and $\mathfrak{u_{i}}\in\mathcal{I}(\tau^{E}\circ\varphi,\sigma)$
for $i=1,2$ are such that $A:=\left(\begin{array}{cc}
\mathfrak{z} & \mathfrak{u}_{1}+\mathfrak{u}_{2}\\
0 & \mathfrak{w}
\end{array}\right)\in\mathcal{U}(\sigma\oplus\tau)$, $B:=\left(\begin{array}{ccc}
\mathfrak{z} & 0 & \mathfrak{u}_{1}\\
0 & \mathfrak{z} & \mathfrak{u}_{2}\\
0 & 0 & \mathfrak{w}
\end{array}\right)\in\mathcal{U}(2\sigma\oplus\tau)$ and $C:=\left(\begin{array}{ccc}
\mathfrak{z} & 0 & \mathfrak{u}_{2}\\
0 & \mathfrak{z} & \mathfrak{u}_{1}\\
0 & 0 & \mathfrak{w}
\end{array}\right)\in\mathcal{U}(2\sigma\oplus\tau)$. Then 
\[
\Delta f_{\sigma,\tau}(\mathfrak{z},\mathfrak{w})(\mathfrak{u}_{1}+\mathfrak{u}_{2})=\Delta f_{\sigma,\tau}(\mathfrak{z},\mathfrak{w})(\mathfrak{u}_{1})+\Delta f(\mathfrak{z},\mathfrak{w})(\mathfrak{u}_{2}).
\]
\end{lem}
\begin{proof}
Considering $f_{2\sigma\oplus\tau}(B)$ and $f_{2\sigma\oplus\tau}(C)$
and using Lemma~\ref{Definition_Delta}(1) with $2\sigma\oplus\tau$
split as $\sigma\oplus(\sigma\oplus\tau)$ we find that these matrices
have the following from: 
\[
f_{2\sigma\oplus\tau}(B)=\left(\begin{array}{ccc}
f_{\sigma}(\mathfrak{z}) & x & y\\
0 & f_{\sigma}(\mathfrak{z}) & \Delta f_{\sigma,\tau}(\mathfrak{z},\mathfrak{w})(\mathfrak{u}_{2})\\
0 & 0 & f_{\tau}(\mathfrak{w})
\end{array}\right)
\]
 and 
\[
f_{2\sigma\oplus\tau}(C)=\left(\begin{array}{ccc}
f_{\sigma}(\mathfrak{z}) & u & v\\
0 & f_{\sigma}(\mathfrak{z}) & \Delta f_{\sigma,\tau}(\mathfrak{z},\mathfrak{w})(\mathfrak{u}_{1})\\
0 & 0 & f_{\tau}(\mathfrak{w})
\end{array}\right)
\]
for some $x,y,u,$ and $v$. Writing $S$ for the $3\times3$ permutation
matrix associated with the transposition $(1,2)$, we see that $C=SBS^{-1}=SBS$.
Thus $f(C)=Sf(B)S$ and, therefore, $y=\Delta f_{\sigma,\tau}(\mathfrak{z},\mathfrak{w})(\mathfrak{u}_{1})$
and 
\[
f_{2\sigma\oplus\tau}(B)=\left(\begin{array}{ccc}
f(\mathfrak{z}) & x & \Delta f_{\sigma,\tau}(\mathfrak{z},\mathfrak{w})(\mathfrak{u}_{1})\\
0 & f(\mathfrak{z}) & \Delta f_{\sigma,\tau}(\mathfrak{z},\mathfrak{w})(\mathfrak{u}_{2})\\
0 & 0 & f_{\tau}(\mathfrak{w})
\end{array}\right).
\]
Now write $D=\left(\begin{array}{ccc}
I & I & 0\\
0 & 0 & I
\end{array}\right)$. Then $DB=AD$ and, thus $Df_{2\sigma\oplus\tau}(B)=f_{\sigma\oplus\tau}(A)D$.
Since $f_{\sigma\oplus\tau}(A)=\left(\begin{array}{cc}
f_{\sigma}(\mathfrak{z}) & \Delta f_{\sigma,\tau}(\mathfrak{z},\mathfrak{w})(\mathfrak{u}_{1}+\mathfrak{u}_{2})\\
0 & f_{\tau}(\mathfrak{w})
\end{array}\right)$, we may equate the $(1,3)$ entries to obtain the result.\end{proof}
\begin{rem}
\label{Rem:_Key3} Again, in Lemma \ref{additivity}, if we set $\tau=\sigma$,
we see that the only matrices used in the proof are matrices that
are block scalar matrices. Combining Lemma \ref{additivity} with
Lemma~\ref{Definition_Delta}(2), we conclude that $\Delta f_{\sigma,\tau}(\mathfrak{z},\mathfrak{w})(\mathfrak{u})$
extends to be linear in $\mathfrak{u}$. 
\end{rem}
We can now deduce the multilinearity of $\Delta^{n}f_{\sigma_{0},\sigma_{1},\cdots,\sigma_{n}}(\mathfrak{z}_{0},\ldots,\mathfrak{z}_{n})(\cdot,\ldots,\cdot)$
(with $\mathfrak{z}_{0},\ldots,\mathfrak{z}_{n}$ fixed) from the
linearity of $\Delta f_{\sigma,\tau}(\mathfrak{z},\mathfrak{w})(\mathfrak{u})$
in $\mathfrak{u}$. In order to avoid specifying the conditions on
the variables as in the statement of the lemma above, we restrict
ourselves to matricial discs.
\begin{cor}
\label{multilinearity} Let $f=\{f_{\sigma}\}_{\sigma\in\Sigma}$
be a matricial family of functions, defined in a matricial disc $\mathbb{D}(\zeta,r)$.
Then for every $\sigma_{0},\ldots,\sigma_{n}$ in $\Sigma$ and every
$\mathfrak{z}_{i}\in\mathbb{D}(\zeta_{\sigma_{i}},r,\sigma_{i})$,
($0\leq i\leq n$), the function 
\[
\Delta^{n}f_{\sigma_{0},\sigma_{1},\cdots,\sigma_{n}}(\mathfrak{z}_{0},\ldots,\mathfrak{z}_{n})(\mathfrak{u}_{1},\ldots,\mathfrak{u}_{n}),
\]
defined for $\mathfrak{u}_{j}\in\mathcal{I}(\sigma_{j}^{E}\circ\varphi,\sigma_{j-1})$
($1\leq j\leq n$) with norm sufficiently small can be extended to
a multilinear map on $\mathcal{I}(\sigma_{1}^{E}\circ\varphi,\sigma_{0})\times\cdots\times\mathcal{I}(\sigma_{n}^{E}\circ\varphi,\sigma_{n-1})$. \end{cor}
\begin{proof}
For $n=1$, this is shown in Lemma~\ref{additivity}. For the general
case, fix $1\leq j\leq n$ and write 
\[
\left(\begin{array}{ccccc}
\mathfrak{z}_{0} & \mathfrak{u}_{1} & 0 & \cdots & 0\\
0 & \mathfrak{z}_{1} & \ddots & \ddots & \vdots\\
\vdots & \ddots & \ddots & \ddots & 0\\
\vdots &  & \ddots & \mathfrak{z}_{n-1} & \mathfrak{u}_{n}\\
0 & \cdots & \cdots & 0 & \mathfrak{z}_{n}
\end{array}\right)=\left(\begin{array}{cc}
X & Z\\
0 & Y
\end{array}\right)
\]
where $X$ is a $j\times j$ block and $Z$ is a $j\times(n-j)$ block
with $\mathfrak{u}_{j}$ in the bottom left corner and all other entries
$0$. Applying $f_{\sigma_{0}\oplus\sigma_{1}\oplus\cdots\oplus\sigma_{n}}$
to this matrix we get, on one hand, a matrix with $\Delta^{n}f_{\sigma_{0},\sigma_{1},\cdots,\sigma_{n}}(\mathfrak{z}_{0},\ldots,\mathfrak{z}_{n})(\mathfrak{u}_{1},\ldots,\mathfrak{u}_{n})$
in the upper right corner and, on the other hand, a $2\times2$ block
matrix that we shall simply write as $\Delta f(X,Y)(Z)$ in the $(1,2)$
corner. Using Lemma~\ref{additivity}, we know that $\Delta f(X,Y)(Z)$
is linear in $Z$ and, thus, linear in $\mathfrak{u}_{j}$. Therefore
$\Delta^{n}f_{\sigma_{0},\sigma_{1},\cdots,\sigma_{n}}(\mathfrak{z}_{0},\ldots,\mathfrak{z}_{n})(\mathfrak{u}_{1},\ldots,\mathfrak{u}_{n})$,
being a corner, is linear in $\mathfrak{u}_{j}$. This completes the
proof. 
\end{proof}
The following theorem is an analogue of Taylor's Taylor Theorem with
remainder \cite[Proposition 4.2]{Tay73a}.
\begin{thm}
\label{Taylor_Remainder} Let $f=\{f_{\sigma}\}_{\sigma\in\Sigma}$
be a matricial family of functions defined on a matricial disc $\mathbb{D}(\zeta,r)$.
Fix $\sigma\in\Sigma$ and choose $\mathfrak{z}$ and $\mathfrak{w}$
in $\mathbb{D}(\zeta_{\sigma},r,\sigma)$ so that the matrix 
\begin{equation}
\left(\begin{array}{ccccc}
\mathfrak{z} & \mathfrak{w} & 0 & \cdots & 0\\
0 & \mathfrak{z} & \ddots & \ddots & \vdots\\
\vdots & \ddots & \ddots & \ddots & 0\\
\vdots &  & \ddots & \mathfrak{z} & \mathfrak{w}\\
0 & \cdots & \cdots & 0 & \mathfrak{z}+\mathfrak{w}
\end{array}\right)\label{eq:Big_matrix}
\end{equation}
 lies in $\mathbb{D}(\zeta_{(n+1)\sigma},r,(n+1)\sigma)$. Then 
\[
f_{\sigma}(\mathfrak{z}+\mathfrak{w})=\sum_{k=0}^{n-1}\Delta^{k}f_{\sigma}(\mathfrak{z})(\mathfrak{w},\ldots,\mathfrak{w})+\Delta^{n}f_{\sigma,\sigma,\cdots,\sigma}(\mathfrak{z},\ldots,\mathfrak{z},\mathfrak{z}+\mathfrak{w})(\mathfrak{w},\ldots,\mathfrak{w}).
\]
 \end{thm}
\begin{proof}
Denote the matrix (\ref{eq:Big_matrix}) by $A$ and note that, 
\[
A\left(\begin{array}{c}
I_{E}\otimes I_{\sigma}\\
\vdots\\
\vdots\\
\vdots\\
I_{E}\otimes I_{\sigma}
\end{array}\right)=\left(\begin{array}{c}
I_{\sigma}\\
\vdots\\
\vdots\\
\vdots\\
I_{\sigma}
\end{array}\right)(\mathfrak{z}+\mathfrak{w}).
\]
Applying the intertwining property of $f_{(n+1)\sigma}$ and Lemma~\ref{fmatrix}
completes the proof. 
\end{proof}
We would like to pass to the limit as $n\to\infty$, in Theorem \ref{Taylor_Remainder},
but to do this, we must impose an additional hypothesis on $f$, viz.
$f$ must be locally uniformly bounded in the sense of Theorem \ref{thm:Matricial_implies_Frechet_hol.}.
The following theorem should be compared with Theorem \ref{boundedhol}.
There, we assumed that the additive subcategory $\Sigma$ is full
and contains a special generator for $NRep(M)$. Here we make no assumptions
on $\Sigma$ other than it is an additive subcategory of $NRep(M)$.
\begin{thm}
\label{TTseries} Let $f=\{f_{\sigma}\}_{\sigma\in\Sigma}$ be a matricial
family of functions defined on a matricial disc $\mathbb{D}(\zeta,r)$
and suppose that $f$ is locally uniformly bounded in the sense of
Theorem \ref{thm:Matricial_implies_Frechet_hol.}. Then: 
\begin{enumerate}
\item Each $f_{\sigma}$ is Frechet differentiable in $\mathfrak{z}$, $\mathfrak{z}\in\mathbb{D}(\zeta_{\sigma},r,\sigma)$,
and
\[
f_{\sigma}'(\mathfrak{z})(\mathfrak{w})=\Delta f(\mathfrak{z})(\mathfrak{w}).
\]
 
\item Each $f_{\sigma}$ may be expanded on $\mathbb{D}(\zeta_{\sigma},r,\sigma)$
as 
\begin{equation}
f_{\sigma}(\zeta_{\sigma}+\mathfrak{z})=\sum_{k=0}^{\infty}\Delta^{k}f_{\sigma}(\zeta_{\sigma})(\mathfrak{z},\ldots,\mathfrak{z}),\label{TT}
\end{equation}
 where the series converges absolutely and uniformly on every disc
$\mathbb{D}(0,r_{0},\sigma)$ with $r_{0}<r$. 
\end{enumerate}
\end{thm}
\begin{proof}
To prove (1), let $r_{0}<r$ and let $M$ be greater than 
\[
\sup_{\sigma\in\Sigma}\sup_{\mathfrak{z}\in\mathbb{D}(\zeta_{\sigma},r_{0})}\Vert f_{\sigma}(\mathfrak{z})\Vert,
\]
which is finite by assumption. Fix $\mathfrak{z}\in\mathbb{D}(\zeta_{\sigma},r_{0},\sigma)$
and note that for $\mathfrak{w}$ with norm sufficiently small (say,
$||\mathfrak{w}||\leq s$), the matrix 
\begin{equation}
\left(\begin{array}{ccc}
\mathfrak{z} & \mathfrak{w} & 0\\
0 & \mathfrak{z} & \mathfrak{w}\\
0 & 0 & \mathfrak{z}+\mathfrak{w}
\end{array}\right)\label{eq:3x3 matrix}
\end{equation}
lies in $\mathbb{D}(\zeta_{3\sigma},r_{0},3\sigma)=\mathbb{D}(\zeta_{\sigma}\oplus\zeta_{\sigma}\oplus\zeta_{\sigma},r_{0},3\sigma)$.
But this implies in particular, that 
\[
\begin{bmatrix}\mathfrak{z} & \mathfrak{w}\\
0 & \mathfrak{z}
\end{bmatrix}
\]
lies in $\mathbb{D}(\zeta_{2\sigma},r_{0},2\sigma)$. Our boundedness
assumption, then implies that for $\Vert\mathfrak{w}\Vert\leq s$,
\[
\Vert\Delta f_{\sigma}(\mathfrak{z})(\mathfrak{w})\Vert\leq\Vert\begin{bmatrix}f_{\sigma}(\mathfrak{z}) & \Delta f_{\sigma}(\mathfrak{z})(\mathfrak{w})\\
0 & f_{\sigma}(\mathfrak{z})
\end{bmatrix}\Vert=\Vert f_{2\sigma}(\begin{bmatrix}\mathfrak{z} & \mathfrak{w}\\
0 & \mathfrak{z}
\end{bmatrix})\Vert\leq M,
\]
i.e., $\Delta f_{\sigma}(\mathfrak{z})(\cdot)$ extends to be a \emph{bounded}
linear operator. Further, the same sort of argument shows that when
we apply $f_{3\sigma}$ to the matrix \eqref{eq:3x3 matrix} for $\Vert\mathfrak{w}\Vert\leq s$,
we have $||\Delta^{2}f_{\sigma,\sigma,\sigma}(\mathfrak{z},\mathfrak{z},\mathfrak{z}+\mathfrak{w})(\mathfrak{w},\mathfrak{w})||\leq M$,
by definition of $\Delta^{2}f_{\cdot,\cdot,\cdot}(\cdot,\cdot,\cdot)(\cdot,\cdot)$.
But by Theorem~\ref{Taylor_Remainder}, $f_{\sigma}(\mathfrak{z}+\mathfrak{w})-f_{\sigma}(\mathfrak{z})-\Delta f_{\sigma}(\mathfrak{z})(\mathfrak{w})=\Delta^{2}f_{\sigma,\sigma,\sigma}(\mathfrak{z},\mathfrak{z},\mathfrak{z}+\mathfrak{w})(\mathfrak{w},\mathfrak{w})$.
So, if we write $\mathfrak{w}_{0}$ for $\frac{s}{||\mathfrak{w}||}\mathfrak{w}$,
then $||f_{\sigma}(\mathfrak{z}+\mathfrak{w})-f_{\sigma}(\mathfrak{z})-\Delta f_{\sigma}(\mathfrak{z})(\mathfrak{w})||=||(\frac{||\mathfrak{w}||}{s})^{2}\Delta^{2}f_{\sigma,\sigma,\sigma}(\mathfrak{z},\mathfrak{z},\mathfrak{z}+\mathfrak{w})(\mathfrak{w}_{0},\mathfrak{w}_{0})||\leq\frac{||\mathfrak{w}||^{2}M}{s^{2}}$,
which proves that 
\[
\frac{1}{||\mathfrak{w}||}||f_{\sigma}(\mathfrak{z}+\mathfrak{w})-f_{\sigma}(\mathfrak{z})-\Delta f_{\sigma}(\mathfrak{z})(\mathfrak{w})||\underset{}{\longrightarrow0}
\]
as $\mathfrak{w}\to0$. Thus (1) is proved.

To prove (2), fix $0<r_{0}<r_{1}<r$, write $q:=r_{1}/r_{0}$ and
let $r_{2}:=\frac{r_{1}}{\sqrt{1+q^{2}}}<r_{0}$. Then for every $n$
and every $\mathfrak{z}\in\mathbb{D}(\zeta_{\sigma},r_{2},\sigma)$,
\[
\mathfrak{w}:=\left(\begin{array}{ccccc}
\zeta_{\sigma} & q(\mathfrak{z}-\zeta_{\sigma}) & 0 & \cdots & 0\\
0 & \zeta_{\sigma} & \ddots & \ddots & \vdots\\
\vdots & \ddots & \ddots & \ddots & 0\\
\vdots &  & \ddots & \zeta_{\sigma} & q(\mathfrak{z}-\zeta_{\sigma})\\
0 & \cdots & \cdots & 0 & \mathfrak{z}
\end{array}\right)
\]
lies in $\mathbb{D}(\zeta_{(n+1)\sigma},r_{1},(n+1)\sigma)\subseteq\mathbb{D}(\zeta_{(n+1)\sigma},r,(n+1)\sigma)$.
Thus 
\[
||\Delta^{n}f_{\sigma,\sigma,\cdots,\sigma}(\zeta_{\sigma},\ldots,\zeta_{\sigma},\mathfrak{z})(q(\mathfrak{z}-\zeta_{\sigma}),\ldots,q(\mathfrak{z}-\zeta_{\sigma}))||\leq||f(\mathfrak{w})||\leq M.
\]
From Lemma~\ref{fmatrix}, we conclude that 
\[
||\Delta^{n}f_{\sigma,\sigma,\cdots,\sigma}(\zeta_{\sigma},\ldots,\zeta_{\sigma},\mathfrak{z})(\mathfrak{z}-\zeta_{\sigma},\ldots,\mathfrak{z}-\zeta_{\sigma})||\leq q^{-n}M\underset{n\rightarrow\infty}{\longrightarrow}0.
\]
So Theorem~\ref{Taylor_Remainder} proves (\ref{TT}) for $\mathfrak{z}\in\mathbb{D}(\zeta_{\sigma},r_{2},\sigma)$.

To pass to $\mathfrak{z}$ in the larger disc $\mathbb{D}(\zeta_{\sigma},r_{0},\sigma)$,
observe that for such $\mathfrak{z}$ 
\[
\mathfrak{u}:=\left(\begin{array}{ccccc}
\zeta_{\sigma} & q(\mathfrak{z}-\zeta_{\sigma}) & 0 & \cdots & 0\\
0 & \zeta_{\sigma} & \ddots & \ddots & \vdots\\
\vdots & \ddots & \ddots & \ddots & 0\\
\vdots &  & \ddots & \zeta_{\sigma} & q(\mathfrak{z}-\zeta_{\sigma})\\
0 & \cdots & \cdots & 0 & \zeta_{\sigma}
\end{array}\right)
\]
lies in the disc $\mathbb{D}(\zeta_{(n+1)\sigma},r_{1},(n+1)\sigma)\subseteq\mathbb{D}(\zeta_{(n+1)\sigma},r,(n+1)\sigma).$
Consequently 
\begin{multline*}
||\Delta^{n}f_{\sigma,\sigma,\cdots,\sigma}(\zeta_{\sigma},\ldots,\zeta_{\sigma})(\mathfrak{z}-\zeta_{\sigma},\ldots,\mathfrak{z}-\zeta_{\sigma})||=\\
q^{-n}||\Delta^{n}f_{\sigma,\sigma,\cdots,\sigma}(\zeta_{\sigma},\ldots,\zeta_{\sigma})(q(\mathfrak{z}-\zeta_{\sigma}),\ldots,q(\mathfrak{z}-\zeta_{\sigma}))||\leq q^{-n}M.
\end{multline*}
It follows that the series in Equation (\ref{TT}) converges absolutely
and uniformly on $\mathbb{D}(0,r_{0})$.

To see that the sum is $f_{\sigma}(\zeta_{\sigma}+\mathfrak{z})$,
fix $\mathfrak{z}\in\mathbb{D}(0,r_{0},\sigma)$ and consider the
function $h(\lambda):=f_{\sigma}(\zeta_{\sigma}+\lambda\mathfrak{z})$.
By part (1) $h$ is analytic in a disc centered at the origin in the
complex plane having radius bigger than $1$. On the other hand we
may also form 
\begin{multline*}
g(\lambda):=\sum_{k=0}^{\infty}\Delta^{k}f_{\sigma,\sigma,\cdots,\sigma}(\zeta_{\sigma},\ldots,\zeta_{\sigma})(\lambda\mathfrak{z},\ldots,\lambda\mathfrak{z})\\
=\sum_{k=0}^{\infty}\lambda^{k}\Delta^{k}f_{\sigma,\sigma,\cdots,\sigma}(\zeta_{\sigma},\ldots,\zeta_{\sigma})(\mathfrak{z},\ldots,\mathfrak{z}),
\end{multline*}
which is also analytic in a disc centered at the origin of radius
bigger than $1$. By what we have just shown, these two Banach space-valued
functions agree on a neighborhood of the origin in the complex plane.
Therefore, they agree on the intersection of their domains of definition
(\cite[Theorem 3.11.5]{HP1974}), which includes $1$, i.e., Equation
\eqref{TT} is valid throughout the disc $\mathbb{D}(0,r_{0},\sigma)$.\end{proof}
\begin{rem}
As a special case of Theorem \ref{TTseries}, we obtain a formula
that was inspired by \cite{K-V2012}. If $\sigma$ is a normal representation
of $M$, then the subcategory $\Sigma$ generated by $\sigma$ is
just the collection of finite multiples of $\sigma$. The collection
of morphisms from $m\sigma$ to $n\sigma$ is just the $n\times m$
matrices over $\sigma(M)'.$ A $\zeta_{0}\in E^{\sigma*}$ generates
an additive field over $\Sigma$ simply by setting $\zeta_{k\sigma}=\zeta_{0}\oplus\zeta_{0}\oplus\cdots\oplus\zeta_{0}$
($k$ summands). Then, if $f=\{f_{m\sigma}\}_{m\in\mathbb{N}}$ is
a locally bounded, matricial function on the matricial disc $\mathbb{D}(\zeta,r)$,
$\zeta=\{\zeta_{m\sigma}\}_{m\in\mathbb{N}}$, then Theorem \ref{TT}
implies that for every $m\in\mathbb{N}$ and every $\mathfrak{z}$
in $\mathbb{D}(\zeta_{m\sigma},r,m\sigma)$, 
\[
f_{m\sigma}(\mathfrak{z})=\sum_{k=0}^{\infty}\Delta^{k}f_{m\sigma}(\zeta_{m\sigma})(\mathfrak{z}-\zeta_{m\sigma},\ldots,\mathfrak{z}-\zeta_{m\sigma}).
\]

\end{rem}
Suppose $f=\{f_{\sigma}\}_{\sigma\in\Sigma}$ is a matricial function
defined on a matricial disc $\mathbb{D}(\zeta,r)$ and suppose that
$f$ is locally uniformly bounded as in Theorems \eqref{thm:Matricial_implies_Frechet_hol.}
and \eqref{TT}. Then we have just seen that each $f_{\sigma}$ is
Frechet differentiable throughout $\mathbb{D}(\zeta_{\sigma},r,\sigma)$.
By \cite[Theorem 26.3.10]{HP1974} $f_{\sigma}$ can be expanded in
a unique power series about each point $\mathfrak{z}\in\mathbb{D}(\zeta_{\sigma},r,\sigma)$,
and the series converges at least in the largest open ball centered
at $\mathfrak{z}$ contained in $\mathbb{D}(\zeta_{\sigma},r,\sigma)$.
The terms of the power series are built from the higher order Frechet
derivatives of $f_{\sigma}$. Recall from the general theory of differentiable
functions on Banach spaces (applied to our setting) that the $n^{th}$
order Frechet derivative of $f_{\sigma}$ is a $B(H_{\sigma})$-valued
function, denoted $D^{n}f_{\sigma}$, that is defined on $\mathbb{D}(\zeta_{\sigma},r,\sigma)\times E^{\sigma*}\times E^{\sigma*}\times\cdots\times E^{\sigma*}$
and has the following properties: For each $\mathfrak{z}$, $D^{n}f_{\sigma}(\mathfrak{z})(\mathfrak{u}_{1},\ldots,\mathfrak{u}_{n})$
is a bounded, symmetric, multilinear function of $(\mathfrak{u}_{1},\ldots,\mathfrak{u}_{n})\in E^{\sigma*}\times E^{\sigma*}\times\cdots\times E^{\sigma*}$,
where the norm $\Vert D^{n}f_{\sigma}(\mathfrak{z})(\cdot,\ldots,\cdot)\Vert$
is locally bounded in $\mathfrak{z}\in\mathbb{D}(0,r,\sigma)$ (\cite[Theorem 26.3.5]{HP1974}).
If we write $D^{n}f_{\sigma}(\mathfrak{z})(\mathfrak{u})$ for $D^{n}f_{\sigma}(\mathfrak{z})(\mathfrak{u},\ldots,\mathfrak{u})$,
then $D^{n}f_{\sigma}(\mathfrak{z})(\mathfrak{u})$ is homogeneous
of degree $n$ in $\mathfrak{u}$ and for each $\mathfrak{a}\in\mathbb{D}(\zeta_{\sigma},r,\sigma)$
there is an $r'$, depending on $\mathfrak{a}$, such that 
\begin{equation}
f_{\sigma}(\mathfrak{z}+\mathfrak{u})=\sum_{n=0}^{\infty}\frac{1}{n!}D^{n}f_{\sigma}(\mathfrak{z})(\mathfrak{u}),
\end{equation}
with the convergence uniform for $||\mathfrak{z}-\mathfrak{a}||<r'$
and $||\mathfrak{u}||<r'$ (\cite[Theorem 3.17.1]{HP1974}).

When $\mathfrak{z}=\zeta_{\sigma}$ we find that 
\begin{equation}
f_{\sigma}(\zeta_{\sigma}+\mathfrak{u})=\sum_{n=0}^{\infty}\frac{1}{n!}D^{n}f_{\sigma}(\zeta_{\sigma})(\mathfrak{u})\label{Tseries}
\end{equation}

Since each of the summands in equations (\ref{TT}) and (\ref{Tseries})
is homogeneous, we conclude that 
\[
\Delta^{k}f_{\sigma}(\zeta_{\sigma})(\mathfrak{z})=\frac{1}{k!}D^{k}f_{\sigma}(\zeta_{\sigma})(\mathfrak{z})
\]
for all $k$ and $\mathfrak{z}$ for which the left hand side is well
defined.

We may therefore summarize our analysis as follows.
\begin{cor}
\label{nlinear} Let $\Sigma$ be an additive subcategory of $NRep(M)$,
let $f=\{f_{\sigma}\}_{\sigma\in\Sigma}$ be a matricial family of
functions defined on a matricial disc $\mathbb{D}(\zeta,r)$, and
suppose that $f$ is locally uniformly bounded on $\mathbb{D}(\zeta,r)$.
Then, for every $\sigma\in\Sigma$, every $k\geq0$, and every $\mathfrak{z}\in\mathbb{D}(0,r,\sigma)$,
Taylor's Taylor derivatives and the Frechet derivatives of $f_{\sigma}$
are related by the equation 
\begin{equation}
\Delta^{k}f_{\sigma}(\zeta_{\sigma})(\mathfrak{z},\ldots,\mathfrak{z})=\frac{1}{k!}D^{k}f(\zeta_{\sigma})(\mathfrak{z}),\label{eq:Taylor's_coefficients}
\end{equation}
and so $\Delta^{k}f_{\sigma}(\zeta_{\sigma})(\mathfrak{z})$ is the
restriction to the diagonal of \foreignlanguage{english}{\textup{$\mathbb{D}(\zeta_{\sigma},r,\sigma)^{k}$}}
of a bounded, \emph{symmetric,} $k$-linear map on $E^{\sigma*}\times E^{\sigma*}\cdots\times E^{\sigma*}$.
Consequently, we may write 
\begin{eqnarray*}
f_{\sigma}(\zeta_{\sigma}+\mathfrak{z}) & = & \sum_{k=0}^{\infty}\Delta^{k}f_{\sigma}(\zeta_{\sigma})(\mathfrak{z})\\
 & = & \sum_{k=0}^{\infty}\frac{1}{k!}D^{k}f(\zeta_{\sigma})(\mathfrak{z})
\end{eqnarray*}
\end{cor}
\begin{rem}
\label{Partial Summary}At this point we pause to take stock of what
we have proved, and with what hypotheses. First of all, we have proved
the first assertion of Theorem \ref{thm:Matricial_implies_Frechet_hol.}.
Further, we identified the higher Frechet derivatives of $f$ with
Taylor's Taylor derivatives (Equation \eqref{eq:Taylor's_coefficients}).
We did this with the minimal hypotheses that we have placed on our
additive categories, vis., that they contain the natural injections
and projections for finite direct sums. 

To obtain the second assertion of Theorem \ref{thm:Matricial_implies_Frechet_hol.},
we will need the assumptions stated there. Even though in some of
the results to follow we can we can get by with slightly weaker assumptions,
for the remainder of this section, we shall assume:
\begin{quote}
\emph{$\Sigma$ is full; every $\sigma$ in $\Sigma$ is faithful;
and $\zeta$ is central.}
\end{quote}
\end{rem}
The assumption that $\zeta$ is central allows us to assume that $\zeta_{\sigma}=0$
for each $\sigma$. That is, we can translate the disc $\mathbb{D}(\zeta,R)$
to $\mathbb{D}(0,R)$ whenever $\zeta$ is central.
\begin{lem}
\label{lem:Translation_invariance}Suppose $\Sigma$ is a full additive
subcategory of $NRep(M)$, and suppose that $\zeta=\{\zeta_{\sigma}\}_{\sigma\in\Sigma}$
is a central additive field on $\Sigma$. Suppose, also, that $f=\{f_{\sigma}\}_{\sigma\in\Sigma}$
is a matricial function define on the matricial disc $\mathbb{D}(\zeta,R)$,
$0<R\leq\infty$, and define $g=\{g_{\sigma}\}_{\sigma\in\Sigma}$
by setting $g_{\sigma}(\mathfrak{z}):=f_{\sigma}(\mathfrak{z}+\zeta_{\sigma})$.
Then $g$ is a matricial function on $\mathbb{D}(0,R)$ and for all
positive integers $k$, 
\[
\Delta^{k}f_{\sigma}(\zeta_{\sigma})(\mathfrak{z}-\zeta_{\sigma})=\Delta^{k}g_{\sigma}(0)(\mathfrak{z}-\zeta_{\sigma}).
\]
\end{lem}
\begin{proof}
The fact that $g$ is matricial is immediate from the observation
that when $\zeta$ is central, $\mathcal{I}(\sigma\times(\mathfrak{z}+\zeta_{\sigma}),\tau\times(\mathfrak{w}+\zeta_{\tau}))=\mathcal{I}(\sigma\times\mathfrak{z},\tau\times\mathfrak{w})$.
The equation for Taylor's derivatives is immediate from their definition,
Definition \ref{def:Delta_n}.
\end{proof}
We want to show that this symmetric multilinear map is the restriction
of a completely bounded map defined on the tensor product $E^{\sigma*}\otimes E^{\sigma*}\cdots\otimes E^{\sigma*}$,
balanced over $\sigma(M)'$. To accomplish this, we need the following
technical lemma.
\begin{lem}
\label{ljmj} Suppose $f=\{f_{\sigma}\}_{\sigma\in\Sigma}$ is a matricial
function defined on a matricial disc centered at the origin $\mathbb{D}(0,r)=\{\mathbb{D}(0,r,\sigma\}_{\sigma\in\Sigma}.$
Suppose that $p$ and $k$ are positive integers and that for $1\leq j\leq k$,
integers $l(j)$ and $m(j)$ are defined and satisfy 
\[
1\leq l(1)<m(1),l(2)<m(2),l(3)<\cdots<m(k)\leq p.
\]
Suppose also that $\mathfrak{u}_{1},\ldots,\mathfrak{u}_{k}$ are
in $E^{\sigma*}$ and let $\mathfrak{U}$ be the $p\times p$ matrix
over $E^{\sigma*}$ (viewed as an element of $E^{p\sigma*}$) whose
$(l(j),m(j))$ entry is $\mathfrak{u}_{j}$, for all $j$, $1\leq j\leq k$,
and whose other entries are all zero. Then $A:=f_{p\sigma}(\mathfrak{U})$,
viewed as a $p\times p$ matrix over $B(H_{\sigma})$, will have non
zero entries only in positions $(l(j),m(j+s))$ where $m(j)=l(j+1),m(j+1)=l(j+2),\ldots,m(j+s-1)=l(j+s)$.
In these positions we will have 
\[
A_{l(j),m(j+s)}=\Delta^{s+1}f_{\sigma}(0)(\mathfrak{u}_{j},\ldots,\mathfrak{u}_{j+s}).
\]
 \end{lem}
\begin{proof}
Suppose first that, for every $j$, $m(j)=l(j)+1$. Then the non zero
entries of $\mathfrak{U}$ are all on the first diagonal above the
main one and the result follows easily from Lemma~\ref{fmatrix}
(noting that $\Delta^{r}f(0)(\mathfrak{z}_{1},\ldots,\mathfrak{z}_{r})=0$
if one of the $\mathfrak{z}_{i}$ is $0$). For the general case,
let $\theta$ be a permutation on $\{1,\ldots,p\}$ such that $\theta(l(1))=1$,
$\theta(m(j))=\theta(l(j))+1$ for all $j$ and $\theta(l(j+1))=\theta(m(j))+|m(j)-l(j+1)|$
for all $1\leq j\leq p-1$. Such $\theta$ always exists (but, in
general, is not unique). Write $S$ for the permutation matrix associated
with $\theta$ and consider $\mathfrak{U}'=S\mathfrak{U}(I\otimes S^{-1})$.
This matrix will have non zero entries only on the first diagonal
above the main diagonal and, thus, $f(\mathfrak{U}')=SAS^{-1}$ will
be of the form described above and it will follow that $A$ satisfies
the assertion of the lemma.
\end{proof}
Recall that if $V_{1},V_{2},\cdots,V_{k},\mbox{ and }X$ are operator
spaces and if $\varphi:V_{1}\times V_{2}\times\cdots\times V_{k}\rightarrow X$
is a multilinear map, then one writes $\varphi_{(n)}$ for the multilinear
map 
\[
\varphi_{(n)}:M_{n}(V_{1})\times\cdots\times M_{n}(V_{k})\rightarrow M_{n}(X)
\]
 defined by 
\[
\varphi_{(n)}((\alpha_{1}\otimes v_{1}),\cdots,(\alpha_{k}\otimes v_{k}))=\alpha_{1}\alpha_{2}\cdots\alpha_{k}\otimes\varphi(v_{1},v_{2},\ldots,v_{k})
\]
 for $\alpha_{i}\in M_{n}(\mathbb{C})$ and $v_{i}\in V_{i}$ \cite[Section 9.1]{Effros2000}.
\begin{lem}
\label{The_map_phi(n)} Let $f=\{f_{\sigma}\}_{\sigma\in\Sigma}$
be a matricial family of functions defined on a matricial $E,\Sigma$-family
of discs $\{\mathbb{D}(0,r,\sigma)\}_{\sigma\in\Sigma}$, where $\Sigma$
is a full additive subcategory of $NRep(M)$. For $\sigma\in\Sigma$,
write $\varphi:E^{\sigma*}\times E^{\sigma*}\cdots\times E^{\sigma*}\rightarrow B(H_{\sigma})$
for the map 
\[
\varphi(\mathfrak{u}_{1},\ldots,\mathfrak{u}_{k})=\Delta^{k}f_{\sigma}(0)(\mathfrak{u}_{1},\ldots,\mathfrak{u}_{k}),\qquad\mathfrak{u}_{i}\in E^{\sigma*},\, i=1,2,\cdots,k.
\]
Then, when $M_{n}(E^{\sigma*})$ is identified with $E^{n\sigma*}$,
we have 
\begin{equation}
\varphi_{(n)}(\mathfrak{U}_{1},\ldots,\mathfrak{U}_{k})=\Delta^{k}f_{\sigma}(0)(\mathfrak{U}_{1},\ldots,\mathfrak{U}_{k})\label{phin}
\end{equation}
 for $\mathfrak{U}_{i}\in E^{n\sigma*}$. \end{lem}
\begin{proof}
Since the functions on both sides of (\ref{phin}) are $k$-linear,
it suffices to prove the lemma for matrices $\mathfrak{U}_{i}$ that
have only one non zero entry. If $\mathfrak{U}$ is such a matrix
and the only non zero entry is $\mathfrak{u}$, which lies in the
$(i,j)$ position, we write $\mathfrak{U}=\varepsilon_{i,j}\otimes\mathfrak{u}$.
So, we write $\mathfrak{U}_{j}=\varepsilon_{r(j),s(j)}\otimes\mathfrak{u_{j}}$
and, using the definition of $\varphi_{(n)}$, we have $\varphi_{(n)}(\mathfrak{U}_{1},\ldots,\mathfrak{U}_{k})=\varepsilon_{r(1),s(1)}\cdots\varepsilon_{r(k),s(k)}\otimes\Delta^{k}f_{\sigma}(0)(\mathfrak{u}_{1},\ldots,\mathfrak{u}_{k})$.
Thus 
\[
\varphi_{(n)}(\mathfrak{U}_{1},\ldots,\mathfrak{U}_{k})=\varepsilon_{r(1),s(k)}\otimes\Delta^{k}f_{\sigma}(0)(\mathfrak{u}_{1},\ldots,\mathfrak{u}_{k})
\]
provided $s(j)=r(j+1)$ for all $1\leq j\leq k-1$, and it equals
$0$ otherwise. In order to compute the right hand side of (\ref{phin}),
we form the matrix 
\[
B:=\left(\begin{array}{ccccc}
0 & \mathfrak{U}_{1} & 0 & \cdots & 0\\
0 & 0 & \ddots & \ddots & \vdots\\
\vdots & \ddots & \ddots & \ddots & 0\\
\vdots &  & \ddots & 0 & \mathfrak{U}_{k}\\
0 & \cdots & \cdots & 0 & 0
\end{array}\right),
\]
write $A$ for the $(k+1)\times(k+1)$ matrix $f_{(k+1)n\sigma}(B)$
and note that $\Delta^{k}f_{n\sigma}(0)(\mathfrak{U}_{1},\ldots,\mathfrak{U}_{k})$
is the $n\times n$ block of $A$ in the $(1,k+1)$ position. If we
now view $A$ as a matrix of size $n(k+1)\times n(k+1)$ (over $B(H_{\sigma})$),
we see that $A$ satisfies the assumptions of Lemma~\ref{ljmj},
with $p=n(k+1)$, $l(1)=r(1)$, $m(1)=n+s(1)$, $l(2)=n+r(2)$ etc.
It follows from that lemma that the only non zero entry in the upper-right
$n\times n$ block can be in the $(l(1),m(k))$ position and this
will be non zero only if $m(1)=l(2),m(2)=l(3),\ldots,m(k-1)=l(k)$.
Using our notation here, this entry will be non zero if and only if
$s(j)=r(j+1)$ for all $1\leq j\leq k-1$. If this is the case, then
by Lemma~\ref{ljmj}, this entry will be $\Delta^{k}f_{\sigma}(0)(\mathfrak{u}_{1},\ldots,\mathfrak{u}_{k})$.
This completes the proof.\end{proof}
\begin{prop}
\label{completely_bounded} Let $f=\{f_{\sigma}\}_{\sigma\in\Sigma}$
be a matricial family of functions defined on a matricial disc $\mathbb{D}(0,r)=\{\mathbb{D}(0,r,\sigma)\}_{\sigma\in\Sigma}$
where $\Sigma$ is a full, additive subcategory of $NRep(M)$. Suppose
$f$ is uniformly bounded in norm by $M$ so that, for every $\sigma\in\Sigma$,
and for every $\mathfrak{z}\in\mathbb{D}(0,r,\sigma)$, $||f_{\sigma}(\mathfrak{z})||\leq M$.
Then, for every $k$, the map 
\begin{equation}
\Delta^{k}f_{\sigma}(0)(\cdot,\cdots,\cdot):E^{\sigma*}\times E^{\sigma*}\cdots\times E^{\sigma*}\rightarrow B(H_{\sigma})\label{eq:Delta_k_as_a_map}
\end{equation}
 is a $k$-linear map, balanced over $\sigma(M)'$, and is completely
bounded, with 
\[
||\Delta^{k}f_{\sigma}(0)(\cdot,\cdots,\cdot)||_{cb}\leq\frac{M}{r^{k}}.
\]
\end{prop}
\begin{proof}
With the notation preceding Lemma~\ref{The_map_phi(n)}, we have
\[
||\Delta^{k}f_{\sigma}(0)(\cdot,\cdots,\cdot)||_{cb}=\sup\{||\Delta^{k}f_{\sigma}(0)_{(n)}||:n\geq1\}.
\]
 Note that in \cite{Effros2000} this norm is denoted $||\cdot||_{mb}$
but we follow the notation in \cite{Blecher2004d} and in other places
in the literature. Using Lemma~\ref{The_map_phi(n)}, it suffices
to show that $||\Delta^{k}f_{\sigma}(0)(\cdot,\cdots,\cdot)||\leq\frac{M}{r^{k}}$.
For this, consider $\mathfrak{u}_{i}\in\mathbb{D}(0,1,\sigma)$ and
write $\mathfrak{u}'_{i}:=r\mathfrak{u}_{i}\in\mathbb{D}(0,r,\sigma)$.
Then
\[
||\Delta^{k}f_{\sigma}(0)(\mathfrak{u}'_{1},\ldots,\mathfrak{u}'_{k})||\leq||f_{(k+1)\sigma}(\left(\begin{array}{ccccc}
0 & \mathfrak{u}'_{1} & 0 & \cdots & 0\\
0 & 0 & \ddots & \ddots & \vdots\\
\vdots & \ddots & \ddots & \ddots & 0\\
\vdots &  & \ddots & 0 & \mathfrak{u}'_{k}\\
0 & \cdots & \cdots & 0 & 0
\end{array}\right))||\leq M.
\]
By the $k$-linearity of the map, we get 
\[
||\Delta^{k}f_{\sigma}(0)(\mathfrak{u}_{1},\ldots,\mathfrak{u}_{k})||\leq\frac{M}{r^{k}},
\]
proving the complete norm estimate. The only thing left to prove is
the fact that the map is balanced, but this follows from Lemma~\ref{Delta_n}.
\end{proof}
Recall that, in the discussion preceding Lemma~\ref{lem:Derivative_of_Z_k},
$\mathcal{Z}_{k}$ is defined by the formula, $\mathcal{Z}_{k}(\mathfrak{z}):=\mathfrak{z}^{(k)}=\mathfrak{z}(I_{E}\otimes\mathfrak{z})\cdots(I_{E^{\otimes k-1}}\otimes\mathfrak{z})$,
for $\mathfrak{z}\in E^{\sigma*}$. For the statement of the following
theorem it will be convenient to use the natural extension of $\mathcal{Z}_{k}$
to $E^{\sigma*}\times E^{\sigma*}\times\cdots\times E^{\sigma*}$
and write $\mathcal{Z}_{k}(\mathfrak{u}_{1},\ldots,\mathfrak{u}_{k}):=\mathfrak{u}_{1}(I_{E}\otimes\mathfrak{u}_{2})\cdots(I_{E^{\otimes k-1}}\otimes\mathfrak{u}_{k})$,
for $\mathfrak{u}_{1},\ldots,\mathfrak{u}_{k}$ in $E^{\sigma*}$.
\begin{thm}
\label{theta_k} Let $f=\{f_{\sigma}\}_{\sigma\in\Sigma}$ be a matricial
family of functions defined on a matricial disc $\mathbb{D}(0,r)$
over a full, additive subcategory $\Sigma$ of faithful representations
in $NRep(M)$. Suppose that $f$ is defined and bounded uniformly
$M$ on $\mathbb{D}(0,r)$. Then, for every $k$, there is a unique
$\theta_{k}\in E^{\otimes k}$, with $||\theta_{k}||\leq\frac{M}{r^{k}}$,
such that 
\begin{equation}
\Delta^{k}f_{\sigma}(0)(\mathfrak{u}_{1},\ldots,\mathfrak{u}_{k})h=\mathcal{Z}_{k}(\mathfrak{u}_{1},\ldots,\mathfrak{u}_{k})(\theta_{k}\otimes h)
\end{equation}
for every $\mathfrak{u}_{1},\ldots,\mathfrak{u}_{k}$ in $E^{\sigma*}$
and every $h\in H_{\sigma}$. \end{thm}
\begin{proof}
Consider the map $\psi:E^{\sigma}\times E^{\sigma}\times\cdots\times E^{\sigma}\rightarrow B(H)$
defined by 
\[
\psi(\eta_{1},\ldots,\eta_{k})=\Delta^{k}f_{\sigma}(0)(\eta_{k}^{*},\ldots,\eta_{1}^{*})^{*}.
\]
For simplicity, we write here $N$ for the von Neumann algebra $\sigma(M)'$.
Using Proposition~\ref{completely_bounded}, we see that this is
a $k$-linear map, balanced over $\sigma(M)'$, with norm not exceeding
$\frac{M}{r^{k}}$. Applying \cite[Theorem 2.3]{Blecher2000}, we
find that it induces a linear, completely bounded map $\Psi:E^{\sigma}\otimes_{hN}\ E^{\sigma}\otimes_{hN}\cdots\otimes_{hN}E^{\sigma}\rightarrow B(H_{\sigma})$,
where $\otimes_{hN}$ is the module Haagerup tensor product, and $||\Psi||\leq\frac{M}{r^{k}}$.
But $E^{\sigma}\otimes_{hN}\ E^{\sigma}\otimes_{hN}\cdots\otimes_{hN}E^{\sigma}=E^{\sigma}\otimes_{C^{*}}\ E^{\sigma}\otimes_{C^{*}}\cdots\otimes_{C^{*}}E^{\sigma}$
where $\otimes_{C^{*}}$ is the internal tensor product of $C^{*}$-correspondences
(see \cite[Theorem 4.3]{Blecher1997d} ). For $b,c\in N$ and $\eta_{1},\ldots,\eta_{k}\in E^{\sigma}$,
we have $\Psi(b\cdot\eta_{1}\otimes\eta_{2}\otimes\cdots\otimes\eta_{k}\cdot c)=\Psi((I_{E^{\sigma}}\otimes b)\eta_{1}\otimes\eta_{2}\otimes\cdots\otimes\eta_{k}c)=\Delta^{k}f_{\sigma}(0)(c^{*}\eta_{k}^{*},\ldots,\eta_{1}^{*}(I\otimes b^{*}))^{*}=(c^{*}\Delta^{k}f(\underline{0})(\eta_{k}^{*},\ldots,\eta_{1}^{*})(I\otimes b^{*}))^{*}=b\Psi(\eta_{1}\otimes\eta_{2}\otimes\cdots\otimes\eta_{k})c$
(using Lemma~\ref{Delta_n}). Thus, $\Psi$ is a bimodule map. Write
$F$ for the $C^{*}$-correspondence $E^{\sigma}\otimes_{C^{*}}\ E^{\sigma}\otimes_{C^{*}}\cdots\otimes_{C^{*}}E^{\sigma}$
(over $N=\sigma(M)'$). Using the terminology of \cite{Muhly1998a},
we say that $(\Psi,\iota)$ is a completely bounded covariant representation
of $F$ where $\iota$ is the identity representation of $\sigma(M)'$
on $H_{\sigma}$. It follows from \cite[Lemma 3.5]{Muhly1998a} that
there is a bounded map $\tilde{\Psi}:F\otimes_{\iota}H_{\sigma}\rightarrow H_{\sigma}$
such that $\tilde{\Psi}(\varphi_{F}(\cdot)\otimes I_{H})=\iota(\cdot)\tilde{\Psi}$
and $||\tilde{\Psi}||\leq\frac{M}{r^{k}}$. Now note that $(E^{\sigma})^{\otimes k}$
is the self-dual completion of $F$ and, using Remark 1.8 in \cite{Viselter2011},
we have $F\otimes_{\iota}H_{\sigma}=(E^{\sigma})^{\otimes k}\otimes_{\iota}H_{\sigma}$.
Thus we can view $\tilde{\Psi}$ as a map from $(E^{\sigma})^{\otimes k}\otimes_{\iota}H_{\sigma}$
into $H_{\sigma}$ satisfying $\tilde{\Psi}(\varphi(\cdot)\otimes I_{H})=\iota(\cdot)\tilde{\Psi}$.
Applying \cite[Theorem 3.6 and Lemma 3.7]{Muhly2004a}, there is an
element $\theta_{k}\in E^{\otimes k}$ that corresponds to $\tilde{\Psi}$
via the isomorphism $((E^{\sigma})^{\otimes k})^{\iota}\cong E^{\otimes k}$.
More precisely, we have, using Equation (3.1) in \cite{Muhly2004a},
for every $\eta_{1},\eta_{2},\ldots,\eta_{k}$ in $E^{\sigma}$ and
every $h\in H_{\sigma}$, 
\begin{equation}
L_{\theta_{k}}^{*}(I_{E^{\otimes(k-1)}}\otimes\eta_{k})\cdots(I_{E}\otimes\eta_{2})\eta_{1}h=\tilde{\Psi}(\eta_{k}\otimes\eta_{k-1}\otimes\cdots\otimes\eta_{1}\otimes h)=
\end{equation}
 
\[
=\Delta^{k}f_{\sigma}(\tilde{0})(\eta_{1}^{*},\ldots,\eta_{k}^{*})^{*}h.
\]
 Taking adjoints and writing $\mathfrak{u}_{i}$ for $\eta_{i}^{*}$,
we obtain the desired result.
\end{proof}
The following corollary is now immediate, by Theorem~\ref{TTseries}
and Theorem~\ref{theta_k}.
\begin{cor}
\label{power_series} Let $f=\{f_{\sigma}\}_{\sigma\in\Sigma}$ be
a uniformly bounded matricial family of functions defined on a matricial
disc $\mathbb{D}(0,r)$, where $\Sigma$ is a full additive subcategory
of faithful representations in $NRep(M)$. Then there is a uniquely
determined series $\theta\sim\sum_{k\geq0}\theta_{k}$ in $\mathcal{T}_{+}((E))$
with $R(\theta)\geq r$ such that $f$ is the family of tensorial
power series $\{\sum_{k\geq0}\mathcal{Z}_{k}(\mathfrak{z})L_{\theta_{k}}\mid\mathfrak{z}\in\mathbb{D}(0,r,\sigma)\}_{\sigma\in\Sigma}$
.\end{cor}
\begin{rem}
\label{rem:Completion_of_the_proof_of_Thm_5.1} Corollary \ref{power_series}
and Lemma \ref{lem:Translation_invariance} immediately yield the
second assertion of Theorem \ref{thm:Matricial_implies_Frechet_hol.}
\end{rem}

\section{Series for matricial families of maps\label{sec:Series-for-matricial}}

\selectlanguage{english}%
In the previous section we studied matricial families of functions
in contexts where special generators are not present in the category
under consideration. In this section we focus on matricial families
of \emph{maps} (as in Theorem~\ref{UEtoUF}). Many of the results
proved for functions extend to the setting of such families with only
minor changes necessary. \foreignlanguage{american}{Indeed, the formulas
that go into defining $\Delta^{k}f_{\sigma}$ when $f$ is a map are
minor variants of the formulas that enter into the definitions of
$\Delta^{k}f_{\sigma}$ when $f$ is a function. One has only to replace
equations like $Cf_{\sigma}(\mathfrak{z})=f_{\tau}(\mathfrak{w})C$
with $Cf_{\sigma}(\mathfrak{z})=f_{\tau}(\mathfrak{w})(I_{F}\otimes C)$.
To illustrate, recall that to say $f$ is a matricial family of maps
means that for each pair, $\sigma$ and $\tau$ in $\Sigma$, for
each $\mathfrak{z}\in\mbox{\ensuremath{\mathbb{D}}(0,r,\ensuremath{\sigma}) }$,
for each $\mathfrak{w}\in\mathbb{D}(0,r,\tau)$, and for each $C\in\mathcal{I}(\sigma,\tau)$
such that $C\mathfrak{z}=\mathfrak{w}(I_{E}\otimes C)$ we have 
\[
Cf_{\sigma}(\mathfrak{z})=f_{\tau}(\mathfrak{w})(I_{F}\otimes C).
\]
So, if $\begin{bmatrix}\mathfrak{z} & \mathfrak{u}\\
0 & \mathfrak{w}
\end{bmatrix}$ lies in $\mathbb{D}(0,r,\sigma\oplus\tau)$ and if $f_{\sigma\oplus\tau}(\begin{bmatrix}\mathfrak{z} & \mathfrak{u}\\
0 & \mathfrak{w}
\end{bmatrix})=\begin{bmatrix}a_{11} & a_{21}\\
a_{21} & a_{22}
\end{bmatrix}$, then because 
\[
\left(\begin{array}{c}
I_{\sigma}\\
0
\end{array}\right)\mathfrak{z}=\left(\begin{array}{cc}
\mathfrak{z} & \mathfrak{u}\\
0 & \mathfrak{w}
\end{array}\right)\left(\begin{array}{c}
I_{E}\otimes I_{\sigma}\\
0
\end{array}\right),
\]
we must have 
\begin{equation}
\begin{pmatrix}I_{\mathfrak{z}}\\
0
\end{pmatrix}f_{\sigma}(\mathfrak{z})=\begin{pmatrix}a_{11} & a_{12}\\
a_{21} & a_{22}
\end{pmatrix}\left(\begin{array}{c}
I_{F}\otimes I_{\sigma}\\
0
\end{array}\right),\label{eq:Intw_bis_2}
\end{equation}
which implies that $a_{11}=f_{\sigma}(\mathfrak{z})$ and $a_{21}=0$.
Formula \eqref{eq:Intw_bis_2} is essentially formula \eqref{eq:Intw_1_bis},
and the other formulas in the analysis of matricial functions have
similar modifications for matricial maps. In particular, one can proceed
to define $\Delta f_{\sigma,\tau}(\mathfrak{z},\mathfrak{w})(\mathfrak{u})$
as $a_{12}$, and prove that $\Delta f_{\sigma,\tau}(\mathfrak{z},\mathfrak{w})(\mathfrak{u})$
is linear in $\mathfrak{u}$. With matricial maps, however, $\Delta f_{\sigma,\tau}(\mathfrak{z},\mathfrak{w})(\cdot)$
is a map from $E^{\sigma*}$ to $F^{\sigma*}$. }

\selectlanguage{american}%
Once the distinction between matricial maps and functions is recognized,
the entire body of results that begins with part (1) of Definition
\ref{Definition_Delta}, ends with Remark \ref{Partial Summary},
and does not inolve the \emph{bimodule properties} of $E^{\sigma*}$
(as a bimodule over $\sigma(M)'$), goes through \emph{mutatis mutandis}
for matricial maps.\foreignlanguage{english}{ In particular, the series
expansions of functions found in Corollary \ref{nlinear} make sense
and remain valid for maps.} \foreignlanguage{english}{However, the
notion of a ''tensorial power series of maps'' has not been defined
and the series expansion found in assertion (2) of Theorem \ref{thm:Matricial_implies_Frechet_hol.}
does not make sense in the setting of maps. So our principal goal,
Theorem \ref{expansion_for_maps}, is to exhibit the appropriate replacement
and to give a definition of tensorial power series of maps. }

We shall assume our category $\Sigma\subseteq NRep(M)$ is additive,
full, and has the property that every representation $\sigma\in\Sigma$
is faithful. Although matricial maps can be defined on arbitrary matricial
sets, for simplicity we shall restrict ourselves to matricial discs
centered at the origin. Thus we will consider matricial maps $f=\{f_{\sigma}\}_{\sigma\in\Sigma}$
defined on an $E,\Sigma$-matricial disc $\mathbb{D}(0,r)$, for some
$r$. We shall assume that $f$ maps into an $F,\Sigma$-matricial
disc. This is tantamount to assuming that $f$ is uniformly bounded
in $\sigma$. 

\selectlanguage{english}%
We will write $_{M}\mathcal{L}_{M}(E,F)$ for the maps in $\mathcal{L}(E,F)$
that are bimodule maps. That is, $T\in\mathcal{L}(E,F)$ lies in $_{M}\mathcal{L}_{M}(E,F)$
if and only if $T(\varphi_{E}(a)\xi b)=\varphi_{F}(a)T(\xi)b$, for
all $a,b\in M$. Our goal is to prove the following theorem that complements
part (2) of Theorem \ref{thm:Matricial_implies_Frechet_hol.}.
\begin{thm}
\label{expansion_for_maps} Let $E$ and $F$ be two $W^{*}$-correspondences
over the same $W^{*}$-algebra, $M$, and suppose $\Sigma$ is a full
additive subcategory of $NRep(M)$ whose objects are all faithful
representations of $M$. If $f=\{f_{\sigma}\}_{\sigma\in\Sigma}$
is a matricial family of maps, mapping an $E,\Sigma$-disc $\mathbb{D}(0,r,E)$
to an $F,\Sigma$-disc $\mathbb{D}(0,R,F)$, then there is a uniquely
defined sequence of bimodule maps $\{\mathfrak{D}^{k}f\}_{k=0}^{\infty}$,
where for each $k$, $\mathfrak{D}^{k}f$ lies in $\,_{M}\mathcal{L}_{M}(F,E^{\otimes k})$,
such that for every $\mathfrak{z}\in\mathbb{D}(0,r,\sigma)$, 
\begin{equation}
f_{\sigma}(\mathfrak{z})=f_{\sigma}(0)+\sum_{k\geq1}\mathcal{Z}_{k}(\mathfrak{z})(\mathfrak{D}^{k}f\otimes I_{H_{\sigma}}).\label{eq:Tensorial_power_series_of_maps}
\end{equation}

\end{thm}
Theorem \ref{expansion_for_maps} immediately suggests the following
definition for a tensorial power series of maps between two correspondence
duals.
\begin{defn}
\label{def:Tensorial_power_series_of_maps}A map $f$ defined from
an open disc $\mathbb{D}(\zeta_{0},r,\sigma)$ in $E^{\sigma*}$ to
$F^{\sigma*}$ is said to have a \emph{tensorial power series expansion}
on $\mathbb{D}(\zeta_{0},r,\sigma)$ in case there is a sequence $\{\Theta_{k}\}_{k\geq0}$,
with $\Theta_{k}$ in $\,_{M}\mathcal{L}_{M}(F,E^{\otimes k})$, such
that 
\[
f(\mathfrak{z})=\sum_{k\geq0}\mathcal{Z}_{k}(\mathfrak{z}-\zeta_{0})(\Theta_{k}\otimes I_{H_{\sigma}})
\]
for all $\mathfrak{z}\in\mathbb{D}(\zeta_{0},r,\sigma)$.
\end{defn}
Under our standing hypotheses on $\Sigma$, the derivatives $\Delta^{k}f_{\sigma}(0$)
are really $k$-linear bimodule maps on $(E^{\sigma*})^{k}$ mapping
to $F^{\sigma*}$, when we view $E^{\sigma*}$ as a bimodule over
$\sigma(M)'$, and they are balanced over $\sigma(M)'$. The principal
obstacle to proving Theorem \ref{expansion_for_maps} turns out to
be isolating the dependence of $\Delta^{k}f_{\sigma}(0)$ on $\sigma$.
This, in turn requires a careful study of the dualities involved.
Our analysis therefore rests on the following two lemmas. Note that
the first is couched in terms of maps from $E^{\sigma}$to $F^{\sigma}$
instead of maps from $E^{\sigma*}$ to $F^{\sigma*}$. The reason
is that as we noted on page~\pageref{E-sigma_correspondence}, \emph{$E^{\sigma}$
}and\emph{ $F^{\sigma}$ }are\emph{ $right$} correspondences over
$\sigma(M)'$ and this allows us to apply our Duality Theorem, \cite[Theorem 3.6]{Muhly2004a},
to identify $E$ and $F$ with the second duals, $E^{\sigma,\iota}$
and $F^{\sigma,\iota}$, respectively, where $\iota$ is the identity
representation of $\sigma(M)'$ on $H_{\sigma}.$ This identification
will prove central to what follows.
\begin{lem}
\label{dual_maps} If $E$ and $F$ are two $W^{*}-correspondences$
over $M$ and if $\sigma\in NRep(M)$ is a faithful normal representation,
then:
\begin{enumerate}
\item Every $\tau\in_{\sigma(M)'}\mathcal{L}_{\sigma(M)'}(E^{\sigma},F^{\sigma})$
induces is a unique bounded map $\tau_{*}:F\rightarrow E$ such that
for every $\eta\in E^{\sigma}$ and every $\theta\in F$, 
\begin{equation}
\eta^{*}L_{\tau_{*}(\theta)}=\tau(\eta)^{*}L_{\theta},\label{tau_star}
\end{equation}
where $L_{\theta}:H_{\sigma}\to F\otimes_{\sigma}H_{\sigma}$ is given
by $L_{\theta}h=\theta\otimes h.$
\item The map $\tau_{*}$ lies in $_{M}\mathcal{L}_{M}(F,E)$ and the correspondence
$\tau\mapsto\tau_{*}$ is contravariant and surjective, i.e., $(\tau_{1}\tau_{2})_{*}=\tau_{2*}\tau_{1*}$
and $(_{\sigma(M)'}\mathcal{L}_{\sigma(M)'}(E^{\sigma},F^{\sigma}))_{*}=_{M}\mathcal{L}_{M}(F,E)$. 
\end{enumerate}
\end{lem}
\begin{proof}
Recall that $E^{\sigma}:=\mathcal{I}(\sigma,\sigma^{E}\circ\varphi)$
is a (right) $W^{*}$-correspondence over $\sigma(M)'.$ Recall, too,
that if $\iota$ denotes the identity representation of $\sigma(M)'$
on $H_{\sigma}$, then the map $W_{E}:E\rightarrow E^{\sigma,\iota}$
such that $W_{E}(\xi)^{*}(\eta\otimes h)=L_{\xi}^{*}(\eta h)$, where
$\xi\in E,$ $\eta\in E^{\sigma}$ and $h\in H_{\sigma}$, and where
$L_{\xi}:H_{\sigma}\to E\otimes H_{\sigma}$ is given by $L_{\xi}h:=\xi\otimes h$,
is a correspondence isomorphism \cite[Theorem 3.6]{Muhly2004a}. Similarly,
one has a correspondence isomorphism $W_{F}:F\to F^{\sigma,\iota}.$
Also, note that, given $g\in F^{\sigma,\iota}$ (so that $g:H_{\sigma}\rightarrow F^{\sigma}\otimes_{\iota}H_{\sigma}$
and $g(bh)=(\varphi_{\iota}(b)\otimes I_{H_{\sigma}})gh$, where $b\in\sigma(M)$
and $\varphi_{\iota}(\cdot)$ is the left action of $\sigma(M)$ on
$F^{\sigma,\iota}$) and $\tau\in_{\sigma(M)'}\mathcal{L}_{\sigma(M)'}(E^{\sigma},F^{\sigma})$,
we have that $(\tau^{*}\otimes I_{H_{\sigma}})g\in E^{\sigma,\iota}$.
Now define 
\[
\tau_{*}(\theta)=W_{E}^{-1}((\tau^{*}\otimes I_{H_{\sigma}})W_{F}(\theta)).
\]
 It follows easily that $\tau_{*}\in_{M}\mathcal{L}_{M}(F,E)$. Also,
$(\tau^{*}\otimes I)W_{F}(\theta)=W_{E}(\tau_{*}(\theta))$ and, for
every $\eta\in E^{\sigma}$ and $h\in H_{\sigma}$, $W_{E}(\tau_{*}(\theta))^{*}(\eta\otimes h)=W_{F}(\theta)(\tau^{*}(\eta)\otimes h)$.
Using the definitions of $W_{E}$ and $W_{F}$, we find that $L_{\tau_{*}(\theta)}^{*}(\eta h)=L_{\theta}^{*}(\tau(\eta)h)$.
By taking adjoints, we obtain Equation (\ref{tau_star}). For the
uniqueness statement in (1), suppose that $\xi\in E$ satisfies $\eta^{*}L_{\xi}=\tau(\eta)^{*}L_{\theta}$
for all $\eta\in E^{\sigma}$. Then, for every $\eta\in E^{\sigma}$
and $h\in H_{\sigma}$, $L_{\xi}^{*}(\eta h)=L_{\tau_{*}(\theta)}^{*}(\eta h)$
and, since the images of all $\eta\in E^{\sigma}$ span $E\otimes H_{\sigma}$,
we find that $L_{\tau_{*}(\theta)}^{*}=L_{\xi}^{*}$ on $E\otimes H_{\sigma}$,
which in turn implies that $\tau_{*}(\theta)=\xi$. The fact that
the map $\tau\mapsto\tau_{*}$ is contravariant follows easily from
the definition. The fact that it is surjective follows from duality.\end{proof}
\begin{lem}
\label{indep_sigma} Suppose $E$ and $F$ are two $W^{*}$-correspondences
over the same $W^{*}$-algebra $M$. For $i=1,2,$ let $\sigma_{1}$
and $\sigma_{2}$ be two faithful representations in $NRep(M)$ and
let $\tau_{i}$ be a map in $_{\sigma_{i}(M)'}\mathcal{L}_{\sigma_{i}(M)'}(E^{\sigma_{i}},F^{\sigma_{i}})$.
If for every $c,d\in\mathcal{I}(\sigma_{1},\sigma_{2})$ we have $\tau_{1}((I\otimes c^{*})\eta d)=(I\otimes c^{*})\tau_{2}(\eta)d$
for every $\eta\in E^{\sigma_{2}}$, then $\tau_{1}=\tau_{2}$. Consequently,
under our standing hypotheses that $\Sigma$ is additive, full, and
composed of faithful representations, if $\{\tau_{\sigma}\}_{\sigma\in\Sigma}$
is a family of maps, with $\tau_{\sigma}\in_{\sigma(M)'}\mathcal{L}_{\sigma(M)'}(E^{\sigma},F^{\sigma})$,
and if for every $\sigma_{1},\sigma_{2}\in\Sigma$ and every $c,d\in\mathcal{I}(\sigma_{1},\sigma_{2})$
we have $\tau_{\sigma_{1}}((I\otimes c^{*})\eta d)=(I\otimes c^{*})\tau_{\sigma_{2}}(\eta)d$
for every $\eta\in E^{\sigma_{2}}$, then the maps $\tau_{\sigma*}$,
$\sigma\in\Sigma,$ obtained from Lemma~\ref{dual_maps} are independent
of $\sigma$.\end{lem}
\begin{proof}
First note that $(I_{E}\otimes c^{*})\eta d$ lies in $E^{\sigma_{1}}$
for every $c,d\in\mathcal{I}(\sigma_{1},\sigma_{2})$ and every $\eta\in E^{\sigma_{2}}$.
If we write $\eta_{c,d}$ for $(I_{E}\otimes c^{*})\eta d$, then
by the assumption, $\tau_{1}(\eta_{c,d})=(I\otimes c^{*})\tau_{2}(\eta)d$.
So, for every $\theta\in F$, we have $\eta_{c,d}^{*}L_{\tau_{1*}(\theta)}=\tau_{1}(\eta_{c,d})^{*}L_{\theta}=((I\otimes c^{*})\tau_{2}(\eta)d)^{*}L_{\theta}=d^{*}\tau_{2}(\eta)^{*}L_{\theta}(I\otimes c)=d^{*}\eta^{*}L_{\tau_{2*}(\theta)}(I\otimes c)=d^{*}\eta^{*}(I\otimes c)L_{\tau_{2*}(\theta)}=\eta_{c,d}^{*}L_{\tau_{2*}(\theta)}$.
Thus 
\begin{equation}
L_{\tau_{1*}(\theta)}^{*}\eta_{c,d}=L_{\tau_{2*}(\theta)}^{*}\eta_{c,d}
\end{equation}
 for every $\eta\in E^{\sigma_{2}}$ and $c,d\in\mathcal{I}(\sigma_{1},\sigma_{2})$.
Since both $\sigma_{1}$ and $\sigma_{2}$ are normal faithful representations
of $M$, $\bigvee\{d(H_{\sigma_{1}}):d\in\mathcal{I}(\sigma_{1},\sigma_{2})\}=H_{\sigma_{2}}$
and $\bigvee\{c^{*}(H_{\sigma_{2}}):c\in\mathcal{I}(\sigma_{1},\sigma_{2})\}=H_{\sigma_{1}}$.
It follows from \cite[Lemma 3.5]{Muhly2004a} that $\bigvee\{\eta(H_{\sigma_{2}}):\eta\in E^{\sigma_{2}}\}=E\otimes_{\sigma_{2}}H_{\sigma_{2}}$
and, therefore, that 
\begin{equation}
\bigvee\{\eta_{c,d}(H_{\sigma_{1}}):\eta\in E^{\sigma_{2}},\; c,d\in\mathcal{I}(\sigma_{1},\sigma_{2})\}=E\otimes_{\sigma_{1}}H_{\sigma_{1}}.
\end{equation}
 Combining these equations above we conclude that $\tau_{1*}=\tau_{2*}$.\end{proof}
\begin{cor}
\label{Deltaf}Under the hypotheses of Theorem \ref{expansion_for_maps},
the matricial map $f$ has a well-defined Taylor derivative $\Delta f_{\sigma}(0)(\cdot)$
that is a map from $E^{\sigma*}$ to $F^{\sigma*}$ in $_{\sigma(M)'}\mathcal{L}_{\sigma(M)'}(E^{\sigma*},F^{\sigma*})$.
Its transposed map $\eta\mapsto(\Delta f_{\sigma}(0)(\eta^{*}))^{*}$
lies in $_{\sigma(M)'}\mathcal{L}_{\sigma(M)'}(E^{\sigma},F^{\sigma})$
and has the following dependence on $\sigma$: There is a map $\mathfrak{D}f\in_{M}\mathcal{L}_{M}(F,E)$
such that for every $\sigma\in\Sigma$ and every $\mathfrak{z}\in E^{\sigma*}$,
\[
\Delta f_{\sigma}(0)(\mathfrak{z})=\mathfrak{z}\circ(\mathfrak{D}f\otimes I_{H_{\sigma}}).
\]
\end{cor}
\begin{proof}
As we indicated at the beginning of this section, the existence and
linearity of $\Delta f_{\sigma}(0)(\cdot)$ is proved using the arguments
of Lemma~\ref{Definition_Delta}(1) and Lemma~\ref{additivity}.
The fact that $\Delta f_{\sigma}(0)(\cdot)$ is a bimodule map uses
the arguments of Lemma~\ref{Definition_Delta}(2) and (3). (Note
that we are entitled to apply these because we are assuming $\Sigma$
is full.) Thus the transposed map $\eta\mapsto(\Delta f_{\sigma}(0)(\eta)^{*})^{*}$
lies in $_{\sigma(M)'}\mathcal{L}_{\sigma(M)'}(E^{\sigma},F^{\sigma})$,
also, and it follows from Lemma~\ref{dual_maps} and Lemma~\ref{indep_sigma}
that there is an element $\mathfrak{D}f$ of $\,_{M}\mathcal{L}_{M}(F,E)$
such that, for every $\sigma\in\Sigma$, every $\eta\in E^{\sigma}$
and every $\theta\in F$, $\eta^{*}L_{\mathfrak{D}f(\theta)}=\Delta f_{\sigma}(0)(\eta^{*})L_{\theta}$.
Writing $\mathfrak{z}$ in place of $\eta^{*}$ and applying the two
sides of this equality to $h\in H_{\sigma}$, we obtain the equation
\[
\Delta f_{\sigma}(0)(\mathfrak{z})(\theta\otimes h)=\mathfrak{z}(\mathfrak{D}f(\theta)\otimes h)=\mathfrak{z}(\mathfrak{D}f\otimes I_{H})(\theta\otimes h)
\]
 and the result follows.
\end{proof}
A similar analysis allows us to identify the dependence of $\Delta^{k}f_{\sigma}(0)$
on $\sigma$, $k>1$.
\begin{lem}
\label{Deltak} Under the hypotheses of Theorem \ref{expansion_for_maps},
we conclude that for all $k\in\mathbb{N}$ and $\sigma\in\Sigma$,
the Taylor derivative $\Delta^{k}f_{\sigma}(0)(\cdot,\ldots,\cdot)$
is well defined and is a completely bounded $k$-linear, bimodule
map from $E^{\sigma*}\times\cdots\times E^{\sigma*}$ to $F^{\sigma*}$
that is balanced over $\sigma(M)'$. Moreover, there is a uniquely
determined map $\mathfrak{D}^{k}f$ in $\,_{M}\mathcal{L}_{M}(F,E^{\otimes k})$
such that for every $\sigma\in\Sigma$ and every $\mathfrak{z}\in E^{\sigma*}$,
\[
\Delta^{k}f_{\sigma}(0)(\mathfrak{z})=\mathcal{Z}_{k}(\mathfrak{z})\circ(\mathfrak{D}^{k}f\otimes I_{H_{\sigma}}).
\]
\end{lem}
\begin{proof}
The proof of the existance of $\Delta^{k}f_{\sigma}(0)(\cdot,\ldots,\cdot)$
uses the same arguments as the analogous result for matricial families
of functions, as we already have mentioned. To get the bimodule properties
of $\Delta^{k}f_{\sigma}(0)(\cdot,\ldots,\cdot),$ note that as in
the proof of Theorem~\ref{theta_k}, $\Delta^{k}f_{\sigma}(0)(\cdot,\ldots,\cdot)$
induces a bimodule map $\Psi:E^{\sigma}\otimes_{C^{*}}E^{\sigma}\otimes_{C^{*}}\cdots\otimes_{C^{*}}E^{\sigma}\rightarrow F^{\sigma}$.
Using $\Psi$ in place of $\tau$ in Lemma~\ref{dual_maps}, we obtain
a map $\Psi_{*}\in_{M}\mathcal{L}_{M}(F,E^{\otimes k})$. Note that
we do not know that $\Psi$ induces a bimodule map on $(E^{\sigma})^{\otimes k}$,
but all we needed in the proof of Lemma~\ref{dual_maps} is to know
that the map $\Psi\otimes I_{H_{\sigma}}:(E^{\sigma})^{\otimes k}\otimes_{\iota}H_{\sigma}\rightarrow F\otimes_{\iota}H_{\sigma}$
is well defined and intertwines the actions of $\sigma(M)'$. This
holds here since $(E^{\sigma})^{\otimes k}\otimes_{\iota}H=E^{\sigma}\otimes_{C^{*}}E^{\sigma}\otimes_{C^{*}}\cdots\otimes_{C^{*}}E^{\sigma}\otimes_{\iota}H$,
thanks to an observation of Viselter \cite[Remark 1.8]{Viselter2011}.
So we do indeed obtain a map $\Psi_{*}\in_{M}\mathcal{L}_{M}(F,E^{\otimes k})$
such that, for every $\theta\in F$ and every $\mathfrak{z}\in E^{\sigma*}$,
\[
\mathcal{Z}_{k}(\mathfrak{z})L_{\Psi_{*}(\theta)}=\Delta^{k}f_{\sigma}(0)(\mathfrak{z})L_{\theta}.
\]
 The map $\mathfrak{D}^{k}f$ that we want is $\Psi_{*}$. 
\end{proof}
The proof of Theorem \ref{expansion_for_maps} is essentially complete.
All that is necessary is to observe that an analogue of the expansion
(\ref{TT}) holds also for matricial families of maps. Since we have
identified the Taylor derivatives with the $\mathfrak{D}^{k}f$, the
proof is complete.

We conclude by showing how the Schur class automorphisms we considered
in \cite{Muhly2008b} fit into theory we have developed here. 
\begin{example}
For a central element $\gamma\in\mathbb{D}(Z(E^{\sigma}))$, the map
$g_{\gamma}:\mathbb{D}(E^{\sigma*})\rightarrow\mathbb{D}(E^{\sigma*})$
is defined by 
\[
g_{\gamma}(\mathfrak{z})=\Delta_{\gamma}(I-\mathfrak{z}\gamma)^{-1}(\gamma^{*}-\mathfrak{z})\Delta_{\gamma^{*}}^{-1}
\]
 where $\Delta_{\gamma}=(I_{H}-\gamma^{*}\gamma)^{1/2}\in Z(\sigma(M)')=\sigma(Z(M))$
and $\Delta_{\gamma^{*}}=(I_{E\otimes H}-\gamma\gamma^{*})^{1/2}\in B(E\otimes_{\sigma}H_{\sigma})$.
In \cite[Lemma 4.20]{Muhly2008b} it is shown that there is a completely
isometric automorphism $\alpha_{\gamma}$ of $H^{\infty}(E)$ such
that 
\[
\widehat{\alpha_{\gamma}(X)}(\mathfrak{z})=\widehat{X}(g_{\gamma}(\mathfrak{z}))
\]
 for $\mathfrak{z}\in\mbox{\ensuremath{\mathbb{D}}(0,1,\ensuremath{\sigma})}$.
Using Theorem~\ref{UEtoUF} we see that $g_{\gamma}$ is a matricial
family of maps (with $E=F$) and, thus, Corollary~\ref{expansion_for_maps}
applies and we can write 
\begin{equation}
g_{\gamma}(\mathfrak{z})=g_{\gamma}(0)+\sum_{k\geq1}\mathcal{Z}_{k}(\mathfrak{z})(\mathfrak{D}^{k}g_{\gamma}\otimes I_{H_{\sigma}}).\label{g1}
\end{equation}
 On the other hand, using the expansion $(I-\mathfrak{z}\gamma)^{-1}=\sum_{k\geq0}(\mathfrak{z}\gamma)^{k}$,
one can write $g_{\gamma}$ as 
\begin{equation}
g_{\gamma}(\mathfrak{z})=\Delta_{\gamma}\gamma^{*}\Delta_{\gamma^{*}}^{-1}+\sum_{k\geq1}\Delta_{\gamma}\mathfrak{z}\gamma\mathfrak{z}\gamma\cdots\gamma\mathfrak{z}\Delta_{\gamma^{*}}\label{g2}
\end{equation}
 where, in the $k$-th term, $\mathfrak{z}$ appears $k$ times. 

At first glance, it might not be evident that the terms in the expansion
(\ref{g2}) can be written in the form of the terms in (\ref{g1}).
To deal with the zeroth term, simply note that $\Delta_{\gamma}\gamma^{*}=\gamma^{*}\Delta_{\gamma*}$
so that $\Delta_{\gamma}\gamma^{*}\Delta_{\gamma^{*}}^{-1}=\gamma^{*}=g_{\gamma}(0)$.
For the first term, note that $\Delta_{\gamma*}^{2}=I-\gamma\gamma^{*}\in(\varphi_{E}(M)\otimes I)'$
(since $\gamma\in E^{\sigma}$) but also, for $a\in\sigma(M)'$, $\gamma^{*}(I\otimes a)=a\gamma$
(since $\gamma$ is in the center of $E^{\sigma}$) and, thus, $\Delta_{\gamma*}^{2}\in(I_{E}\otimes\sigma(M)')'=\mathcal{L}(E)\otimes I_{H}$.
It follows that one can write $\Delta_{\gamma*}=X\otimes I_{H}$ for
an $X\in\mathcal{L}(E)\cap\varphi_{E}(M)'$. Note also that $\Delta_{\gamma}$
is in $Z(\sigma(M))=\sigma(Z(M))$ and (identifying $\Delta_{\gamma}$
with $\sigma^{-1}(\Delta_{\gamma})$), the first term in (\ref{g2})
can be written $\Delta_{\gamma}\mathfrak{z}\Delta_{\gamma*}=\mathfrak{z}(\varphi_{E}(\Delta_{\gamma})\otimes I_{H_{\sigma}})(X\otimes I_{H_{\sigma}})=\mathfrak{z}((\varphi_{E}(\Delta_{\gamma})X)\otimes I_{H_{\sigma}})$
so that, writing $\mathfrak{D}g_{\gamma}=\ \varphi_{E}(\Delta_{\gamma})X\in\mathcal{L}(E)\cap\varphi_{E}(M)'=_{M}\mathcal{L}_{M}(E)$,
we have the first term of (\ref{g2}) written in the form of (\ref{g1}).
(Note that $\varphi_{E}(\Delta_{\gamma})\in\varphi_{E}(M)'$ since
$\Delta_{\gamma}\in Z(M)$).

For $k\geq2$ a similar computation is possible. It is a little less
straightforward than the computation of $\mathfrak{D}g_{\gamma}$,
but the case when $k=2$ illustrates amply what to do. With $\gamma$
as above, define a map $Y:E\otimes H\rightarrow E^{\otimes2}\otimes H$
by $Y\zeta h=(I_{E}\otimes\zeta)\gamma h$ for every $\zeta\in E^{\sigma}$
and $h\in H_{\sigma}$. To see that this map is well defined and bounded,
compute $||\sum_{i}(I\otimes\zeta_{i})\gamma h_{i}||^{2}=\sum_{i,j}\langle(I\otimes\zeta_{i})\gamma h_{i},(I\otimes\zeta_{j})\gamma h_{j}\rangle=\sum_{i,j}\langle\gamma^{*}(I\otimes\zeta_{j}^{*}\zeta_{i})\gamma h_{i},h_{j}\rangle$.
Since $\gamma$ is in the center of $E^{\sigma}$, the last expression
equals $\sum_{i,j}\langle\gamma^{*}\gamma\zeta_{j}^{*}\zeta_{i}h_{i},h_{j}\rangle$.
As $\zeta_{j}\in E^{\sigma}$, this is equal to $\sum_{i,j}\langle\zeta_{j}^{*}(\varphi_{E}(\gamma^{*}\gamma)\otimes I_{H})\zeta_{i}h_{i},h_{j}\rangle=||(\varphi_{E}(\gamma^{*}\gamma)\otimes I_{H})^{1/2}\sum_{i}\zeta_{i}h_{i}||^{2}$.
Thus $||Y||\leq||\gamma||$. It is also easy to check that $Y$ intertwines
$I_{E}\otimes b$ and $I_{E^{\otimes2}}\otimes b$ for every $b\in\sigma(M)'$.
Thus, there is some $Y_{0}\in_{M}\mathcal{L}_{M}(E,E^{\otimes2})$
such that $Y=Y_{0}\otimes I_{H}$. For $\mathfrak{z}\in E^{\sigma*}$
and $\zeta\in E^{\sigma}$, we have $\mathfrak{z}\zeta\in\sigma(M)'$
and, since $\gamma$ is central, we conclude that for $h\in H_{\sigma}$,
$\gamma\mathfrak{z}\zeta h=(I\otimes\mathfrak{z}\zeta)\gamma h=(I\otimes\mathfrak{z})(Y_{0}\otimes I)\zeta h$.
Thus $\gamma\mathfrak{z}=(I_{E}\otimes\mathfrak{z})(Y_{0}\otimes I_{H_{\sigma}})$.
We now compute: 
\[
\Delta_{\gamma}\mathfrak{z}\gamma\mathfrak{z}\Delta_{\gamma*}=\mathfrak{z}\gamma\mathfrak{z}(\varphi_{E}(\Delta_{\gamma})\otimes I_{H})(X\otimes I_{H})=\mathcal{Z}_{2}(\mathfrak{z})(Y_{0}\varphi_{E}(\Delta_{\gamma})X\otimes I_{H}),
\]
which shows that $\mathfrak{D}^{2}g_{\gamma}=Y_{0}\varphi_{E}(\Delta_{\gamma})X$.
\end{example}
\selectlanguage{american}%
\bibliographystyle{amsalpha}
\bibliography{Master120711}

\end{document}